\documentclass[12pt]{amsart}

\usepackage{graphicx}
\usepackage{amssymb}
\usepackage{float}
\usepackage{fullpage}
\usepackage{physics}

\usepackage[dvipsnames]{xcolor}

\newtheorem{theorem}{Theorem}[section]
\newtheorem{lemma}[theorem]{Lemma}
\newtheorem{corollary}[theorem]{Corollary}
\newtheorem{proposition}[theorem]{Proposition}

\theoremstyle{definition}
\newtheorem{definition}[theorem]{Definition}

\theoremstyle{remark}
\newtheorem{remark}[theorem]{Remark}

\numberwithin{equation}{section}

\usepackage{physics}
\usepackage{multicol}

\usepackage{comment}

\usepackage{mathtools}

\newcommand{\C}{{\mathbb C}}
\newcommand{\R}{{\mathbb R}}

\newcommand{\Z}{{\mathbb Z}}
\newcommand{\N}{{\mathbb N}}

\newcommand{\D}{{\mathbb D}}

\renewcommand{\phi}{{\varphi}}
\renewcommand{\le}{{\,\leqslant}}
\renewcommand{\ge}{{\,\geqslant}}

\newcommand{\eps}{\epsilon}

\renewcommand{\Re}{\mathop{\rm Re}}

\renewcommand{\a}{\mathfrak a}
\renewcommand{\b}{\mathfrak b}
\renewcommand{\c}{\mathfrak c}

\newcommand{\1}{{\mathbf 1}}

\newcommand{\Mod}[1]{\ (\mathrm{mod}\ #1)}

\addtocontents{toc}{\protect\setcounter{tocdepth}{1}}

\begin{document}

\title[Primes in large arithmetic progressions]{Primes in arithmetic progressions to large moduli,
and Goldbach beyond the square-root barrier
}

\author{Jared Duker Lichtman,
\MakeLowercase{and appendix with} Sary Drappeau}

\address{I2M, Universit\'e d'Aix-Marseille, CNRS, Case 907, Campus de Luminy, 13288 Marseille Cedex 9, France}

\email{sary-aurelien.drappeau@univ-amu.fr}

\address{Mathematical Institute, University of Oxford, Oxford, OX2 6GG, UK}

\email{jared.d.lichtman@gmail.com}





\begin{abstract}
We show the primes have level of distribution $66/107\approx 0.617$ using triply well-factorable weights. This gives the highest level of distribution for primes in any setting, improving on the prior record level $3/5=0.60$ of Maynard. We also extend this level to $5/8=0.625$, assuming Selberg's eigenvalue conjecture.
As applications of the method, we obtain new upper bounds for twin primes and for Goldbach representations of even numbers $a$. For the Goldbach problem, this is the first use of a level of distribution beyond the `square-root barrier', and leads to the greatest improvement on the problem since Bombieri--Davenport from 1966.
Our proof optimizes the Deshouillers--Iwaniec spectral large sieve estimates, both in the exceptional spectrum and uniformity in the residue $a$, refining Drappeau--Pratt--Radziwi{\l}{\l} and Assing--Blomer--Li.
\end{abstract}

\maketitle

\tableofcontents

\section{Introduction}\label{sec:Introduction}

Denote by $\pi_2(x)$ the number of twin primes $p$, $p+2$ up to $x$.
The twin prime conjecture asserts that $\pi_2(x)$ diverges as $x\to\infty$. Also denote by ${\rm G}(a)$ the number of representations of an even integer $a$ as a sum of two primes $a=p_1+p_2$. The Goldbach conjecture asserts that $G(a)\ge1$ for every even $a\ge4$. These celebrated conjectures were made quantitatively precise by Hardy and Littlewood \cite{HL} in 1923, who proposed asymptotic formulae,
\begin{align}
\pi_2(x) \ \sim \ \Pi_2(x) \qquad\text{and}\qquad {\rm G}(a) \ \sim \ \Pi_a(a),
\end{align}
where
\begin{align*}
\Pi_a(x) :=\frac{2\mathfrak{S}_a\,x}{(\log x)^2} \qquad\text{and}\qquad
\mathfrak{S}_a := \prod_{2<p}\frac{1-2/p}{(1-1/p)^{2}}\,\prod_{2<p\mid a}\frac{1-1/p}{1-2/p}.
\end{align*}

We prove upper bounds for twin primes and Goldbach representations that are within factors 3.23 and 3.40, respectively, of the predicted Hardy--Littlewood asympototics.

\begin{theorem}\label{thm:twinbound}
For $x\in\R$ sufficiently large, we have
$\pi_2(x) \, \lesssim \, 3.2290\,\Pi_2(x)$.
\end{theorem}

\begin{theorem}\label{thm:Goldbachbound}
For any even integer $a\in\N$ sufficiently large, we have
${\rm G}(a) \,\lesssim \, 3.3907\, \Pi_a(a)$.
\end{theorem}

See the table below for a chronology of the known upper bounds on $\pi_2(x)$ and ${\rm G}(a)$.

Theorems \ref{thm:twinbound} and \ref{thm:Goldbachbound} are obtained using refined linear sieve estimates of the author in \cite{Ltwin}. To apply such sieve estimates, one must control the remainder terms using equidistribution estimates for primes in arithmetic progressions. Our main proof ingredient is an improved level of distribution for primes, given in Theorem \ref{theorem:Factorable} below. In particular, Theorem \ref{thm:Goldbachbound} uses for the first time a level of distribution beyond the `square-root barrier' for the Goldbach problem.

\vspace{1.5mm}

\begin{center}
\begin{tabular}{c|l|l|l}
Year & Author(s) & Twin primes & Goldbach\\
& & $\pi_2(x)/\Pi_2(x)\,\lesssim$ & ${\rm G}(a)/\Pi_a(a)\,\lesssim$\\
\hline    
1947 & Selberg \cite{Selb} & 8 & 8\\
1964 & Pan \cite{Pan}  & 6 & 6\\
1966 & Bombieri--Davenport \cite{BD}  & 4 & 4\\
1978 & Chen \cite{Chen} & 3.9171 & 3.9171 \\
1983 & Fouvry--Iwaniec \cite{FI}  & $3.7777\cdots=34/9$ & ---\\
1984 & Fouvry \cite{F} & $3.7647\cdots=64/17$ &  ---\\
1986 & Bombieri--Friedlander--Iwaniec \cite{BFI1} & 3.5 & --- \\
1986 & Fouvry--Grupp \cite{FG} & 3.454 & ---\\
1990 & Wu \cite{WuI} & 3.418 & ---\\
2003 & Cai--Lu \cite{CL} & 3.406 & ---\\
2004 & Wu \cite{WuII} & 3.3996 & 3.9104\\
2022 & Lichtman \cite{Ltwin} & 3.2996 &  ---
\end{tabular}
\end{center}

\vspace{1.mm}

In particular, our new bound for ${\rm G}(a)$ gives a $13.25\%$ refinement of the prior record of Wu \cite{WuII} from 2004 (Wu in turn had refined Chen \cite{Chen} by $0.17\%$). This gives the greatest refinement on the problem since Bombieri--Davenport \cite{BD} from 1966.

\subsection{Primes in arithmetic progressions to large moduli}

All bounds for $G(a)$ prior to Theorem \ref{thm:twinbound} used the classical Bombieri--Vinogradov theorem from 1965. As usual, we denote by $\pi(x)$ the number of primes up to $x$, and $\pi(x;q,a)$ the number of primes up to $x$ congruent to $a\Mod{q}$. The Bombieri--Vinogradov theorem states that for every $\eps,A>0$,
\begin{align}\label{eq:BV}
\sum_{q\le x^{\vartheta}}\sup_{(a,q)=1}\Big|\pi(x;q,a)-\frac{\pi(x)}{\phi(q)}\Big| \ \ll_{A} \ \frac{x}{(\log x)^A},
\end{align}
with exponent $\vartheta<\frac{1}{2}$, which is often called the `level of distribution' for primes. The estimate \eqref{eq:BV} may be viewed as an assertion of the Generalized Riemann Hypothesis on average over moduli up to $x^{\vartheta}$. It remains an important open problem to extend the range in \eqref{eq:BV} to $\vartheta = \frac{1}{2}+\delta$ for some fixed $\delta>0$. Indeed, Elliott and Halberstam \cite{EH} conjectured such an extension up to $\vartheta = 1-\eps$ for any $\eps>0$.

In contrast to Goldbach representations ${\rm G}(a)$, there has been comparatively more progress on upper bounding twin primes $\pi_2(x)$, by leveraging improved level of distribution results in special cases. For example, in case of a fixed nonzero residue $a\in\Z$ and `well-factorable' weights $\lambda_q\in\C$ (defined below), Bombieri--Friedlander--Iwaniec \cite{BFI1} proved in 1986 that
\begin{align}\label{eq:BFIwellfactor}
\sum_{\substack{q\le x^\vartheta\\(q,a)=1}}\lambda_q\Bigl(\pi(x;q,a)-\frac{\pi(x)}{\phi(q)}\Bigr)\ll_{a,A}\frac{x}{(\log{x})^A},
\end{align}
up to level $\vartheta < \frac{4}{7}$. They deduced their bound $\pi_2(x) \lesssim 3.5\, \Pi_2(x)$, by choosing the weights in \eqref{eq:BFIwellfactor} to be (essentially) the upper bound weights $\lambda_q^\pm$ for the linear sieve. Such weights satisfy the following `well-factorable' property:
\begin{definition}[Well-factorable]
Let $Q\in\mathbb{R}$. A sequence $\lambda_q\in\C$ is \textbf{well-factorable of level $Q$} if, for any choice of factorization $Q=Q_1Q_2$ with $Q_1,Q_2\ge 1$, there exist 1-bounded sequences $\gamma^{(1)}_{q_1},\gamma^{(2)}_{q_2}$, supported on $1\le q_i\le Q_i$ for $i\in\{1,2\}$ such that $\lambda = \gamma^{(1)}\ast \gamma^{(2)}$, i.e.
\[
\lambda_q=\sum_{q=q_1q_2}\gamma^{(1)}_{q_1}\gamma^{(2)}_{q_2}.
\]
\end{definition}
%
%
%
%
%
%
%
%
Heuristically, one may view well-factorability as a property of integers $q$ in the support $\{q:\lambda_q\neq0\}$: for any splitting $Q=Q_1Q_2$, we may factor $q=q_1q_2$ into integers $q_i\le Q_i$.

For a quarter century, the well-factorable result \eqref{eq:BFIwellfactor} constituted the largest level $\vartheta< \frac{4}{7}$ for primes in arithmetic progressions in any setting. In breakthrough work, Maynard \cite{JM2} extended \eqref{eq:BFIwellfactor} to level $\vartheta< \frac{3}{5}$, for $\lambda_q$ in the stricter class of `triply well-factorable' weights:

\begin{definition}[Triply well-factorable]\label{definition:WellFactorable}
Let $Q\in\mathbb{R}$. We say a sequence $\lambda_q$ is \textbf{triply well-factorable of level $Q$} if, for any choice of factorization $Q=Q_1Q_2Q_3$ with $Q_1,Q_2,Q_3\ge 1$, there exist 1-bounded sequences $\gamma^{(1)}_{q_1},\gamma^{(2)}_{q_2},\gamma^{(3)}_{q_3}$, supported on $1\le q_i \le Q_i$ for $i\in\{1,2,3\}$, such that $\lambda = \gamma^{(1)}\ast \gamma^{(2)}\ast \gamma^{(3)}$, i.e.
\[
\lambda_q=\sum_{q=q_1q_2q_3}\gamma^{(1)}_{q_1}\gamma^{(2)}_{q_2}\gamma^{(3)}_{q_3}.
\]
\end{definition}

Heuristically, one may view triple well-factorability as a property of integers $q$ the support $\{q:\lambda_q\neq0\}$: for any splitting $Q=Q_1Q_2Q_3$, we may factor $q=q_1q_2q_3$ into integers $q_i\le Q_i$.

As our main result, we show that primes have level of distribution $\vartheta < \tfrac{66}{107}\approx 0.617$ with triply well-factorable weights. This gives the largest level for primes in arithmetic progressions in any setting.

\begin{corollary}\label{cor:Factorable}
Take nonzero $a\in\mathbb{Z}$ and let $A,\epsilon>0$. Let $\lambda_q$ be triply well-factorable of level $x^{\vartheta}$ with $\vartheta <\tfrac{66}{107}$. Then we have
\begin{align}
\sum_{\substack{q\le x^{\vartheta}\\ (a,q)=1}}\lambda_q\Bigl(\pi(x;q,a)-\frac{\pi(x)}{\phi(q)}\Bigr)\ll_{a,A,\epsilon} \frac{x}{(\log{x})^A}.
\end{align}
\end{corollary}

We note that the prior level $\vartheta < \frac{3}{5}$ has been a natural barrier from works of Bombieri--Friedlander--Iwaniec, Fouvry--Tenenbaum, Drappeau, and Maynard \cite{BFI1, BFI2, FT, Drapp, JM2}.

\subsection{Exceptional eigenvalues}
Our key technical results are given in terms of the best known exponent ${\theta}$ towards Selberg's eigenvalue conjecture.\footnote{In this section and hereafter, we denote ${\pmb \vartheta}$ `$\backslash$vartheta' in bold as the level of distribution for primes, and ${\theta}$ `$\backslash$theta' as the exponent towards Selberg's eigenvalue conjecture. This is aimed to avoid confusion, while keeping with traditional notation.}

\begin{definition}
For $q \in\N$, denote the largest eigenvalue $\lambda_1=\lambda_1(q)$ of the Laplacian for the congruence subgroup $\Gamma_0(q)$. Define $\theta_q := \max(0, \sqrt{1 - 4\lambda_1})$ and ${\theta} := \sup_{q\in\N} \theta_q$. \footnote{We use the definition of ${\theta}$ as in \cite{DI,Drapp}. However, as a caution, we note some authors' differ by a factor of 2, e.g. \cite[Theorem A]{KimSarn} display a bound of $\frac{7}{64}$ for $|\Re\mu_{j,\infty}| \le {\theta}/2$.}
\end{definition}

Selberg's eigenvalue conjecture asserts that ${\theta}=0$, namely, $\lambda\ge \frac{1}{4}$ for all eigenvalues $\lambda$ of the (hyperbolic) Laplacian for $\Gamma_0(q)$. As such, if $\lambda<\frac{1}{4}$ then such an eigenvalue $\lambda$ is called \emph{exceptional}.\footnote{We elaborate on the role of exceptional eigenvalues in \S\ref{sec:largesieve}.}
Corollary \ref{cor:Factorable} is a consequence of our main technical result, using the current record bound ${\theta}\le \tfrac{7}{32}$ of Kim--Sarnak \cite[Appendix 2]{KimSarn}.

%
%
%
%
%
%
%
%
\begin{theorem}\label{theorem:Factorable}
Let $a\in\mathbb{Z}_{\neq0}$ and $A,\epsilon>0$. Let $\lambda_q$ be triply well-factorable of level $x^{{\pmb \vartheta}}$, with
\begin{align}
{\pmb \vartheta}=\frac{5-4{\theta}}{8-6{\theta}}-\epsilon.
\end{align}
Then in the range $|a|< x^{\eps}$ we have
\begin{align}\label{eq:main}
\sup_{0<|a|< x^{\eps}}\sum_{\substack{q\le x^{{\pmb \vartheta}}\\ (a,q)=1}}\lambda_q\Bigl(\pi(x;q,a)-\frac{\pi(x)}{\phi(q)}\Bigr)\ll_{A,\epsilon} \frac{x}{(\log{x})^A}.
\end{align}
In addition, for ${\pmb \vartheta_1} = (5-\theta)/8 -\epsilon$, in the larger range $|a|< x^{1+\eps}$ we have
\begin{align}\label{eq:mainuniform}
\sup_{0<|a|< x^{1+\eps}}\sum_{\substack{q\le\, x^{{\pmb \vartheta_1}}\\ (a,q)=1}}\lambda_q\Bigl(\pi(x;q,a)-\frac{\pi(x)}{\phi(q)}\Bigr)\ll_{A,\epsilon} \frac{x}{(\log{x})^A}.
\end{align}
\end{theorem}

Theorem \ref{theorem:Factorable} implies Corollary \ref{cor:Factorable} with ${\theta}= \tfrac{7}{32}$, extending the level to ${\pmb \vartheta} < \tfrac{66}{107} \approx 0.617$ beyond ${\pmb \vartheta} < \tfrac{3}{5}$ from Maynard \cite{JM2}. Using established sieve methods, we input our new level of distribution results for $\pi(x;q,-2)$ and $\pi(a;q,a)$, respectively, to obtain the bounds $\pi_2(x) \,\lesssim \, 3.23 \Pi_2(x)$ and and ${\rm G}(a) \,\lesssim \, 3.40\, \Pi_a(a)$, in Theorems \ref{thm:twinbound} and \ref{thm:Goldbachbound}.

In particular, all prior bounds on ${\rm G}(a)$ used level $\pmb{\vartheta}=\frac{1}{2}$ from the Bombieri--Vinogradov theorem. Now with Theorem \ref{theorem:Factorable} we obtain level $\pmb{\vartheta}_1$ beyond the `square-root barrier', as large as $\frac{153}{256}\approx 0.597$ with uniformity in the residue $|a|< x^{1+\eps}$.\footnote{As a technical note, for the results uniform in the residue $|a|< x^{1+\eps}$, we actually need $\theta$ for Ramanujan--Petersson, generalizing Selberg eigenvalue. Here the Kim--Sarnak bound $\theta\le \frac{7}{32}$ still holds.} This key ingredient leads to the greatest improvement on bounding ${\rm G}(a)$ since Bombieri--Davenport \cite{BD}.

The proof of Theorem \ref{theorem:Factorable} makes use of a refined estimate for quintilinear forms of Kloosterman sums. Such estimates have a rich history, based on the celebrated work of Deshouillers--Iwaniec \cite{DI}.

\begin{theorem}\label{thm:DI12}
Let $C, D, N, R, S \geq 1$, $q_0, c_0, d_0 \in \N$ with $(c_0d_0, q_0) = 1$. Let $b_{n, r, s}$ be a sequence supported inside $(0, N] \times (R, 2R] \times (S, 2S] \cap \N^3$. Let $g : \R^5_{>0} \rightarrow \C$ be a smooth function with compact support in $(C, 2C] \times (D, 2D] \times \R_{>0}^3$ such that
\begin{align*}
\frac{\partial^{\nu_1 + \cdots + \nu_5}}{\partial c^{\nu_1} \partial d^{\nu_2} \partial n^{\nu_3} \partial r^{\nu_4} \partial s^{\nu_5} } g(c, d, n, r, s) \ll_{{\bf \nu}} ( c^{-\nu_1} d^{-\nu_2} n^{-\nu_3} r^{-\nu_4} s^{-\nu_5})^{1-\varepsilon_0}
\end{align*}
for all ${\bf \nu} \in (\Z_{\ge0})^5$ and some small fixed $\varepsilon_0 > 0$. Then uniformly for any $a\in\Z_{\neq0}$, we have
\begin{align}
\sum_{\substack{r\sim R\\ s\sim S\\(r,s)=1}} \sum_{0<n\le N} b_{n,r,s} \sum_{\substack{c\equiv c_0 \, ({\rm mod}\, q_0)\\ d\equiv d_0 \, ({\rm mod}\, q_0) \\ (q_0 rd, sc)=1}} g(c, d, n, r, s) e\Big( an \frac{\overline{rd}}{sc}\Big) \ll_{\varepsilon, \varepsilon_0} (aq_0CDNRS)^{\varepsilon + O(\varepsilon_0)} q_0^{3/2} \mathcal{J}
\end{align}
where $\|{\bf b}\|_2 = \|\widetilde{\bf b}(1)\|_2$, $\|\widetilde{\bf b}(n'')\|_2^2 = \sum_{n,r,s}|b_{nn'',r,s}|^2$, and $\mathcal{J}=\mathcal{J}(a, C, D, N, R, S)$ is given by
\begin{align}\label{eq:JDI12}
\mathcal{J}^2  & = q_0 \sum_{\substack{n'' \mid a^{\infty}\\ n'' \leq 2N}} (an'')^{\theta}  \Big(CS\Big(\frac{N}{n''}+RS\Big)( C+DR)  + aNRS\Big)\|\widetilde{\bf b}(n'')\|_2^2 \\
& \quad +
\bigg(CS(CD\sqrt{R})^{2{\theta}} (N+RS)^{1-{\theta}}(C+DR)^{1-2{\theta}} + D^2 N R\bigg)\|{\bf b}\|_2^2. \nonumber
\end{align}
\end{theorem}

A key feature of Theorem \ref{thm:DI12} is the explicit dependence of $\mathcal J$ in terms of the exceptional eigenvalue bound $\theta$. Such an approach was initiated by Drappeau \cite{Drapp2} and Drappeau--Pratt--Radziwi{\l}{\l} \cite{DPR}. Theorem \ref{thm:DI12} refines the $\theta$-dependence of \cite{DPR} in the exceptional spectrum. From private communication, this has since induced the authors to revise \cite[Proposition 6]{DPR} in their published version, which was then applied to prove extended support for one-level density estimates for Dirichlet L-functions. In Appendix \ref{appndx:optimal}, we show that such $\theta$-dependence is optimal in this context. In particular, we recover that the original argument of Deshoulliers--Iwaniec \cite{DI} is optimal when $\theta=\frac{1}{2}$.

We also emphasize the explicit $a$-dependence of Theorem \ref{thm:DI12} with the factor of $a^\theta$, following Assing, Blomer, and Li \cite{ABL}\footnote{As mentioned, in \cite{ABL} Assing--Blomer--Li write $a^{2\theta}$ using different notation for $\theta$.}\,\footnote{In applications, the $a$-dependence of the term $aNRS$ in \eqref{eq:JDI12} is often negligible.} on $a$ in $\mathcal J(a,C,D,N,R,S)^2$. This leads to the uniform estimates in Theorem \ref{theorem:Factorable}, and in turn Theorem \ref{thm:Goldbachbound} for the Goldach problem.

As a minor remark, Theorem \ref{thm:DI12} holds for general modulus $q_0$, following \cite{Drapp2}, but for our application to Theorem \ref{theorem:Factorable}, we only use $q_0=1$.
In addition, one should regard the first term $n''=1$ as the main contribution, roughly speaking, and the sum over $n''>1$ as just a technicality. This may be seen heuristically, since the norm $\|\widetilde{\bf b}(n'')\|_2^2$ decreases with $n''$, and the divisor bound handles the sum over $n''\mid a^\infty$, $n''\le 2N$.

\section{Outline of the proof}

In this section, we outline the proof of Theorem \ref{theorem:Factorable}, giving new level of distribution for primes with triply well-factorable weights; we sketch applications to Theorems \ref{thm:twinbound} and \ref{thm:Goldbachbound} of bounding twin primes and Goldbach respresentations; and we outline the proof of the key input in Theorem \ref{thm:DI12}, giving refined estimates for quintilinear forms of Kloosterman sums.
%
%
%
%
%

\subsection*{Proof of Theorem \ref{theorem:Factorable}} The proof is a variant of the approach used to prove Theorem 1.1 in \cite{JM2}, proceeding by the dispersion method. Namely, we use a combinatorial decomposition for the primes (Heath-Brown's identity) to reduce the problem to estimating certain bilinear sums in arithmetic progressions. By several intermediate manipulations using Poisson summation and Cauchy--Schwarz, this reduces to estimating certain multidimensional exponential sums, which are ultimately bounded using the spectral theory of automorphic forms via the Kuznetsov trace formula. The proof of Maynard in \cite[Theorem 1.1]{JM2} uses well-known bounds of Deshouillers--Iwaniec \cite[Theorem 12]{DI}. Instead, we apply Theorem \ref{thm:DI12}, which gives sharper bounds coming from the exceptional spectrum, and also holds uniformly for all residues $a\neq0$ in the range $|a|< x^{1+\eps}$.

For level of distribution with triply well-factorable weights $\lambda_q$ we may exploit the additional flexibility of factorizations of the moduli, as observed in \cite{JM2}. Here we benefit from the fact that now the weights can be factored into three pieces, rather than two in the case of well-factorable weights $\lambda_q$ in \cite{BFI1}. Interestingly, there appears to be no further advantage obtained by assuming $\lambda_q$ are `quadruply well-factorable' (which can be factored into four pieces).

For example, when $a\ll1$ (i.e. $\alpha=0$) we prove that any given weights $\lambda_q$ satisfy
\begin{align*}
\sum_{\substack{q\le x^{{\pmb \vartheta}}\\ (a,q)=1}}\lambda_q\Bigl(\pi(x;q,a)-\frac{\pi(x)}{\phi(q)}\Bigr)\ll \frac{x}{(\log{x})^A},
\end{align*}
provided that (heuristically speaking) all the moduli $q$ in the support $\{q: \lambda_q\neq0\}$ have suitable factorization properties: Namely, for any $Q_1<x^{1/3}$ there exists a splitting $Q=Q_1Q_2Q_3$ such that the modulus $q$ factors as $q=q_1q_2q_3$ into integers $q_i\le Q_i$ with
\begin{align}\label{eq:outlinesystem}
Q_1^2 Q_2 Q_3^2 & \ < \ x, \nonumber \\
Q_1^2 Q_2^5 Q_3^2 & \ < \ x^2,\\
(Q_1Q_3/Q_2)^{2{\theta}} \, Q_1^2 Q_2^5 Q_3^2 & \ < \ x^{2}. \nonumber
\end{align}

In particular, triply-well factorable weights $\lambda_q$ will always satisfy these factorization properties: Indeed, for any $Q_1<x^{1/3}$ we can essentially choose a factorization $q=q_1q_2q_3$ into integers $q_i\le Q_i$, where $(Q_2,Q_3)\approx (x^{2{\pmb \vartheta}-1},\,x^{1-{\pmb \vartheta}}/Q_1)$. This choice gives us enough freedom to perfectly balance the contributions coming from diagonal and off-diagonal terms in a wide range. Plugging this in, after a short calculation we find that \eqref{eq:outlinesystem} holds up to level ${\pmb \vartheta}=\frac{5-4{\theta}}{8-6{\theta}}$, as in Theorem \ref{theorem:Factorable}. Using $\theta = \frac{7}{32}$ from Kim--Sarnak \cite{KimSarn} this gives ${\pmb \vartheta}=\frac{66}{107}\approx 0.617$, which improves over the previous record $\pmb{\vartheta}=\frac{3}{5}=0.60$ of Maynard \cite{JM2}.

With some additional technicalities, one may similarly handle uniformity in the residue $|a|< x^{\alpha}$ for the range $\alpha \le 1+\eps$. This essentially modifies the system \eqref{eq:outlinesystem} by an extra factor $x^{\alpha\theta}$, given in Lemma \ref{lemma:Factorable}. Proceeding as above one obtains level ${\pmb \vartheta} = \min\big(\frac{5-4{\theta}}{8-6{\theta}},\, \frac{5-\alpha\theta}{8}\big).$ See Proposition \ref{proposition:WellFactorable}.

\subsection*{Proofs of Theorems \ref{thm:twinbound} and \ref{thm:Goldbachbound}}
We apply the improved level of distribution results, used in Theorem \ref{theorem:Factorable}, with the weights $\lambda_q^\pm$ from the linear sieve.\footnote{We actually use Iwaniec's modified linear sieve weights $\widetilde{\lambda}_q^\pm$, which are not triply well-factorable, but are well-factorable. See \S\ref{sec:linearsievelevel}.} However, Theorem \ref{theorem:Factorable} itself does not directly apply here, since $\lambda_q^\pm$ are not triply well-factorable. Instead, we essentially work with the system \eqref{eq:outlinesystem} to obtain increased level of distribution for $\lambda_q^\pm$. Here the improved level ${\pmb \vartheta}$ will vary, depending on the `anatomy' of an individual modulus $q$. That is, heuristically speaking, the level ${\pmb \vartheta}$ will increase proportionally to how close some divisor $q_2\mid q$ get to the value $q_2\approx x^{\frac{1-\theta}{4-3\theta}}$. This is made precise in Proposition \ref{prop:piecewiseeta} below. Here ${\pmb \vartheta}$ will always be at least $\frac{7}{12} \approx 0.583$ by \cite{JM2}, and may get up to $\frac{66}{107} \approx 0.617$, as in the triply well-factorable case.

We apply our new level of distribution using linear sieve weights, which control the remainder terms in the linear sieve bounds. This increased level is then combined with sieve-theoretic techniques, including the Buchstab identity, the Chen--Iwaniec switching principal, and a recursive argument of Wu \cite{WuII}. Together, these sieve bounds leads to new results for twin primes and Golbach representations. Specifically, in Theorems \ref{thm:twinbound} and \ref{thm:Goldbachbound} we obtain $\pi_2(x) \lesssim 3.23 \Pi_2(x)$, with residue $a=-2$ (i.e. $\alpha=0$), and obtain ${\rm G}(a) \,\lesssim \, 3.40\, \Pi_a(a)$ with residue $a=x$ (i.e. $\alpha=1$).

\subsection*{Proof of Theorem \ref{thm:DI12}}

The Kuznetsov trace formula translates the quintilinear sums of Kloosterman sums into a spectral expression involving sums 
$$\sum_{n\sim N} a_n \rho_{j,1/s} (an)$$
for an arbitrary sequence $a_n$, a cusp form $f_j$ with Fourier coefficients $\rho_{j,1/s}$ at a cusp $\mathfrak{a}=1/s$.

Morally, we first wish to `factor' $\rho_{j,1/s} (an) \approx \rho_{j,1/s} (a) \rho_{j,1/s} (n)$ and then proceed to apply the spectral large sieve, as in Deshouillers--Iwaniec \cite{DI}. 
This ultimately leads to the bound $\rho_{j,1/s} (a)\ll a^{\theta}$ appearing as the $a$-dependence\footnote{In this case, $a^{\theta}$ comes from bounding the Hecke eigenvalues at the finite places. As such, it is noteworthy that the Kim--Sarnak bound $\theta\le \frac{7}{32}$, towards the Ramanujan--Petersson conjecture, holds equally for the finite places and infinite place.} for $\mathcal J(a, C, D, N, R, S)^2$. However, there are complications to this moral, since the Fourier coefficients $\rho_{j,1/s}$ are not necessarily multiplicative (much less, completely multiplicative), and moreover, one must take care to separate out the newforms from the oldforms in the spectrum. Such analysis is performed in Assing--Blomer--Li \cite{ABL}, yielding an approximate factorization for $\rho_{j,1/s}(an)$. See Lemma \ref{lem:ABLfactor} below for a precise formulation.

Finally, we establish a refined large sieve estimate for the exceptional spectrum, in terms of the best known bound $\theta \le \frac{7}{32}$, from Kim--Sarnak \cite{KimSarn}. This extends work of Drappeau--Pratt--Radziwi{\l}{\l} \cite{DPR} (and induced a revision to be made in their published version). 
To this, we consider sums over exceptional eigenvalues $\lambda_j$ of $\Gamma_0(q)$, on average over the level $q\sim Q$, roughly of the form
\begin{align*}
S(Q, N, Y) = \sum_{q\sim Q} \sum_{\lambda_j<1/4}^{(q)} Y^{\theta} \Big| \sum_{n\sim N} a_n\rho_{j\infty}(n)\Big|.
\end{align*}
We adapt the approach in Deshouillers--Iwaniec, which uses a recursive relation between $S(Q, N, Y)$ and $S(NY/Q, N, Y)$ in certain ranges. The main idea is to iteratively apply this recursive relation in steps, interleaved with the spectral large sieve and basic estimates, in a particular order which depends on the relative sizes of the parameters $Q$, $N$, and $Y$. This gives bound roughly of the form, for example in Theorem \ref{thm:DI7} with $a_n=\1_{n\sim N}$,
\begin{align*}
S(Q, N, Y) \ll (QNY)^\eps\big(Q+N + (NY)^{\theta}[N^{1-2\theta} + Q^{1-2\theta}]\big)N.
\end{align*}

In Appendix \ref{appndx:optimal}, we prove that such estimates have optimal $\theta$-dependence in this context. In particular, we recover that the original argument of Deshoulliers--Iwaniec \cite{DI} is optimal when $\theta=\frac{1}{2}$. To this, we study a heuristic model of the large sieve estimates for the exceptional spectrum. This model is slightly different than the proof, but we feel it better motivates the iterative steps in the original arguments of~\cite{DI}, and explains the shape of the final bound.

%
%
%
%
%
%
\addtocontents{toc}{\protect\setcounter{tocdepth}{0}}
\section*{Acknowledgements}
%
%
%
%
%
%
%
%
The author would like to thank Sary Drappeau, James Maynard, Lasse Grimmelt, Jori Merikoski, Alex Pascadi, and Junxian Li for many helpful discussions. In particular, the author is grateful to Sary Drappeau and Universit\'e d'Aix-Marseille for their hospitality, during which a portion of this article was written.
The author was supported by Balliol College and the Clarendon Scholarship at the University of Oxford, as well as the European Research Council (ERC) under the European Union’s Horizon 2020 research and innovation programme (grant agreement No 851318).

%
%
%
%
%
%
%
%
\section*{Notation}
%
%
%
%
%
%
%
%
We will use the Vinogradov $\ll$ and $\gg$ asymptotic notation, and the big oh $O(\cdot)$ and $o(\cdot)$ asymptotic notation. $f\asymp g$ will denote the conditions $f\ll g$ and $g\ll f$ both hold. We will also use the standard notation $f\sim g$ to denote $f=(1+o(1)) g$, as well as the (non-standard) notation $f\lesssim g$ and $f\gtrsim g$ to denote $f \le (1+o(1)) g$ and $f \ge (1+o(1)) g$, respectively. Dependence on a parameter will be denoted by a subscript.

The letter $p$ will always be reserved to denote a prime number. We use $\phi$ to denote the Euler totient function, $e(x):=e^{2\pi i x}$ the complex exponential, $\tau_k(n)$ the $k$-fold divisor function, $\mu(n)$ the M\"obius function. We let $P^-(n)$, $P^+(n)$ denote the smallest and largest prime factors of $n$ respectively. We let $n\mid a^\infty$ denote the condition $p\mid n \implies p\mid a$.

We let $\hat{f}$ denote the Fourier transform of $f$ over $\mathbb{R}$ - i.e. $\hat{f}(\xi)=\int_{-\infty}^{\infty}f(t)e(-\xi t)dt$. We use $\mathbf{1}$ to denote the indicator function of a statement. For example,
\[
\mathbf{1}_{n\equiv a\Mod{q}}=\begin{cases}1,\qquad &\text{if }n\equiv a\Mod{q},\\
0,&\text{otherwise}.
\end{cases}
\]
For $(n,q)=1$, we will use $\overline{n}$ to denote the inverse of the integer $n$ modulo $q$; the modulus will be clear from the context. For example, we may write $e(a\overline{n}/q)$ - here $\overline{n}$ is interpreted as the integer $m\in \{0,\dots,q-1\}$ such that $m n\equiv 1\Mod{q}$. Occasionally we will also use $\overline{\lambda}$ to denote complex conjugation; the distinction of the usage should be clear from the context.  For a complex sequence $\alpha_{n_1,\dots,n_k}$, $\|\alpha\|_2$ will denote the $\ell^2$ norm $\|\alpha\|_2=(\sum_{n_1,\dots,n_k}|\alpha_{n_1,\dots,n_k}|^2)^{1/2}$.

Summations assumed to be over all positive integers unless noted otherwise. We use the notation $n\sim N$ to denote the conditions $N<n\le 2N$.

We will let $z_0:=x^{1/(\log\log{x})^3}$ and $y_0:=x^{1/\log\log{x}}$ two parameters depending on $x$, which we will think of as a large quantity. We will let $\psi_0:\mathbb{R}\rightarrow\mathbb{R}$ denote a fixed smooth function supported on $[1/2,5/2]$ which is identically equal to $1$ on the interval $[1,2]$ and satisfies the derivative bounds $\|\psi_0^{(j)}\|_\infty\ll (4^j j!)^2$ for all $j\ge 0$. (See \cite[Page 368, Corollary]{BFI2} for the construction of such a function.)

We will repeatedly make use of the following condition.
\begin{definition}[Siegel-Walfisz condition]
We say that a complex sequence $\alpha_n$ satisfies the \textbf{Siegel-Walfisz condition} if for every $d\ge 1$, $q\ge 1$ and $(a,q)=1$ and every $A>1$ we have
\begin{equation}
\Bigg|\sum_{\substack{n\sim N\\ n\equiv a\Mod{q}\\ (n,d)=1}}\alpha_n-\frac{1}{\phi(q)}\sum_{\substack{n\sim N\\ (n,d q)=1}}\alpha_n\Bigg|\ll_A \frac{N\tau(d)^{O(1)}}{(\log{N})^A}.
\label{eq:SiegelWalfisz}
\end{equation}
\end{definition}
We note that $\alpha_n$ satisfies the Siegel-Walfisz condition if $\alpha_n=1$, if $\alpha_n=\mu(n)$, or if $\alpha_n$ is the indicator function of the primes.
%
%
%
%
%
%
%
\addtocontents{toc}{\protect\setcounter{tocdepth}{1}}
\section{Level of distribution for primes}\label{sec:Factorable}
%
%
%
%
%
%
%
%
In this section we establish Theorem \ref{theorem:Factorable} assuming two propositions, namely Proposition \ref{proposition:WellFactorable} and Proposition \ref{proposition:DoubleDivisor}, given below.
%
%
%
%
%
%
%
%
\begin{proposition}[Well-factorable Type II estimate]\label{proposition:WellFactorable}
Let $a\in\mathbb{Z}_{\neq0}$, $|a|< x^{\alpha}$, $NM\asymp x$,
\[
x^\epsilon\le N\le x^{1/3+\epsilon}.
\]
Let $\lambda_q$ be triply well-factorable of level $Q\le x^{\pmb \vartheta}$ for
$${\pmb \vartheta} = \min\Big(\frac{5-4{\theta}}{8-6{\theta}},\, \frac{5-\alpha\theta}{8}\Big)-50\epsilon.$$ 
Let $\alpha_n,\beta_m$ be complex sequences such that $|\alpha_n|,|\beta_n|\le \tau(n)^{B_0}$ and $\alpha_n$ satisfies the Siegel-Walfisz condition \eqref{eq:SiegelWalfisz} and is supported on $P^-(n)\ge z_0$. Then we have that for every choice of $A>0$ and every interval $\mathcal{I}\subseteq[x,2x]$
\[
\sup_{0<|a|< x^{\alpha}}\sum_{q\le Q}\lambda_q\sum_{n\sim N}\alpha_n\sum_{\substack{m\sim M\\ mn\in\mathcal{I}}}\beta_m\Bigl(\mathbf{1}_{nm\equiv a\Mod{q}}-\frac{\mathbf{1}_{(nm,q)=1}}{\phi(q)}\Bigr)\ll_{A,B_0}\frac{x}{(\log{x})^A}.
\]
\end{proposition}
%
%
%
%
%
%
%
%
Proposition \ref{proposition:WellFactorable} is our key new ingredient behind the proof, and will be established in Section \ref{sec:WellFactorable}.
%
%
%
%
%
%
%
%
\begin{proposition}[Divisor function in progressions] \label{proposition:DoubleDivisor}
Let $N_1,N_2\ge x^{3\epsilon}$ and $N_1N_2M\asymp x$ and
\begin{align*}
Q&\le \Bigl(\frac{x}{M}\Bigr)^{2/3-3\epsilon}.
\end{align*}
Let $\mathcal{I}\subset[x,2x]$ be an interval, and let $\alpha_m$ a complex sequence with $|\alpha_m|\le \tau(m)^{B_0}$. Then we have that for every $A>0$
\[
\sup_{a\neq0}\sum_{\substack{q\sim Q\\(a,q)=1}}\Bigl|\sum_{\substack{n_1\sim N_1\\ P^-(n)\ge z_0}}\sum_{\substack{n_2\sim N_2\\ P^-(n)\ge z_0}}\sum_{\substack{m\sim M\\ m n_1n_2\in\mathcal{I} }}\alpha_m\Bigl(\mathbf{1}_{m n_1 n_2\equiv a\Mod{q}}-\frac{\mathbf{1}_{(m n_1 n_2,q)=1}}{\phi(q)}\Bigr)\Bigr|\ll_{A,B_0} \frac{x}{(\log{x})^A}.
\]
Moreover, the same result holds when the summand is multiplied by $\log{n_1}$.
\end{proposition}
\begin{proof}
The proof is given in \cite[Proposition 5.2]{JM2}.
This is a quick consequence of the Weil bound (hence admitting the uniformity in the residue, $\sup_a$) and the fundamental lemma of sieves. This is originally due to independent unpublished work of Selberg and Hooley.
\end{proof}
%
%
%
%
%
%
%
%

Finally, we require a suitable combinatorial decomposition of the primes.
%
%
%
%
%
%
%
%
\begin{lemma}[Heath-Brown identity]\label{lemma:HeathBrown}
Let $k\ge 1$ and $n\le 2x$. Then we have
\[
\Lambda(n)=\sum_{j=1}^k (-1)^j \binom{k}{j}\sum_{\substack{n=n_1\cdots n_{j}m_1\cdots m_j\\ m_1,\dots,m_j\le 2x^{1/k}}}\mu(m_1)\cdots \mu(m_j)\log{n_{1}}.
\]
\end{lemma}
\begin{proof}
See \cite{HBVaughan}.
\end{proof}
%
%
%
%
%
%
%
%
\begin{lemma}[Consequence of the fundamental lemma of the sieve]\label{lemma:Fundamental}
Let $q,t,x\ge 2$ satisfy $q x^\epsilon \le t$ and let $(b,q)=1$. Recall $z_0=x^{1/(\log\log{x})^3}$. Then we have
\[
\sum_{\substack{n\le t\\ n\equiv b\Mod{q}\\ P^-(n)\ge z_0}}1=\frac{1}{\phi(q)}\sum_{\substack{n\le t\\ P^-(n)\ge z_0}}1+O_A\Bigl(\frac{t}{q(\log{x})^A}\Bigr).
\]
\end{lemma}
\begin{proof}
This is an immediate consequence of the fundamental lemma of sieve methods - see, for example, \cite[Theorem 6.12]{Opera}.
\end{proof}
%
%
%
%
%
%
%
%
\begin{proof}[Proof of Theorem \ref{theorem:Factorable} assuming Proposition \ref{proposition:WellFactorable} and \ref{proposition:DoubleDivisor}]
By partial summation (noting that prime powers contribute negligibly and retaining the conditon $P^-(n)\ge z_0$), it suffices to show that for all $t\in [x,2x]$
\[
\sup_{0<|a|< x^\alpha}\sum_{q\le x^{\pmb \vartheta} }\lambda_q\sum_{\substack{x\le n\le t\\ P^-(n)\ge z_0}}\Lambda(n)\Bigl(\mathbf{1}_{n\equiv a\Mod{q}}-\frac{\mathbf{1}_{(n,q)=1}}{\phi(q)}\Bigr)\ll_A \frac{x}{(\log{x})^A}.
\]
We now apply Lemma \ref{lemma:HeathBrown} with $k=3$ to expand $\Lambda(n)$ into various subsums, and put each variable into one of $O(\log^6{x})$ dyadic intervals. Thus it suffices to show that for all choices of $N_1,N_2,N_3,M_1,M_2,M_3$ with $M_1M_2M_3N_1N_2N_3\asymp x$ and $M_i\le x^{1/3}$ we have
\begin{align*}
\sup_{0<|a|< x^\alpha}\sum_{q\le x^{\pmb \vartheta}} \lambda_q\sum_{\substack{m_1,m_2,m_3,n_1,n_2,n_3\\ n_i\sim N_i\,\forall i\\ m_i\sim M_i\,\forall i\\ x\le n \le t\\ P^-(n_i),P^-(m_i)\ge z_0\,\forall i}}\mu(m_1)\mu(m_2)\mu(m_3)(\log{n_1})\Bigl(\mathbf{1}_{n\equiv a\Mod{q}}-\frac{\mathbf{1}_{(n,q)=1}}{\phi(q)}\Bigr)\\
\ll_A \frac{x}{(\log{x})^{A+6}},
\end{align*}
where we have written $n=n_1n_2n_3m_1m_2m_3$ in the expression above for convenience. 

By grouping all but one variable together, Proposition \ref{proposition:WellFactorable} gives this if any of the $N_i$ or $M_i$ lie in the interval $[x^\epsilon,x^{1/3+\epsilon}]$, and so we may assume all are either smaller than $x^\epsilon$ or larger than $x^{1/3+\epsilon}$. Since $M_i\le x^{1/3}$, we may assume that $M_1,M_2,M_3\le x^\epsilon$. There can be at most two of the $N_i$'s which are larger than $x^{1/3+\epsilon}$ since $M_1M_2M_3N_1N_2N_3\asymp x$. 

If only one of the $N_i$'s are greater than $x^{1/3+\epsilon}$ then they must be of size $\gg x^{1-5\epsilon}>x^\epsilon q$, and so the result is trivial by summing over this variable first and using Lemma \ref{lemma:Fundamental}.

If two of the $N_i$'s are greater, say $N_1,N_2 > x^{1/3+\epsilon}$, and all the other variables are less than $x^\epsilon$, then the result follows from Proposition \ref{proposition:DoubleDivisor} with $M=M_1M_2M_3N_3 \le x^{4\epsilon}$, since $Q \le x^{5/8} \le (x/M)^{2/3-3\epsilon}$. This gives the result.
\end{proof}
%
%
%
%
%
%
%
%
To complete the proof of Theorem \ref{theorem:Factorable}, we are left to establish Proposition \ref{proposition:WellFactorable}, which we will do in Section \ref{sec:WellFactorable}.
%
%
%
%
%
%
%
%
\section{Preliminary dispersion method lemmas}\label{sec:Lemmas}

In this section, we collect preliminary lemmas that will be used in the dispersion method.

%
%
%
%
%
%
%
%
\begin{lemma}[Divisor function bounds]\label{lemma:Divisor}
Let $|b|< x-y$ and $y\ge q x^\epsilon$. Then we have
\[
\sup_{a\neq0}\sum_{\substack{x-y\le n\le x\\ n\equiv a\Mod{q}}}\tau(n)^C\tau(n-b)^C \ll \frac{y}{q} (\tau(q)\log{x})^{O_{C}(1)}.
\]
\end{lemma}
\begin{proof}
This follows from Shiu's Theorem \cite{Shiu}, and is given in \cite[Lemma 8.7]{JM1}.
\end{proof}
\begin{lemma}[Separation of variables from inequalities]\label{lemma:Separation}
Let $\mathcal Q\subset[ x^\eps, x^{1-\epsilon}]$. Let $N_1,\dots, N_r\ge z_0$ satisfy $N_1\cdots N_r\asymp x$. Let $\alpha_{n_1,\dots,n_r}$ be a complex sequence with $|\alpha_{n_1,\dots,n_r}|\le (\tau(n_1)\cdots \tau(n_r))^{B_0}$. Then, for any choice of $A>0$ there is a constant $C=C(A,B_0,r)$ and intervals $\mathcal{I}_1,\dots,\mathcal{I}_r$ with $\mathcal{I}_j\subseteq [P_j,2P_j]$ of length $\le P_j(\log{x})^{-C}$ such that
\begin{align*}
\sup_{a\neq0}\sum_{\substack{q\in \mathcal Q\\ (q,a)=1}}&\Bigl|\,\sideset{}{^*}\sum_{\substack{n_1,\dots,n_r\\ n_i\sim N_i\forall i}}\alpha_{n_1,\dots,n_r}
\Bigl(\mathbf{1}_{n_1\cdots n_r\equiv a\Mod{q}}-\frac{\mathbf{1}_{(n_1\cdots n_r,q)=1}}{\phi(q)}\Bigr)\Bigr| \\
&  \ll_r \ \frac{x}{(\log{x})^A} + (\log{x})^{r C}\sup_{a\neq0}\sum_{\substack{q\in \mathcal Q\\ (q,a)=1}}\Bigl|\sum_{\substack{n_1,\dots,n_r\\ n_i\in \mathcal{I}_i\forall i}}\alpha_{n_1,\dots,n_r}\Bigl(\mathbf{1}_{n_1\cdots n_r\equiv a\Mod{q}}-\frac{\mathbf{1}_{(n_1\cdots n_r,q)=1}}{\phi(q)}\Bigr)\Bigr|.
\end{align*}
Here $\sum^*$ means that the summation is restricted to $O(1)$ inequalities of the form $n_1^{\alpha_1}\cdots n_r^{\alpha_r}\le B$ for some constants $\alpha_1,\dots \alpha_r$ and some quantity $B$.  The implied constant may depend on all such exponents $\alpha_i$, but none of the quantities $B$.
\end{lemma}
\begin{proof}
This is \cite[Lemma 8.10]{JM1}, noting that the argument is uniform in the residue $a$.
\end{proof}

%
%
%
%
%
%
%
%
\begin{lemma}\label{lemma:SiegelWalfiszMaintain}
Let $C,B>0$ be constants and let $\alpha_n$ be a sequence satisfing the Siegel-Walfisz condition \eqref{eq:SiegelWalfisz}, supported on $n\le 2x$ with $P^-(n)\ge z_0=x^{1/(\log\log{x})^3}$ and satisfying $|\alpha_n|\le \tau(n)^B$. Then $\mathbf{1}_{\tau(n)\le (\log{x})^C}\alpha_n$ also satisfies the Siegel-Walfisz condition.
\end{lemma}
\begin{proof}
This is \cite[Lemma 13.7]{JM1}.
\end{proof}

%
%
%
%
%
%
%
%
\begin{lemma}[Most moduli have small smooth part]\label{lemma:RoughModuli}
Let $Q<x^{1-\epsilon}$ and $A,B>0$. Let $\gamma_b$ be a complex sequence with $|\gamma_b|\le \tau(n)^{B}$ and set $z_0:=x^{1/(\log\log{x})^3}$ and $y_0:=x^{1/\log\log{x}}$. Let $sm(n;z)$ denote the $z$-smooth part of $n$. (i.e. $sm(n;z)=\prod_{p\le z}p^{\nu_p(n)}$). Then we have that
\[
\sum_{\substack{q\sim Q\\ sm(q;z_0)\ge y_0}}\sup_{(a,q)=1}\Bigl|\sum_{b\le  x}\gamma_b\Bigl(\mathbf{1}_{b\equiv a\Mod{q}}-\frac{\mathbf{1}_{(b,q)=1}}{\phi(q)}\Bigr)\Bigr|\ll_{A,B} \frac{x}{(\log{x})^A}.
\]
\end{lemma}
\begin{proof}
This is \cite[Lemma 10.11]{JM3}.
\end{proof}
%
%
%
%
%
%
%
%
\begin{proposition}[Reduction to exponential sums]\label{proposition:GeneralDispersion}
Let $\alpha_n,\beta_m,\gamma_{q,d},\lambda_{q,d,r}$ be complex sequences with $|\alpha_n|,|\beta_n|\le \tau(n)^{B_0}$ and $|\gamma_{q,d}|\le \tau(q d)^{B_0}$ and $|\lambda_{q,d,r}|\le \tau(q d r)^{B_0}$. Let $\alpha_n$ and $\lambda_{q,d,r}$ be supported on integers with $P^-(n)\ge z_0$ and $P^-(r)\ge z_0$, and let $\alpha_n$ satisfy the Siegel-Walfisz condition \eqref{eq:SiegelWalfisz}. Let
\[
\mathcal{S}:=\sup_{0<|a|< x^{1+\eps}} \sum_{\substack{d\sim D\\ (d,a)=1}}\sum_{\substack{q\sim Q\\ (q,a)=1}}\sum_{\substack{r\sim R\\ (r,a)=1}}\lambda_{q,d,r}\gamma_{q,d}\sum_{m\sim M}\beta_m\sum_{n\sim N}\alpha_n\Bigl(\mathbf{1}_{m n\equiv a\Mod{q r d}}-\frac{\mathbf{1}_{(m n,q r d)=1}}{\phi(q r d)}\Bigr).
\]
Let $A>0$ and $C=C(A,B_0)$ be sufficiently large in terms of $A,B_0$, and let $N,M$ satisfy
\[
N>Q D (\log{x})^{C},\qquad M>(\log{x})^C.
\]
Then we have
\[
|\mathcal{S}| \ll_{A,B_0} \frac{x}{(\log{x})^A}+M D^{1/2}Q^{1/2}(\log{x})^{O_{B_0}(1)}\Bigl(|\mathcal{E}_1|^{1/2}+|\mathcal{E}_2|^{1/2}\Bigr),
\]
where
\begin{align*}
\mathcal{E}_{1}&:= \sup_{0<|a|< x^{1+\eps}}\sum_{\substack{q\\ (q,a)=1}}\sum_{\substack{d\sim D\\ (d,a)=1}}\sum_{\substack{r_1,r_2\sim R\\ (r_1r_2,a)=1}}\psi_0\Bigl(\frac{q}{Q}\Bigr)\frac{\lambda_{q,d,r_1}\overline{\lambda_{q,d,r_2}} }{\phi(q d r_2)q d r_1}\sum_{\substack{n_1,n_2\sim N\\ (n_1,q d r_1)=1\\(n_2,q d  r_2)=1}}\alpha_{n_1}\overline{\alpha_{n_2}}\\
&\qquad \times\sum_{1\le |h|\le H_1}\hat{\psi}_0\Bigl(\frac{h M}{q d r_1}\Bigr)e\Bigl( \frac{a h \overline{ n_1}}{q d r_1}\Bigr),\\
\mathcal{E}_2&:= \sup_{0<|a|< x^{1+\eps}}\sum_{\substack{q\\ (q,a)=1}}\psi_0\Bigl(\frac{q}{Q}\Bigr)\sum_{\substack{d\sim D\\ (d,a)=1}}\sum_{\substack{r_1,r_2\sim R\\ (r_1,a r_2)=1\\ (r_2,a q d r_1)=1}}\frac{\lambda_{q,d,r_1}\overline{\lambda_{q,d,r_2}}}{q d r_1 r_2}\sum_{\substack{n_1,n_2\sim N\\ n_1\equiv n_2\Mod{q d}\\ (n_1,n_2 q d r_1)=1\\(n_2,n_1 q d r_2)=1\\ |n_1-n_2|\ge N/(\log{x})^C}}\alpha_{n_1}\overline{\alpha_{n_2}}\\
&\qquad \times\sum_{1\le |h|\le H_2}\hat{\psi}_0\Bigl(\frac{h M}{q d r_1 r_2}\Bigr)e\Bigl(\frac{ah\overline{n_1r_2}}{q d r_1}+\frac{ah\overline{n_2 q d r_1}}{r_2}\Bigr),\\
H_1&:=\frac{Q D R}{M}\log^5{x},\\
H_2&:=\frac{Q D R^2}{M}\log^5{x}.
\end{align*}
\end{proposition}
\begin{proof}
This is \cite[Proposition 14.4]{JM1} with $E=1$. The argument is uniform in the residue class $a$ (moreso than we need).
\end{proof}
%
%
%
%
%
%

The following lemma imposes the uniformity constraint $|a|< x^{1+\eps}$.
\begin{lemma}[Simplification of exponential sum]\label{lemma:Simplification}
Let $N,M,Q,R \le x$ with $NM\asymp x$ and 
\begin{align}
Q R&<x^{2/3},\label{eq:CrudeSize}\\
Q R^2&< M x^{1-3\epsilon}.\label{eq:CrudeSize2}
\end{align}
Let $\lambda_{q,r}$ and $\alpha_n$ be complex sequences supported on $P^-(n),P^-(r)\ge z_0$ with $|\lambda_{q,r}|\le \tau(qr)^{B_0}$ and $|\alpha_n|\le \tau(n)^{B_0}$. Let $H:=\frac{Q R^2}{M}\log^5{x}$ and let
\begin{align*}
\mathcal{E}&:= \sup_{0<|a|< x^{1+\eps}}\sum_{\substack{(q,a)=1}}\psi_0\Bigl(\frac{q}{Q}\Bigr)\sum_{\substack{r_1,r_2\sim R\\ (r_1,a r_2)=1\\ (r_2,a q r_2)=1}}\frac{\lambda_{q,r_1}\overline{\lambda_{q,r_2}}}{q r_1 r_2}\sum_{\substack{n_1,n_2\sim N\\ n_1\equiv n_2\Mod{q}\\ (n_1,n_2qr_1)=1\\(n_2,n_1qr_2)=1\\ |n_1-n_2|\ge N/(\log{x})^C}}\alpha_{n_1}\overline{\alpha_{n_2}}\\
&\qquad\qquad \times\sum_{1\le |h|\le H}\hat{\psi}_0\Bigl(\frac{h M}{q r_1 r_2}\Bigr)e\Bigl(\frac{ah\overline{n_1 r_2}}{q r_1}+\frac{ah\overline{n_2 q r_1}}{r_2}\Bigr).
\end{align*}
Then we have (uniformly in $C$)
\[
\mathcal{E}\ll_{B_0} \exp((\log\log{x})^5) \sup_{\substack{0<|a|< x^{1+\eps}\\R_1,R_2\le 2R}} \sup_{\substack{H'\le H\\ Q'\le 2Q}} |\mathcal{E}'|+\frac{N^2}{Qx^\epsilon},
\]
where
\[
\mathcal{E}'=\sum_{\substack{Q\le q\le Q'\\ (q,a)=1}}\sum_{\substack{R\le r_1\le  R_1\\ R\le r_2\le R_2\\ (r_1a r_2)=1\\ (r_2,a q r_1)=1}}\frac{\lambda_{q,r_1}\overline{\lambda_{q,r_2}}}{q r_1 r_2}\sum_{\substack{n_1,n_2\sim N\\ n_1\equiv n_2\Mod{q}\\ (n_1,qr_1n_2)=1\\ (n_2,qr_2n_1)=1\\ (n_1r_2,n_2)\in\mathcal{N}\\ |n_1-n_2|\ge N/(\log{x})^C}}\alpha_{n_1}\overline{\alpha_{n_2}}\sum_{1\le |h| \le H'} e\Bigl(\frac{ ah\overline{n_2 q r_1}(n_1-n_2)}{n_1 r_2}\Bigr),
\]
and $\mathcal{N}$ is a set with the property that if $(a,b)\in\mathcal{N}$ and $(a',b')\in\mathcal{N}$ then we have $\gcd(a,b')=\gcd(a',b)=1$.
\end{lemma}
\begin{proof}
This follows as in \cite[Lemma 14.5]{JM1}. The only minor modification to the proof needed to obtain uniformity in $a$ is in the application of Bezout's identity. We provide this for completeness: Indeed, by Bezout's identity,
\begin{align*}
\frac{ah\overline{n_1 r_2}}{q r_1}&=\frac{-ah\overline{q r_1}}{n_1 r_2}+\frac{ah}{q r_1 r_2 n_1}\Mod{1}.
\end{align*}
Since $|h|\le H=(QR\log^5{x})/M$, the final fraction is of size $O(a\log^5{x}/x)$, and so we see that
\[
e\Bigl(\frac{ah\overline{n_1r_2}}{q r_1}+\frac{ah \overline{n_2 qr_1}}{r_2}\Big)=e\Bigl(\frac{ ah\overline{n_2 q r_1}(n_1-n_2)}{n_1 r_2}\Bigr)+O\Bigl(\frac{a}{x}\log^6{x}\Bigr).
\]
Assuming the residue is $|a|< x^{1+\eps}$, the error term above contributes to $\mathcal{E}$ a total
\[
\ll N\Bigl(\frac{N}{Q}+1\Bigr)\frac{Q R^2(\log{x})^{O_{B_0}(1)}}{M x^{1-\eps}} \ll_{B_0} \frac{N^2 x^{\eps}}{Q}\Bigl(\frac{Q  R^2}{M x}+\frac{Q^2 R^2}{x^2}\Bigr) \ll_{B_0} \frac{N^2}{Q x^\epsilon},
\]
using \eqref{eq:CrudeSize} and \eqref{eq:CrudeSize2}. The proof that $\mathcal{E}\ll_{B_0} \exp((\log\log{x})^5)\sup|\mathcal{E}'|+ N^2/Qx^\epsilon$ now follows exactly as in \cite[Lemma 14.5]{JM1}.
\end{proof}
%
%
%
%
%
%
%
%
\begin{lemma}[Second exponential sum estimate]\label{lemma:BFI2}
Let
\begin{align}
D R N^{3/2}&< x^{1-2\epsilon},\\
Q D R&< x^{1-2\epsilon}.
\end{align}
Let $\alpha_n$, $\lambda_{d,r}$ be complex sequences with $|\lambda_{d,r}|,|\alpha_n|\le x^{o(1)}$. Let $H_1:=N Q D R(\log{x})^5/x$ and let 
\[
\widetilde{\mathcal{B}}:= \sup_{a\neq0} \sum_{\substack{q\\ (q,a)=1}}\sum_{\substack{d\sim D\\ (d,a)=1}}\sum_{\substack{r_1,r_2\sim R\\ (r_1r_2,a)=1}}\psi_0\Bigl(\frac{q}{Q}\Bigr)\frac{\lambda_{d,r_1}\overline{\lambda_{d,r_2}} }{\phi(q d r_2)q d r_1}\sum_{\substack{n_1,n_2\sim N\\ (n_1,q d r_1)=1\\(n_2,q d  r_2)=1}}\alpha_{n_1}\overline{\alpha_{n_2}}\sum_{1\le |h|\le H_1} \hat{\psi}_0\Bigl(\frac{h M}{q d r_1}\Bigr)e\Bigl( \frac{a h \overline{ n_1}}{q d r_1}\Bigr)
\]
Then we have
\[
\widetilde{\mathcal{B}}\ll\frac{N^2}{Q D x^\epsilon}.
\]
\end{lemma}
\begin{proof}
This is \cite[Lemma 6.10]{JM2}. 
The argument is uniform in the residue $a$, only applying Cauchy-Schwarz, a smooth majorant, and the Weil bound. See \cite[Lemma 17.3]{JM1}.
\end{proof}
%
%
%
%
%
%
%
%
\begin{lemma}[Reduction to smoothed sums]\label{lemma:SmoothReduction}
Let $N\ge x^\epsilon$ and $z\le z_0$ and let $\alpha_m$, $c_q$ be 1-bounded complex sequences.

Imagine that for every choice of $N',D,A,C>0$ with $N' D\asymp N$ and $D\le y_0$, and every smooth function $f$ supported on $[1/2,5/2]$ satisfying $f^{(j)}\ll_j (\log{x})^{C j}$, and for every $1$-bounded complex sequence $\beta_d$ we have the estimate
\[
\sup_{a\neq0} \sum_{q\sim Q} c_q\sum_{m\sim M}\alpha_m\sum_{d\sim D}\beta_d\sum_{n'}f\Bigl(\frac{n'}{N'}\Bigr)\Bigl(\mathbf{1}_{m n' d\equiv a\Mod{q}}-\frac{\mathbf{1}_{(m n' d,q)=1}}{\phi(q)}\Bigr)\ll_{A,C} \frac{x}{(\log{x})^A}.
\]
Then for any $B>0$ and every interval $\mathcal{I}\subseteq [N,2N]$ we have
\[
\sup_{a\neq0}\sum_{q\sim Q}c_q \sum_{m\sim M} \alpha_m\sum_{\substack{n\in\mathcal{I}\\ P^-(n)>z}}\Bigl(\mathbf{1}_{mn\equiv a\Mod{q}}-\frac{\mathbf{1}_{(m n,q)=1}}{\phi(q)}\Bigr)\ll_{B} \frac{x}{(\log{x})^B}.
\]
\end{lemma}
\begin{proof}
This is \cite[Lemma 19.2]{JM1}. The argument is completely uniform in the residue $a$, using a smooth partition of unity and the Fundamental Lemma of the Sieve.
\end{proof}
%
%
%
%
%
%
%
%
%
%
%
%
%
%
%
%
\section{Well-factorable disperson estimates}\label{sec:WellFactorable}
%
%
%
%
%
%
%
%
In this section we establish Proposition \ref{proposition:WellFactorable}, which is the main result toward Theorem \ref{theorem:Factorable}. This can be viewed as a refinement of \cite[Theorem 1]{BFI1}. Indeed, Proposition \ref{proposition:WellFactorable} essentially includes \cite[Theorem 1]{BFI1} as the special case $R=1$. The key advantage in our setup is to make use of the additional flexibility afforded by having a third factor available when manipulating the exponential sums. The argument does not have a specific regime when it is weakest; the critical case for Theorem \ref{theorem:Factorable} is the whole range $x^{1/10}\le N\le x^{1/3}$. (The terms with $N\le x^{1/10}$ or $N>x^{1/3}$ can be handled by a combination of the result for $N\in[x^{1/10},x^{1/3}]$ and Proposition \ref{proposition:DoubleDivisor}.)
%
%
%
%
%
%
%
%
\begin{lemma}[Well-factorable exponential sum estimate]\label{lemma:Factorable}
Let $Q < Nx^{-\epsilon}$, $NM\asymp x$, and 
\begin{align}
N^2 R^2 S&< x^{1-7\epsilon},\\
a^{\theta}\; N R^2 S^5 Q&<x^{2-14\epsilon},\\
(NR/S)^{2{\theta}}\;N R^2 S^5 Q&<x^{2-14\epsilon}. 
\end{align}
Let $Q'\le 2Q$, $H'\le x^{o(1)} QR^2 S^2/M$, and $\gamma_r,\lambda_s,\alpha_n$ be 1-bounded complex coefficients, and denote
\begin{align*}
\mathcal{W}&:= \sup_{0<|a|< x^{1+\eps}}\sum_{\substack{Q\le q\le Q'\\ (q,a)=1}}\sum_{\substack{r_1,r_2\sim R}}\sum_{\substack{s_1,s_2\sim S \\ (r_1s_1,a r_2s_2)=1\\ (r_2s_2,a q d r_1 s_1)=1\\ r_1s_1\le B_1\\ r_2s_2\le B_2}}\frac{\gamma_{r_1}\lambda_{s_1}\overline{\gamma_{r_2}\lambda_{s_2}}}{r_1r_2s_1s_2q}\sum_{\substack{n_1,n_2\sim N \\ n_1\equiv n_2\Mod{q d}\\ (n_1,n_2 q d r_1 s_1)=1\\ (n_2,n_1 q d r_2 s_2)=1\\ (n_1r_2s_2,n_2)\in\mathcal{N}\\ |n_1-n_2|\ge N/(\log{x})^C }}\alpha_{n_1}\overline{\alpha_{n_2}}\\
&\qquad\times\sum_{1\le |h| \le H'}e\Bigl(\frac{ah(n_1-n_2)\overline{n_2 r_1 s_1 d q}}{ n_1 r_2s_2}\Bigr)
\end{align*}
for some $(d,a)=1$ where $\mathcal{N}$ is a set with the property that if $(a,b)\in\mathcal{N}$ and $(a',b')\in\mathcal{N}$ then $\gcd(a,b')=\gcd(a',b)=1$.

Then we have
\[
\mathcal{W}\ll \frac{N^2}{Q x^\epsilon}.
\]
\end{lemma}
\begin{proof}
We first make a change of variables. Since we have $n_1\equiv n_2\Mod{q d}$, we let $f d q=n_1-n_2$ for some integer $|f|\le 2N/d Q\le 2N/Q$, and we wish to replace $q$ with $(n_1-n_2)/d f$. We see that 
\[
(n_1-n_2)\overline{d q}=f\Mod{n_1 r_2 s_2}.
\]
Thus the exponential simplifies to
\[
e\Bigl(\frac{ah f\overline{r_1s_1n_2}}{n_1r_2s_2}\Bigr).
\]
The conditions $(n_1,n_2)=1$ and $n_1\equiv n_2\Mod{d q}$ automatically imply $(n_1n_2,d q)=1$, and so we find
\begin{align*}
\mathcal{W}&= \sup_{0<|a|< x^{1+\eps}}\sum_{1\le |f|\le 2N/Q}\sum_{\substack{r_1,r_2\sim R\\ (r_1r_2,a)=1}}\sum_{\substack{s_2\sim S\\ (r_2s_2,a d r_1)=1\\ r_2s_2\le B_2}}\sideset{}{'}\sum_{\substack{n_1,n_2\sim N\\ n_1\equiv n_2\Mod{d f}}}\frac{\gamma_{r_1}\overline{\gamma_{r_2}\lambda_{s_2}}d f}{r_1r_2 s_2(n_1-n_2)}\\
&\qquad\times\sum_{\substack{s_1\sim S\\ (s_1,a n_1 r_2 s_2)=1\\ r_1s_1\le B_1}}\frac{\lambda_{s_1}}{s_1}\sum_{1\le |h|\le H'}\alpha_{n_1}\overline{\alpha_{n_2}}e\Bigl(\frac{a h f\overline{r_1s_1n_2}}{n_1r_2s_2}\Bigr).
\end{align*}
Here we have used $\sum'$ to denote that fact that we have suppressed the conditions
\begin{align*} &(n_1,n_2 r_1 s_1)=1,&\quad &(n_2,n_1 r_2s_2)=1,&\quad& (n_1r_2s_2,n_2)\in\mathcal{N},&\\
 &|n_1-n_2|\ge N/(\log{x})^C,& \quad& ((n_1-n_2)/d f,a r_2 s_2)=1,&\quad &Q d f\le n_1-n_2\le Q' d f. &
\end{align*}
We first remove the dependency between $r_1$ and $s_1$ from the constraint $r_1s_1\le B_1$ by noting 
\begin{align*}
\mathbf{1}_{r_1s_1\le B_1}&=\int_0^1\Bigl(\sum_{j\le B_1/r_1}e(-ju)\Bigr)e(s_1u)du\\
&=\int_0^1 c_{r_1,u} \min\Bigl(\frac{B_1}{R},|u|^{-1}\Bigr)e(s_1u)du
\end{align*}
for some 1-bounded coefficients $c_{r_1,u}$. Thus
\begin{align*}
\mathcal{W}&=\int_0^1\min\Bigl(\frac{B_1}{R},|u|^{-1}\Bigr)\mathcal{W}_2(u)du\ll (\log{x}) \sup_{u}|\mathcal{W}_2(u)|,
\end{align*}
where $\mathcal{W}_2=\mathcal{W}_2(u)$ is given by
\begin{align*}
\mathcal{W}_2&:= \sup_{0<|a|< x^{1+\eps}}\sum_{1\le |f|\le 2N/Q}\sum_{\substack{r_1,r_2\sim R\\ (r_1r_2,a)=1}}\sum_{\substack{s_2\sim S\\ (r_2s_2,a d r_1)=1\\ r_2s_2\le B_2}}\sideset{}{'}\sum_{\substack{n_1,n_2\sim N\\ n_1\equiv n_2\Mod{d f}}}\frac{\gamma_{r_1}c_{r_1,u}\overline{\gamma_{r_2}\lambda_{s_2}}d f}{r_1r_2 s_2(n_1-n_2)}\\
&\qquad\times\sum_{\substack{s_1\sim S\\ (s_1,an_1r_2s_2)=1}}\frac{e(s_1u)\lambda_{s_1}}{s_1}\sum_{1\le |h|\le H'}\alpha_{n_1}\overline{\alpha_{n_2}}e\Bigl(\frac{a h f\overline{r_1s_1n_2}}{n_1r_2s_2}\Bigr).
\end{align*}
In order to show $\mathcal{W}\ll N^2/(Qx^\epsilon)$ we see it is sufficient to show $\mathcal{W}_2\ll N^2/(Q x^{2\epsilon})$. We now apply Cauchy-Schwarz in the $f$, $n_1$, $n_2$, $r_1$, $r_2$ and $s_2$ variables. This gives
\begin{align*}
\mathcal{W}_2\ll \frac{N R S^{1/2}(\log{x})^2}{Q R^2 S^2}\mathcal{W}_3^{1/2},
\end{align*}
where
\begin{align*}
\mathcal{W}_3&:= \sup_{0<|a|< x^{1+\eps}}\sum_{1\le |f|\le 2N/Q}\sum_{\substack{n_1,n_2\sim N\\ n_1\equiv n_2\Mod{d f}}}\sum_{r_1,r_2\sim R}\\
&\qquad\times\sum_{\substack{s_2\sim S\\ (n_2r_1,n_1r_2s_2)=1}}\Bigl|\sum_{\substack{s_1\sim S\\ (s_1,an_1r_2s_2)=1}}\sum_{1<|h|\le H'}\lambda_{s_1}' e\Bigl(\frac{ah f\overline{r_1s_1n_2}}{n_1r_2s_2}\Bigr)\Bigr|^2,
\end{align*}
and where
\[
\lambda_s':=\frac{S}{s}\lambda_{s}e(su)
\]
are 1-bounded coefficients. Note that we have dropped many of the constraints on the summation for an upper bound. In order to show that $\mathcal{W}_2\ll N^2/(Q x^{2\epsilon})$ we see it is sufficient to show that $\mathcal{W}_3\ll N^2 R^2 S^3/x^{5\epsilon}$. We first drop the congruence condition on $n_1,n_2\Mod{d f}$ for an upper bound, and then we combine $n_2r_1$ into a single variable $b$ and $n_1 r_2 s_2$ into a single variable $c$. Using the divisor bound to control the number of representations of $c$ and $b$, and inserting a smooth majorant, this gives
\begin{align*}
\mathcal{W}_3&\le x^{o(1)}\sup_{\substack{ B \ll N R\\ C\ll N R S \\ F\ll N/Q}} \mathcal{W}_4,
\end{align*}
where
\begin{align*}
\mathcal{W}_4&:= \sup_{0<|a|< x^{1+\eps}}\sum_{b}\sum_{\substack{c\\ (b,c)=1}}g(b,c)\sum_{f\sim F}\Bigl|\sum_{\substack{s_1\sim S\\ (s_1,a c)=1}}\sum_{1<|h|\le H'}\lambda_{s_1}' e\Bigl(\frac{a h f\overline{b s_1}}{c}\Bigr)\Bigr|^2\\
g(b,c)&:=\psi_0\Bigl(\frac{b}{B}\Bigr)\psi_0\Bigl(\frac{c}{C}\Bigr).
\end{align*}
In order to show $\mathcal{W}_3\ll N^2 R^2 S^3/x^{5\epsilon}$, it is sufficient to show that
\begin{equation}
\mathcal{W}_4\ll \frac{N^2R^2 S^3}{x^{6\epsilon}}.\label{eq:W4Target}
\end{equation}
We expand the square and swap the order of summation, giving
\[
\mathcal{W}_4= \sup_{0<|a|< x^{1+\eps}}\sum_{\substack{s_1,s_2\sim S\\ (s_1s_2,a)=1}}\sum_{1<|h_1|,|h_2|\le H'}\lambda_{s_1}'\overline{\lambda_{s_2}'}\sum_{b}\sum_{f\sim F}\sum_{\substack{c\\ (c,b s_1s_2)=1}}g(b,c)e\Bigl(a fk\frac{\overline{b s_1s_2}}{c}\Bigr),
\]
where
\[
k=h_1s_1-h_2s_2.
\]
We now split the sum according to whether $k=0$ or not.
\[
\mathcal{W}_4=\mathcal{W}_{k=0}+\mathcal{W}_{k\ne 0}.
\]
To show \eqref{eq:W4Target} it is sufficient to show
\begin{equation}
\mathcal{W}_{k=0}\ll \frac{N^2R^2 S^3}{x^{6\epsilon}}\qquad\text{and}\qquad \mathcal{W}_{k\ne 0}\ll \frac{N^2R^2 S^3}{x^{6\epsilon}}.\label{eq:WTargets}
\end{equation}

We first consider $\mathcal{W}_{k=0}$, and so terms with $h_1s_1=h_2s_2$. Given $h_1,s_1$ there are at most $x^{o(1)}$ choices of $h_2,s_2$, and so at most $x^{o(1)}HS$ choices of $h_1,h_2,s_1,s_2$. Thus we see that
\begin{align*}
\mathcal{W}_{k=0} \ll x^{o(1)}H S B F C&\ll x^{o(1)}\frac{R^2 S^2 Q}{M}\cdot S\cdot N R\cdot \frac{N}{Q}\cdot N R S\\
&\ll \frac{N^4 R^4 S^4}{x^{1-\epsilon}}.\label{eq:WDiag}
\end{align*}
This gives $\mathcal{W}_{k=0}\ll N^2R^2 S^3x^{-6\epsilon}$ as in \eqref{eq:WTargets}, provided
\begin{equation}
N^2 R^2 S\ll x^{1-7\epsilon}.\label{eq:FactorCond1}
\end{equation}

We now consider $\mathcal{W}_{k\ne 0}$. We let $y=fk=f(h_1 s_1-h_2s_2)\ll x^{o(1)} N R^2 S^3/M$ and $z=s_1 s_2\ll S^2$. Recall $y=fk\neq0$. Putting these variables in dyadic intervals and using the symmetry between $y$ and $-y$, we see that
\[
\mathcal{W}_{k\ne 0}\ll \log{x}\sup_{0<|a|< x^{1+\eps}}\sum_{z\sim Z}\sum_{y\sim Y}b_{z,y}\Bigl|\sum_{b}\sum_{\substack{c\\ (c,z b)=1}}g(b,c)e\Bigl(\frac{ay\overline{z b}}{c}\Bigr)\Bigr|,
\]
where $Z\asymp S^2$, $Y\ll x^{o(1)} N R^2 S^3/M$ and
\[
b_{z,y}=\mathop{\sum_{s_1,s_2\sim S}\sum_{1\le |h_1|,|h_2|\le H'}\sum_{f\sim F}}\limits_{\substack{s_1 s_2=z\\ f(h_1s_1-h_2s_2)=y}}1.
\]
By Theorem \ref{thm:DI12} with ({\it `C', `D',`N',`R',`S',`q'} ) $\rightarrow(C, B, Y, Z, 1,1)$, we have that
\begin{equation}
\mathcal{W}_{k\ne 0}\ll x^\epsilon \;\mathcal{J},
\label{eq:WOffDiag1}
\end{equation}
where
\begin{align}\label{eq:Jbound1}
\mathcal{J}^2 & \ll \sum_{\substack{n'' \mid a^{\infty}\\ n'' \leq 2N}} (an'')^{\theta} \Big(CS\Big(RS+\frac{N}{n''}\Big)( C+DR)  + aNRS\Big)\| \tilde{\textbf{b}}(n'') \|_2 \nonumber \\
& \quad + \Big( CS(CD\sqrt{R})^{2{\theta}} (N+RS)^{1-{\theta}}(C+DR)^{1-2{\theta}} + D^2 N R\Big) \| \textbf{b} \|_2^2 \nonumber\\
& \ll \sum_{\substack{n'' \mid a^{\infty}\\ n'' \leq 2Y}} (an'')^{\theta}  \Big(C\Big(Z+\frac{Y}{n''}\Big)( C+BZ)  + aYZ\Big) \| \tilde{\textbf{b}}(n'') \|_2\\
& \quad + \Big( C(CB\sqrt{Z})^{2{\theta}} (Y+Z)^{1-{\theta}}(C+BZ)^{1-2{\theta}}
+ B^2 Y Z\Big) \| \textbf{b} \|_2^2 \nonumber
\end{align}
Since the above bound on $\mathcal{J}^2$ is increasing and polynomial in $C,B,Z,Y$, the maximal value is at most $x^{o(1)}$ times the value when
$C=N R S$, $Z=S^2$, $Y=N R^2 S^3/M$ and $B=N R$, and so it suffices to consider this case. We note that our bound $M>NR^2 S $ from \eqref{eq:FactorCond1} then implies that that $Z>Y$, and so, noting that $B Z>C$ and $CBZ^2>B^2YZ$, this simplifies \eqref{eq:Jbound1} to
\begin{align}\label{eq:Jbound2}
\mathcal{J}^2
& \ll \sum_{\substack{n'' \mid a^{\infty}\\ n'' \leq 2Y}} (an'')^{\theta} \Big(CZ(BZ)  + aYZ\Big) \| \tilde{\textbf{b}}(n'') \|_2^2 \nonumber\\
& \quad + \Big( C(CB\sqrt{Z})^{2{\theta}} (Z)^{1-{\theta}}(BZ)^{1-2{\theta}}
+ B^2 Y Z\Big) \| \textbf{b} \|_2^2  \nonumber\\
& \ll  \Big(CBZ  + aY\Big) Z \sum_{\substack{n'' \mid a^{\infty}\\ n'' \leq 2Y}} (an'')^{\theta} \| \tilde{\textbf{b}}(n'') \|_2^2 
\ + \ \Big( C^{1+2{\theta}}BZ^{2-2{\theta}} + B^2 Y Z\Big) \| \textbf{b} \|_2^2.
\end{align}

Now we consider $\| \tilde{\textbf{b}}(n'') \|_2$. We note that given a choice of $z,y$ there are $x^{o(1)}$ choices of $s_1,s_2,f,k$ with $z=s_1s_2$ and $n''y= fk$ by the divisor bound. Thus by Cauchy-Schwarz,
\begin{align*}
 \| \tilde{\textbf{b}}(n'') \|_2^2 &= \sum_{z\sim Z}\sum_{y\sim Y/n''}b_{z,n''y}^2
=\sum_{z\sim Z}\sum_{y\sim Y/n''}\Bigl(\sum_{\substack{s_1,s_2\sim  S\\ s_1 s_2=z}}\sum_{\substack{f\sim F}}\sum_{\substack{ k\asymp Y/F\\ fk=n''y}}\sum_{\substack{1\le |h_1|,|h_2|\ll H \\ k=h_1 s_1-h_2s_2}}1\Bigr)^2\\
& \ll x^{o(1)}\sum_{y\sim Y/n''}
\sum_{\substack{s_1,s_2\sim  S}}\sum_{\substack{f\sim F}}\sum_{\substack{k\asymp Y/F\\ fk=n''y}}\sum_{\substack{ 1\le |h_1|,|h_1'|,|h_2|,|h_2'|\ll H \\ 
 k=h_1 s_1-h_2s_2=h_1' s_1-h_2's_2}}1\\
 & \ll x^{o(1)} \sum_{\substack{s_1,s_2\sim  S}}\sum_{k\asymp Y/F} \sum_{\substack{ 1\le |h_1|,|h_1'|,|h_2|,|h_2'|\ll H \\ 
 k=h_1 s_1-h_2s_2=h_1' s_1-h_2's_2}} \sum_{\substack{\frac{n''}{(n'',k)}\mid f\sim F}}1
\end{align*}
Note $k\asymp Y/F\asymp HS$. Then letting $c=n''/(n'',k) < 2F$, we have $\sum_{\substack{c\mid f\sim F}}1 < 4F/c$ and $n''/c\mid k$, so that
\begin{align}\label{eq:bnnorm}
 \| \tilde{\textbf{b}}(n'') \|_2^2 
&\ll x^{o(1)} \sum_{\substack{c\ll F\\ c\mid n''}} \sum_{\substack{k\asymp HS \\ n''/c \mid k}} \sum_{s_1,s_2\sim  S} 
\sum_{\substack{ 1\le |h_1|,|h_1'|,|h_2|,|h_2'|\ll H  \nonumber \\ 
 k=h_1 s_1-h_2s_2=h_1' s_1-h_2's_2}} \sum_{\substack{c\mid f\sim F}}1\\
&\ll x^{o(1)} \sum_{\substack{c\ll F\\ c\mid n''}} \sum_{\substack{k\asymp HS \\ n''/c \mid k}} \sum_{\substack{s_1,s_2\sim  S\\ (s_1,s_2)=1}} \sum_{\substack{ |h_1|,|m_1|,|h_2|,|m_2|\ll H \\ k=h_1 s_1-h_2s_2\\
t=m_1s_1=m_2s_2}}\frac{F}{c}.
\end{align}
Here we set $m_1 = h_1-h_1'$ and $m_2=h_2-h_2'$, and noted $k=h_1 s_1-h_2s_2=h_1' s_1-h_2's_2$ implies that $t:=(h_1-h_1')s_1=(h_2-h_2')s_2$. We now handle the inner sum in \eqref{eq:bnnorm}:

If $t=0$, this forces $m_1=m_2=0$. Then there are $O(HS)$ choices $h_1,s_1$ and given $k$ by the divisor bound, $x^{o(1)}$ further choices of $h_2,s_2$ such that $k=h_1 s_1-h_2s_2$. Thus
\begin{align}\label{eq:bnnormteq0}
 \| \underset{[t=0]}{\tilde{\textbf{b}}(n'') \|_2^2}
 &\ll x^{o(1)} \sum_{\substack{c\ll F\\ c\mid n''}} \sum_{\substack{k\asymp HS \\ n''/c \mid k}} \sum_{\substack{s_1,s_2\sim  S}} \sum_{\substack{|h_1|,|h_2|\ll H \\ k=h_1 s_1-h_2s_2\\ t=0}} \frac{F}{c} \nonumber\\
&\ll x^{o(1)} \sum_{\substack{c\ll F\\ c\mid n''}} \sum_{\substack{k\asymp HS \\ n''/c \mid k}}(HS)\frac{F}{c}
\ll x^{o(1)} \sum_{\substack{n''/HS \ll c\ll F\\ c\mid n''}} \Big(\frac{HS}{n''}+\frac{1}{c}\Big)(HS)F \ll
x^{o(1)} (HS)^2\frac{F}{n''}.
\end{align}
Otherwise suppose $t\ne 0$. For this, we factor out $d=(s_1,s_2)\ll S$ and $g=(m_1,m_2)$. There are $O(HS/d)$ choices of $h_1,s_1$, and then given $k$ there are $x^{o(1)}$ choices of $h_2,s_2$ such that $k=h_1 s_1-h_2s_2$ by the divisor bound. Further, nonzero $t=m_1s_1=m_2s_2$ forces $m_2=s_1$ and $s_2=m_1$ by coprimality. In particular, this forces $H/g\asymp S/d$. Thus 
\begin{align*}
\| \underset{[t\neq 0]}{\tilde{\textbf{b}}(n'') \|_2^2}
&\ll x^{o(1)} \sum_{\substack{c\ll F\\ c\mid n''}} \sum_{d\ll S} \sum_{g\asymp dH/S}\sum_{\substack{k\asymp HS/d \\ n''/c \mid kd}} \sum_{\substack{s_1,s_2\sim  S/d\\(s_1,s_2)=1}}\sum_{\substack{|m_1|,|m_2|\ll H/g\\ (m_1,m_2)=1}}
\sum_{\substack{|h_1|,|h_2|\ll H\\ k=h_1 s_1-h_2s_2\\ t=m_1s_1=m_2s_2\neq0}} \frac{F}{c} \\
&\ll x^{o(1)} \sum_{\substack{c\ll F\\ c\mid n''}} \sum_{d\ll S} \sum_{g\asymp dH/S}\sum_{\substack{k\asymp HS/d \\ n''/bc \mid k}} 
\frac{HS}{d}\frac{F}{c},
\end{align*}
noting $n''/c\mid kd$ implies $n''/bc\mid k$, where $b := (d,n''/c)$. In particular $n''/bc\ll HS/d$, so
\begin{align}\label{eq:bnnormtnot0}
\| \underset{[t\neq 0]}{\tilde{\textbf{b}}(n'') \|_2^2}
&\ll x^{o(1)} \sum_{d\ll S} \sum_{g\asymp dH/S}\sum_{\substack{n''d/bHS \ll c\ll F\\ c\mid n''}} \Big(\frac{HS/d}{n''/bc}+1\Big)\frac{HS}{d}\frac{F}{c} \nonumber\\
&\ll x^{o(1)} \sum_{d\ll S} \sum_{g\asymp dH/S}\frac{(d,n'')}{d^2n''}(HS)^2F \nonumber\\
&\ll x^{o(1)} H^3S\frac{F}{n''}\sum_{d\ll S}\frac{(d,n'')}{d}
\ll x^{o(1)} H^3S\frac{F}{n''}\sum_{b\mid n''}\sum_{d'\ll S/b}\frac{1}{d'}
\ \ll \ x^{o(1)} H^3S\frac{F}{n''}
\end{align}
by the divisor bound. Note $S>NR^2 S^2/M>H$, since $M>NR^2 S $ by \eqref{eq:FactorCond1}. Thus $(HS)^2>H^3S$, and so \eqref{eq:bnnormteq0} and \eqref{eq:bnnormtnot0} give
\begin{align}
\| \tilde{\textbf{b}}(n'') \|_2^2 
&\ll \| \underset{[t= 0]}{\tilde{\textbf{b}}(n'') \|_2^2} + \| \underset{[t\neq 0]}{\tilde{\textbf{b}}(n'') \|_2^2} \ \ll \
x^{o(1)} (HS)^2\frac{F}{n''}.
\end{align}
In particular, for $n''=1$ we have $\| \textbf{b} \|_2^2 = \| \tilde{\textbf{b}}(1) \|_2^2 \ll x^{o(1)}(HS)^2F$.

Plugging back into \eqref{eq:Jbound2}, we obtain
\begin{align*}
\mathcal{J}^2\, x^{-o(1)}
& \ll a^{\theta}\Big(CBZ  + aY\Big) Z\sum_{\substack{n'' \mid a^{\infty}\\ n'' \leq 2Y}} (n'')^{\theta} (HS)^2\frac{F}{n''} \ + \ \Big( C^{1+2{\theta}}BZ^{2-2{\theta}} + B^2 Y Z\Big) (HS)^2F \\
& \ll  \bigg[a^{\theta} \big(CBZ  + x^{1+\eps}Y\big)+ \Big( C^{1+2{\theta}}BZ^{1-2{\theta}} + B^2 Y \Big)\bigg]Z(HS)^2F
\end{align*}
using $aY\ll x^{1+\eps}Y$, and $\sum_{n''} (n'')^{\theta-1} \ll x^\eps$.

Next, observe $CBZ < xY$: Indeed, to see this, we have $CBZ = (NRS)(NR)(S^2) = N^2R^2S^3$ and $xY = x(NR^2S^3/M) = xNR^2S^5/M \asymp N^2R^2S^5$. By \eqref{eq:FactorCond1}, $B^2 = N^2 R^2 < x$ and so $B^2Y < xY$. Thus we obtain
\begin{align}
\mathcal{J}^2 & \ll x^{\eps}\bigg[a^{\theta} xY + CBZ(C/Z)^{2{\theta}}\bigg]Z(HS)^2F.
\end{align}

Thus $\mathcal{W}_{k\neq0}^2 < x^\epsilon \mathcal{J}^2$ is less than $(N^2R^2S^3x^{-6\epsilon})^2$, provided
\begin{align*}
\bigg[a^{\theta} x Y + CBZ(C/Z)^{2{\theta}}\bigg]
< \Big(\frac{N^2R^2S^3}{x^{7\epsilon}HS}\Big)^2/ZF
& = \Big(\frac{N^2R^2S^2}{x^{7\epsilon}QR^2S^2/M}\Big)^2/(S^2 N/Q) \\ 
& = \Big(\frac{N^2M}{x^{7\epsilon}QS}\Big)^2\frac{Q}{N} \asymp \frac{Nx^{2-14\epsilon}}{QS^2}.
\end{align*}
This is equivalent to the inequalities 
\begin{align*}
\frac{Nx^{2-14\epsilon}}{QS^2} &> a^{\theta} x Y = a^{\theta} x \frac{N^2R^2S^3}{x}\\
\frac{Nx^{2-14\epsilon}}{QS^2} &> CBZ(C/Z)^{2{\theta}} = (NRS)(NR)(S^2)\Big(\frac{NRS}{S^2}\Big)^{2{\theta}}
= N^2R^2S^3\Big(\frac{NR}{S}\Big)^{2{\theta}}
\end{align*}
Simplifying, we conclude $\mathcal{W}_{k\neq0} < N^2R^2S^3x^{-6\epsilon}$ holds provided
\begin{align}
a^{\theta} NR^2S^5Q & < x^{2-14\epsilon},\\
(NR/S)^{2{\theta}} NR^2S^5Q  & <  x^{2-14\epsilon}.
\end{align}
\end{proof}
%
%
%
%
%
%
%
%
\begin{proposition}[Well-factorable estimate for convolutions]\label{proposition:MainProp}
Let $NM\asymp x$, and $Q_1,Q_2,Q_3$ satisfy
\begin{align*}
Q_1&<\frac{N}{x^\epsilon},\\
N^2 Q_2 Q_3^2&<x^{1-8\epsilon},\\
a^{\theta}\; N Q_1 Q_2^5 Q_3^2&<x^{2-15\epsilon},\\
(NQ_3/Q_2)^{2{\theta}}\;N Q_1 Q_2^5 Q_3^2 &<x^{2-15\epsilon}. 
\end{align*}
Let $\alpha_n,\beta_m$ be $1$-bounded complex sequences such that $\alpha_n$ satisfies the Siegel-Walfisz condition \eqref{eq:SiegelWalfisz} and $\alpha_n$ is supported on $n$ with all prime factors bigger than $z_0=x^{1/(\log\log{x})^3}$. Let $\gamma_{q_1},\lambda_{q_2},\nu_{q_3}$ be 1-bounded complex coefficients supported on $(q_i,a)=1$ for $i\in\{1,2,3\}$. Let
\[
\Delta(q):=\sum_{n\sim N}\alpha_n\sum_{m\sim M}\beta_m\Bigl(\mathbf{1}_{nm\equiv a\Mod{q}}-\frac{\mathbf{1}_{(n m,q)=1}}{\phi(q)}\Bigr).
\]
Then for every $A>0$ we have
\[
\sup_{0<|a|< x^{1+\eps}}\sum_{q_1\sim Q_1}\sum_{q_2\sim Q_2}\sum_{q_3\sim Q_3}\gamma_{q_1}\lambda_{q_2}\nu_{q_3}\Delta(q_1q_2q_3)\ll_A\frac{x}{(\log{x})^A}.
\]
\end{proposition}
\begin{proof}
First we factor $q_2=q_2'q_2''$ and $q_3=q_3'q_3''$ where $P^-(q_2'),P^-(q_3')>z_0\ge P^+(q_2''),P^+(q_3'')$ into parts with large and small prime factors. By putting these in dyadic intervals, we see that it suffices to show for every $A>0$ and every choice of $Q_2'Q_2''\asymp Q_2$, $Q_3'Q_3''\asymp Q_3$ that
 \begin{align*}
 &\sup_{0<|a|< x^{1+\eps}}\sum_{q_1\sim Q_1}\sum_{\substack{q_2'\sim Q_2'\\ P^-(q_2')>z_0}}\sum_{\substack{q_2''\sim Q_2''\\ P^+(q_2'')\le z_0}}\sum_{\substack{q_3'\sim Q_3'\\ P^-(q_3')\ge z_0}}\sum_{\substack{q_3''\sim Q_3''\\ P^+(q_3'')\le z_0}}\gamma_{q_1}\lambda_{q_2'q_2''}\nu_{q_3'q_3''}\Delta(q_1q_2'q_2''q_3'q_3'')\ll_A\frac{x}{(\log{x})^A}.
 \end{align*}
By Lemma \ref{lemma:RoughModuli} we have the result unless $Q_2'',Q_3''\le y_0=x^{1/\log\log{x}}$. We let $d=q_2'' q_3''$ and define
\[
\lambda_{q,d,r}:=\mathbf{1}_{P^-(r)> z_0}\sum_{\substack{q_2''q_3''=d\\ q_1\sim Q_1\\ q_2''\sim Q_2''\\ q_3''\sim Q_3''\\ P^+(q_2''q_3'')\le z_0}}\,\,\sum_{\substack{q_2'q_3'=r\\ q_2'\sim Q_2'\\ q_3'\sim Q_3'}}\lambda_{q_2'q_2''}\nu_{q_3'q_3''}.
\]
We note that $\lambda_{q,d,r}$ doesn't depend on $q$. With this definition we see it suffices to show that for every $A>0$ and every choice of $D,R$ with $D R\asymp Q_2 Q_3$ and $D\le y_0^2$ we have that
\[
\sup_{0<|a|< x^{1+\eps}}\sum_{q\sim Q_1}\sum_{d\sim D}\sum_{r\sim R}\gamma_{q}\lambda_{q,d,r}\Delta(q d r)\ll_A \frac{x}{(\log{x})^A}.
\]
We now apply Proposition \ref{proposition:GeneralDispersion} (we may apply this since $N>Q_1 x^\epsilon>Q_1D(\log{x})^C$ and $N<x^{1-\epsilon}$ by assumption of the lemma). This shows that it suffices to show that
\[
|\mathcal{E}_1|+|\mathcal{E}_2|\ll \frac{N^2}{D Q_1 y_0},
\]
where
\begin{align*}
\mathcal{E}_{1}&:=\sup_{0<|a|< x^{1+\eps}}\sum_{\substack{q\\ (q,a)=1}}\sum_{\substack{d\sim D\\ (d,a)=1}}\sum_{\substack{r_1,r_2\sim R\\ (r_1r_2,a)=1}}\psi_0\Bigl(\frac{q}{Q_1}\Bigr)\frac{\lambda_{q,d,r_1}\overline{\lambda_{q,d,r_2}} }{\phi(q d r_2)q d r_1}\sum_{\substack{n_1,n_2\sim N\\ (n_1,q d r_1)=1\\(n_2,q d  r_2)=1}}\alpha_{n_1}\overline{\alpha_{n_2}}\\
&\qquad \times\sum_{1\le |h|\le H_1}\hat{\psi}_0\Bigl(\frac{h M}{q d r_1}\Bigr)e\Bigl( \frac{a h \overline{ n_1}}{q d r_1}\Bigr),\\
\mathcal{E}_2&:=\sup_{0<|a|< x^{1+\eps}}\sum_{\substack{q\\ (q,a)=1}}\psi_0\Bigl(\frac{q}{Q_1}\Bigr)\sum_{\substack{d\sim D\\ (d,a)=1}}\sum_{\substack{r_1,r_2\sim R\\ (r_1,a r_2)=1\\ (r_2,a q d r_1)=1}}\frac{\lambda_{q,d,r_1}\overline{\lambda_{q,d,r_2}}}{q d r_1 r_2}\sum_{\substack{n_1,n_2\sim N\\ n_1\equiv n_2\Mod{q d}\\ (n_1,n_2 q d r_1)=1\\(n_2,n_1 q d r_2)=1\\ |n_1-n_2|\ge N/(\log{x})^C}}\alpha_{n_1}\overline{\alpha_{n_2}}\\
&\qquad \times\sum_{1\le |h|\le H_2}\hat{\psi}_0\Bigl(\frac{h M}{q d r_1 r_2}\Bigr)e\Bigl(\frac{ah\overline{n_1r_2}}{q d r_1}+\frac{ah\overline{n_2 q d r_1}}{r_2}\Bigr),\\
H_1&:=\frac{Q D R}{M}\log^5{x},\\
H_2&:=\frac{Q D R^2}{M}\log^5{x}.
\end{align*}
Since $\lambda_{q,d,r}$ is independent of $q$, we may apply Lemma \ref{lemma:BFI2} to conclude that
\[
\mathcal{E}_1\ll \frac{N^2}{Q_1 D x^\epsilon},
\]
provided we have
\begin{align}
D R N^{3/2}&<x^{1-2\epsilon},\\
Q_1 D R&<x^{1-2\epsilon}.
\end{align}
These are both implied by the conditions of the lemma, recalling that $DR\asymp Q_2Q_3$. 

Thus it remains to bound $\mathcal{E}_2$. Since $D\le y_0^2=x^{o(1)}$, it suffices to show 
\[
\mathcal{E}_3\ll \frac{N^2}{Q_1 x^{\epsilon/10}},
\]
for each $d\le y_0^2$, where $\mathcal{E}_3=\mathcal{E}_3(d)$ is given by
\begin{align*}
\mathcal{E}_3&:= \sup_{0<|a|< x^{1+\eps}}\sum_{\substack{(q,a)=1}}\psi_0\Bigl(\frac{q}{Q_1}\Bigr)\sum_{\substack{r_1,r_2\sim R\\ (r_1,a r_2)=1\\ (r_2,a q d r_1)=1}}\frac{\lambda_{q,d,r_1}\overline{\lambda_{q,d,r_2}}}{q r_1 r_2}\sum_{\substack{n_1,n_2\sim N\\ n_1\equiv n_2\Mod{q d}\\ (n_1,n_2 q d r_1)=1\\(n_2,n_1 q d r_2)=1\\ |n_1-n_2|\ge N/(\log{x})^C}}\alpha_{n_1}\overline{\alpha_{n_2}}\\
&\qquad\qquad \times\sum_{1\le |h|\le H_2}\hat{\psi}_0\Bigl(\frac{h M}{q d r_1 r_2}\Bigr)e\Bigl(\frac{ah\overline{n_1r_2}}{q d r_1}+\frac{ah\overline{n_2 q d r_1}}{r_2}\Bigr).
\end{align*}
Since $\lambda_{q,d,r}$ is independent of $q$ and we treat each $d$ separately, we may suppress the $q,d$ dependence by writing $\lambda_{r}$ in place of $\lambda_{q,d,r}$. We now apply Lemma \ref{lemma:Simplification}. This shows it suffices to show that
\[
\sup_{\substack{0<|a|< x^{1+\eps}\\R_1,R_2\le 2R}} \sup_{\substack{H'\le H\\ Q'\le 2Q}}\mathcal{E}'\ll \frac{N^2}{Q_1 x^{\epsilon/2}},
\]
where
\[
\mathcal{E}':= \sum_{\substack{Q_1\le q\le Q_1'\\ (q,a)=1}}\sum_{\substack{R\le r_1\le  R_1\\ R\le r_2\le R_2\\ (r_1, a r_2)=1\\ (r_2,a q d r_1)=1}}\frac{\lambda_{r_1}\overline{\lambda_{r_2}}}{q d  r_1 r_2}\sum_{\substack{n_1,n_2\sim N\\ n_1\equiv n_2\Mod{q d}\\ (n_1,q d r_1n_2)=1\\ (n_2,q d r_2n_1)=1\\ (n_1r_2,n_2)\in\mathcal{N}}}\alpha_{n_1}\overline{\alpha_{n_2}}\sum_{1\le |h| \le H'} e\Bigl(\frac{ ah\overline{n_2 q d r_1}(n_1-n_2)}{n_1 r_2}\Bigr).
\]

We recall the definition of $\lambda_{q,d,r}$ and expand it as a sum. Since $d$ is fixed, there are $x^{o(1)}$ possible choices of $q_2'',q_3''$. Fixing one such choice, we then see $\mathcal{E}'$ is precisely of the form considered in Lemma \ref{lemma:Factorable}. This then gives the result, provided
\begin{align*}
Q_1&<\frac{N}{x^\epsilon},\\
N^2 Q_2' Q_3'{}^2 &<x^{1-7\epsilon},\\
a^{\theta} \; N Q_1 Q_2'{}^5 Q_3'{}^2&<x^{2-14\epsilon},\\
(NQ_3'/Q_2')^{2{\theta}}\;N Q_1 Q_2'{}^5 Q_3'{}^2&<x^{2-14\epsilon}.
\end{align*}
Since $Q_2'\le Q_2$ and $Q_3'\le Q_3$, these bounds follow from the assumptions of the lemma.
\end{proof}
%
%
%
%
%
%
%
%
\begin{proof}[Proof of Proposition \ref{proposition:WellFactorable}]
First we note that by Lemma \ref{lemma:Divisor} the set of $n,m$ with
$\max(|\alpha_n|,|\beta_m|)$ $\ge(\log{x})^C$ and $n m\equiv a\Mod{q}$ has size $\ll x(\log{x})^{O_{B_0}(1)-C}/q$, so these terms contribute negligibly if $C=C(A,B_0)$ is large enough. Thus, by dividing through by $(\log{x})^{2C}$ and considering $A+2C$ in place of $A$, it suffices to show the result when all the sequences are 1-bounded. ($\alpha_n$ still satisfies \eqref{eq:SiegelWalfisz} by Lemma \ref{lemma:SiegelWalfiszMaintain}.) The result follows from the Bombieri-Vinogradov Theorem if $Q\le x^{1/2-\epsilon}$, so we may assume that $Q\in[x^{1/2-\epsilon},x^{5/8-10\epsilon}]$.

We use Lemma \ref{lemma:Separation} to remove the condition $n m\in\mathcal{I}$, and see suffices to show for $B=B(A)$ sufficiently large in terms of $A$
\begin{equation}
\sup_{0<|a|< x^{1+\eps}} \sum_{q\le Q}\lambda_q\sum_{n\in\mathcal{I}_N}\alpha_n\sum_{\substack{m\in \mathcal{I}_M}}\beta_m\Bigl(\mathbf{1}_{n m\equiv a\Mod{q}}-\frac{\mathbf{1}_{(nm,q)=1}}{\phi(q)}\Bigr)\ll_B\frac{x}{(\log{x})^B}
\label{eq:FactorableTarget}
\end{equation}
uniformly over all intervals $\mathcal{I}_N\subseteq[N,2N]$ and $\mathcal{I}_M\subseteq[M,2M]$.

Write $Q=x^{\pmb \vartheta}$ and take $v\in [.35,.50]$ to be determined. Let us define for $x^\epsilon\le N\le x^{1/3+\epsilon}$
\[
Q_1:=\frac{N}{x^\epsilon},\qquad Q_2:=x^{{\pmb \vartheta} -v+\epsilon},\qquad Q_3:=\frac{x^{v}}{N}.
\]
We note that $Q_1Q_2Q_3=x^{\pmb \vartheta}$ and $Q_1,Q_2,Q_3\ge 1$. Since $\lambda_q$ is triply well-factorable of level $Q=x^{\pmb \vartheta}$, we can write
\begin{equation}
\lambda_{q}=\sum_{q_1q_2q_3=q}\gamma^{(1)}_{q_1}\gamma^{(2)}_{q_2}\gamma^{(3)}_{q_3},
\label{eq:LambdaFact}
\end{equation}
for some 1-bounded sequences $\gamma^{(1)},\gamma^{(2)},\gamma^{(3)}$ with $\gamma^{(i)}_q$ supported on $q\le Q_i$ for $i\in\{1,2,3\}$.

We now  substitute \eqref{eq:LambdaFact} into \eqref{eq:FactorableTarget} and put each of $q_1,q_2,q_3$ into one of $O(\log^3{x})$ dyadic intervals $(Q_1',2Q_1']$, $(Q_2',2Q_2']$ and $(Q_3',2Q_3']$ respectively. Since $Q_1'\le Q_1$, $Q_2'\le Q_2$ and $Q_3'\le Q_3$ and $Q_1Q_2Q_3=Q=x^{\pmb \vartheta}$ we have
\begin{align*}
Q_1'&\le\frac{N}{x^\epsilon},\\
N^2 Q_2' Q_3'{}^2 &\le \frac{Q^2}{Q_2}x^{\epsilon} \le Qx^{v+\epsilon} <x^{1-7\epsilon},\\
a^{\theta}\; N Q_1' Q_2'{}^5 Q_3'{}^2
&\le a^{\theta}\;N^2 Q_2^5 Q_3^2 
\le a^{\theta}\; Q^2 Q_2^3\\
& \ =
\frac{x^{\alpha\theta}\; Q^5}{x^{3v-3\epsilon}}<x^{2-14\epsilon}
\end{align*}
since $|a|^\theta \ll x^{\alpha\theta}$, as well as,
\begin{align*}
(NQ_3'/Q_2')^{2{\theta}}\;N Q_1' Q_2'{}^5 Q_3'{}^2 & \le (x^{2v}/Q)^{2{\theta}}\;\frac{Q^5}{x^{3v-3\epsilon}}\\
& \ = Q^{5-2{\theta}}x^{v(4{\theta}-3)+3\epsilon}<x^{2-14\epsilon}
\end{align*}
provided
\begin{align*}
{\pmb \vartheta}+v &< 1-8\epsilon,\\
\alpha\theta + 5{\pmb \vartheta}-3v &<2-17\epsilon\\
{\pmb \vartheta}(5-2{\theta})+v(4{\theta}-3) &< 2-17\epsilon 
\end{align*}
Choosing $v = 1-{\pmb \vartheta}-9\epsilon$, this simplifies as
\begin{align*}
\alpha\theta + 8{\pmb \vartheta} &<5-44\epsilon,\\
{\pmb \vartheta}(8-6{\theta}) &< 5-4{\theta}-44\epsilon 
\end{align*}
Thus for ${\pmb \vartheta}= \min(\frac{5-4{\theta}}{8-6{\theta}},\, \frac{5-\alpha\theta}{8})$,
we see that Proposition \ref{proposition:MainProp} now gives the result.

\noindent
Note such ${\pmb \vartheta}\in [.57,.625]$ by ${\theta} \le \frac{7}{32}$, in particular $v = 1-{\pmb \vartheta}-9\epsilon \in [.35,.5]$ as desired.
\end{proof}
We have now established both Proposition \ref{proposition:DoubleDivisor} and Proposition \ref{proposition:WellFactorable}, thereby completing the proof of Theorem \ref{theorem:Factorable}.
%
%
%
%
%
%
%
%
\section{Factorization of well-factorable support}
For ${\pmb \vartheta}\in[\tfrac{1}{2},1]$, let $D = x^{{\pmb \vartheta}-2\delta}$
In this section, we establish a suitable factorization of the integers $d$ in the well-factorable support $\mathcal D^{\textnormal{well}}$,
\begin{align}\label{eq:Dwelldef}
\mathcal D^{\textnormal{well}}(D) \ = \ \big\{ \, d=p_1\cdots p_r \ : \  p_1\cdots p_{m-1}p_m^2 < D \quad {\rm for \ all} \ \ m\le r\big\}.
\end{align}
This will serve as the key technical input for the proof of Theorems \ref{thm:twinbound} and \ref{thm:Goldbachbound}, since $\mathcal D^{\textnormal{well}}$ contains the support for the (upper and lower bound) linear sieve weights.

Let us define
\begin{align}\label{eq:thetat}
{\pmb \vartheta}(t) = {\pmb \vartheta}_\alpha(t) &:= \min\Big(\frac{1+t}{2}, \ 
\frac{1-(3/2-2{\theta})t}{1+{\theta}}, \ 
1 - \frac{\alpha\theta + 3t}{2}\Big).
\end{align}
Observe ${\pmb \vartheta}(t)$ is a unimodal function of $t$, with maximum at the {\it balance point} $\mu$, given by $\mu = \min(\frac{1-{\theta}}{4-3{\theta}},\, \frac{1-\alpha\theta}{4})$, so we may express ${\pmb \vartheta}(t)$ as
\begin{align*}
{\pmb \vartheta}(t) &=
\begin{cases}
\min\big(
\frac{1-(3/2-2{\theta})t}{1+{\theta}}, \ 
1 - \frac{\alpha\theta + 3t}{2}\big) & \text{if} \ \ t > \mu,\\
\frac{1+t}{2} & \text{if} \ \ t \le \, \mu.
\end{cases}
\end{align*}
\begin{proposition}\label{prop:piecewiseeta}
Let $0<\delta<10^{-5}$ and $A\in[x^{\delta},x^{1/3-\delta/2}]$.
Let $p_1\ge\cdots\ge p_r$ be primes and write $p_i=x^{t_i}$. If $d=p_1\cdots p_r\in\mathcal D^{\textnormal{well}}(x^{{\pmb \vartheta}-2\delta})$, then there is a factorization $d=abc$ such that $a\le A$ and
\begin{equation}\label{eq:JMfactorable}
\begin{split}
A^2 \ bc^2 & \ \le \ x^{1-3\delta},\\
x^{\alpha\theta} A a b^5 c^2 & \ \le \ x^{2-3\delta},\\
(Ac/b)^{2{\theta}} \,A a b^5 c^2 & \ \le \ x^{2-2\delta},
\end{split}
\end{equation} 
provided
\begin{align*}
{\pmb \vartheta} \ \le \ {\pmb \vartheta}(t_1),
\end{align*}
for ${\pmb \vartheta}(t)$ as in \eqref{eq:thetat}. Moreover if $t_1 \le \,\mu := \min(\frac{1-{\theta}}{4-3{\theta}},\, \frac{1-\alpha\theta}{4})$, then it suffices that
\begin{align}\label{eq:thetat123}
{\pmb \vartheta} \ \le \ {\pmb \vartheta}(t_1,t_2,t_3) := \max\Big\{ {\pmb \vartheta}(t_1)&,\,{\pmb \vartheta}(t_2),\,{\pmb \vartheta}(t_1+t_2),\,{\pmb \vartheta}(t_1+t_2+t_3),\,w(t_1,t_2,t_3),\,w(t_1,t_3,t_2) \nonumber\\
& \ \ 
\psi\big({\pmb \vartheta}(t_1+t_3),\,t_1+2t_2+t_3\big), \ \psi\big({\pmb \vartheta}(t_2+t_3),\,2t_1+t_2+t_3\big)\Big\},
\end{align}
where $\psi(x,y) := x\cdot\1_{x\ge y}$ and
\begin{align*}
w(t_1,t_2,t_3) &= \psi\Big(\min\Big\{\;\frac{(5-4{\theta})-(3-4{\theta})t_3}{8-6{\theta}},\; \frac{5-{\alpha\theta}-3t_3}{8},\;1-2t_2\Big\},\,\frac{1+t_1}{2}\Big).
\end{align*}
\end{proposition}

\begin{remark}
Using Selberg's bound ${\theta}=\frac{1}{2}$, 
the result recovers Corollary 3.8 from \cite{Ltwin}.
\end{remark}

For notational ease we are using $A = Nx^{-\delta}$. Also, on the first attempt working through technicalities, we encourage the reader to set $\delta=0$ in order to better view the key features.

Before proving the proposition, we need some lemmas. The first gives a general-purpose criterion to factor an integer $d$.

\begin{lemma}\label{lem:criterion}
Let $D = x^{{\pmb \vartheta}-2\delta}$ for ${\pmb \vartheta} <\frac{5}{8}$. If $d=abc$ factors for some $a,b,c\ge1$, with
\begin{align}\label{eq:criterion}
a\le A, \quad c\le D/Ab, \quad b\in [x^v,x^u],
\end{align}
then factorization $d=abc$ satisfies \eqref{eq:JMfactorable}.
Here the \emph{critical interval} $[x^v,x^u]$ is defined by
\begin{align}\label{def:critint}
v := 2{\pmb \vartheta}-1 
\qquad {\rm and}\qquad
u := \min\Big(\frac{1-(1+{\theta}){\pmb \vartheta}}{3/2-2{\theta}},\, \frac{2 -2 {\pmb \vartheta}-{\alpha\theta}}{3} \Big).
\end{align}
\end{lemma}
\begin{proof}
Using $Ac\le D/b$ we have
\begin{align*}
A^2 \ b c^2 & \ \le \ D^2/b \le x^{2({\pmb \vartheta}-2\delta)-(2{\pmb \vartheta}-1)} = x^{1-4\delta}.
\end{align*}
as well as
\begin{align*}
(Ac/b)^{2{\theta}} A a b^5 c^2  \le  (D/b^2)^{2{\theta}}D^2b^3
&= D^{2+2{\theta}}b^{3-4{\theta}}\\
& \le\, x^{2-4\delta} \qquad\qquad \text{provided}\qquad 
b^{3-4{\theta}} \le x^{2-(2+2{\theta}){\pmb \vartheta}}.\\
x^{{\alpha\theta}} A a b^5 c^2 & \le  x^{{\alpha\theta}}D^2b^3\\
& \le\, x^{2-4\delta} \qquad\qquad \text{provided}\qquad
b^3 \le\, x^{2-2{\pmb \vartheta}-{\alpha\theta}}.
\end{align*}
That is, $b\le\, x^u$ for $u=\min\big(\frac{1-(1+{\theta}){\pmb \vartheta}}{3/2-2{\theta}},\, \frac{2 -2 {\pmb \vartheta}-{\alpha\theta}}{3}\big)$. 

Hence for $b\in [x^v,x^u]$ the factorization $d=abc$ satisfies \eqref{eq:JMfactorable}.
\end{proof}

Next, if the primes dividing $d$ are small enough, we may use the greedy algorithm to factor $d$ as follows.

\begin{lemma}\label{lem:greedy}
Let $D = x^{{\pmb \vartheta}-2\delta}$ for ${\pmb \vartheta}<\frac{5-{\alpha\theta}}{8}$, and take $v=2{\pmb \vartheta}-1$ and $u$ as in \eqref{def:critint}. For $r\ge3$, let $p_1\ge\cdots\ge p_r$ be primes for which $d=p_1\cdots p_r\in \mathcal D^{\textnormal{well}}(D)$, and $p_3 < x^{u-v}$ and ($p_2^2 < x^{1-{\pmb \vartheta}}$, $p_1 < x^v$) or ($p_1^2 < x^{1-{\pmb \vartheta}}$, $p_2 < x^v$). Then $d$ has a factorization $d=abc$ satisfying \eqref{eq:JMfactorable}.
\end{lemma}
\begin{proof}
Let $(D_1,D_2,D_3)=(A,\,D^2/x,\,x^{1-2\delta}/(DA))$. By assumption $p_1\le D_2$ and $p_2^2\le D_1D_3$, or $p_2\le D_2$ and $p_1^2\le D_1D_3$. Thus for some choice $\{d_1,d_2,d_3\}=\{p_1,p_2,p_3\}$ we have $d_i\le D_i$ for all $i$.

We now greedily append primes to $d_i$ while preserving $d_i\le D_i$ for all $i$, i.e. where at the $j$th step we replace $d_i\mapsto d_ip_j$ (for one of $i=1,2,3$) provided $d_ip_j\le D_i$. So starting from $j=3$, we stop either when we have exhausted all primes (i.e. $j=r$), or $d_ip_j> D_i$ for each $i=1,2,3$. 

In the former case, $d_1d_2d_3 = d= p_1\cdots p_r$ and $d_i\le D_i$ so we easily get $d_1 \le D_1 = A$ and
\begin{align*}
A^2 D_2D_3^2 & = x^{1-4\delta},
\end{align*}
as well as
\begin{align*}
(Ad_3/d_2)^{2{\theta}} Ad_1d_2^5 d_3^2 &= A^{2{\theta}} d_1 d_2^{5-2{\theta}}d_3^{2+2{\theta}}
\le A^{2+2{\theta}} D_2^{5-2{\theta}} D_3^{2+2{\theta}}\\
&=  A^{2+2{\theta}} (D^2/x)^{5-2{\theta}} (x^{1-2\delta}/DA)^{2+2{\theta}}\\
&= D^8 x^{-3-2\delta} 
= x^{8{\pmb \vartheta}-3-2\delta}
< x^{2-2\delta}
\end{align*}
using ${\pmb \vartheta}<\frac{5}{8}$. Similarly, we have
\begin{align*}
x^{\alpha\theta} Ad_1d_2^5 d_3^2 &= 
x^{\alpha\theta} A^2 (D^2/x)^5 (x^{1-2\delta}/DA)^2\\
& = x^{{\alpha\theta}-3-4\delta}D^8
= x^{{\alpha\theta}-3-4\delta + 8{\pmb \vartheta}} 
 < x^{2-2\delta}
\end{align*}
using ${\pmb \vartheta}<\frac{5-{\alpha\theta}}{8}$. Thus letting $a=d_1, b=d_2, c=d_3$ gives the desired factorisation for \eqref{eq:JMfactorable}.

In the latter case, there exists a terminal index $j<r$ for which $d_ip_j> D_i$ for all $i=1,2,3$. By assumption $j\ge3$, and so $p_j\le p_3\le x^{u-v}$ is smaller than the width of the interval $[x^{v}, x^{u}]$. And since $d_2\le D_2=x^{v}<d_2p_j$,
we deduce $b := d_2p_j$ lies in the interval $b \in [x^{v}, x^{u}]$.
Thus letting $C := D_2D_3/b$, for each $l>j$ we have
\begin{align*}
p_l^2 \le \frac{D}{p_1\cdots p_{l-1}} = \frac{D_1D_2D_3}{d_1d_2d_3p_j\cdots p_{l-1}} = \frac{D_1C}{d_1d_3p_{j+1}\cdots p_{l-1}}.
\end{align*}
So there is a factorization $ac = d_1d_3p_{j+1}\cdots p_l$ with $a\le D_1=A$ and $c\le C= D/(Abx^{2\delta})$. When $l=r$, recalling $b=d_2p_j$, we deduce a factorization
\begin{align*}
abc = d_1d_2d_3p_{j}\cdots p_r = p_1\cdots p_r = d.
\end{align*}
Hence Lemma \ref{lem:criterion} implies that $d=abc$ satisfies \eqref{eq:JMfactorable}.
\end{proof}

In the following result, we factorize $d\in\mathcal D^{\textnormal{well}}(D)$ for variable level $D$, depending on the anatomy of $d$. As $\mathcal D^\pm \subset \mathcal D^{\textnormal{well}}$, this has implications to both upper and lower bounds for the standard linear sieve.

\begin{proposition}\label{prop:Dwell}
Let $D=x^{{\pmb \vartheta}-2\delta}$ with ${\pmb \vartheta}<\frac{5}{8}$, and take $v$, $u$ as in \eqref{def:critint}.
Let $p_r\le\cdots\le p_1\le x^u$ be primes for which $d=p_1\cdots p_r\in \mathcal D^{\textnormal{well}}(D)$.
Then $d$ has factorization $d=abc$ satisfying \eqref{eq:JMfactorable},
provided one of the following holds:
\begin{itemize}
\item[(i)] $b\in [x^v, \, x^u]$ for some $b\in\{p_1,\,p_2,\,p_1p_2,\,p_1p_2p_3\}$
\item[(ii)] $p_1p_3\in [x^v, \, x^u]$ and $p_2^2 \le D/p_1p_3$
\item[(iii)] $p_2p_3\in [x^v, \, x^u]$ and $p_1^2 \le D/p_2p_3$
\item[(iv)] $p_3 \ \le \ x^{u-v}$, and ($p_1 \ \le \ x^v$ and $p_2^2 \le x^{1-2\delta}/D$) or ($p_2 \ \le \ x^v$ and $p_1^2 \le x^{1-2\delta}/D$)
\end{itemize}
\end{proposition}
\begin{proof}
First suppose $b=p_1\cdots p_i\in [x^v, \, x^u]$ for some $i\in\{1,2,3\}$, and let $C=D/Ab$. Since $p_1\cdots p_ip_{i+1}^2\le D$, we get $p_{i+1}^2 \le D/b=AC$ so that either $p_{i+1}\le A$ or $p_{i+1}\le C$. Similarly $p_1\cdots p_{j-1}p_j^2\le D$ for all $i<j\le r$, we get $p_j^2 \le \frac{AC}{p_{i+1}\cdots p_{j-1}}$ and so by induction we may factor $p_{i+1}\cdots p_r=ac$ where $a\le A$, $c\le C$. Hence since $b\in [x^v, \, x^u]$, by Lemma \ref{lem:criterion} $p_1\cdots p_r=abc$ satisfies \eqref{eq:JMfactorable}.

Else $p_1,p_1p_2,p_1p_2p_3\notin [x^v, \, x^u]$. Suppose $b\in [x^v, \, x^u]$ for some $b\in\{p_2p_3,\,p_1p_3,\,p_2\}$. By assumption $p_1\le\, x^{u}$, so $p_1\notin [x^v, \, x^u]$ further implies $p_1<x^{v}$. We have:
\begin{itemize}
\item If $b=p_2$ then $p_3^2 \le D/p_1p_2= D/p_1b$ implies $p_1p_3=ac$ factors for $a\le A, c\le D/Ab$.
\item If $b=p_1p_3$ and $p_2^2 \le D/b$ by (ii), then $p_2\le A$ or $p_2\le D/Ab$.
\item If $b=p_2p_3$ and $p_1^2 \le D/b$ by (iii), then $p_1\le A$ or $p_1\le D/Ab$.
\end{itemize}
For each $b\in [x^v, \, x^u]$ above, we factored $p_1p_2p_3=abc$ where $a\le A, c\le D/Ab$. Then since $p_1\cdots p_{j-1}p_j^2\le D$ for all $j\le r$, by induction we may factor $p_1\cdots p_r=abc$ for $a\le A, c\le D/Ab$. By Lemma \ref{lem:criterion} $p_1\cdots p_r=abc$ will satisfy \eqref{eq:JMfactorable}.

Finally, suppose (iv) holds: $p_3 \ \le \ x^{u-v}$, $p_2^2 \le x^{1-2\delta}/D$, and $d_2:=p_1 \ \le \ x^v$. Then Lemma \ref{lem:greedy} completes the proof.
\end{proof}

\begin{proof}[Proof of Proposition \ref{prop:piecewiseeta}]
Recall $D=x^{{\pmb \vartheta}-2\delta}$ and $p_i=x^{t_i}$.
By Proposition \ref{prop:Dwell}, $d=p_1\cdots p_r$ has a factorization $d=abc$ satisfying \eqref{eq:JMfactorable} provided $t_1\le u$, and one of the following holds:
\begin{itemize}
\item[(i)] $\tau \in [v,u]$ for some subsum $\tau\in\{t_1,\,t_2,\,t_1+t_2,\,t_1+t_2+t_3\}$
\item[(ii)] $t_1+t_3\in [v,u]$ and ${\pmb \vartheta} \ge t_1+2t_2+t_3$
\item[(iii)] $t_2+t_3\in [v,u]$ and ${\pmb \vartheta} \ge 2t_1+t_2+t_3$
\item[(iv)] $t_3 \le u-v$ and ($t_1\le v$ and $2t_2 \le 1-{\pmb \vartheta}$) or ($t_2\le v$ and $2t_1 \le 1-{\pmb \vartheta}$)
\end{itemize}

To this, note
\begin{align*}
2{\pmb \vartheta}-1 =v \ \le \ t \ \le \ u= \min\Big(\frac{1-(1+{\theta}){\pmb \vartheta}}{3/2-2{\theta}},\, \frac{2 -2 {\pmb \vartheta}-{\alpha\theta}}{3} \Big)
\end{align*}
so that
\begin{align}\label{eq:tinvuiff}
t\in[v,u] \qquad \Longleftrightarrow \qquad
{\pmb \vartheta}  \ \le \ \min\Big(\frac{1+t}{2}, \ 
\frac{1-(3/2-2{\theta})t}{1+{\theta}}, \ 
1 - \frac{{\alpha\theta} + 3t}{2}\Big) =: {\pmb \vartheta}(t).
\end{align}

For (i), by \eqref{eq:tinvuiff} we see \eqref{eq:JMfactorable} holds if ${\pmb \vartheta} \le \max\big\{{\pmb \vartheta}(t_1),\,{\pmb \vartheta}(t_2),\,{\pmb \vartheta}(t_1+t_2),\,{\pmb \vartheta}(t_1+t_2+t_3)\big\}$. 

For (ii), by \eqref{eq:tinvuiff} we see \eqref{eq:JMfactorable} holds if $t_1+2t_2+t_3\le {\pmb \vartheta} \le {\pmb \vartheta}(t_1+t_3)$. Recalling the notation $\psi(x,y) = x\cdot\1_{x\ge y}$, this is equivalent to the condition ${\pmb \vartheta} \le \psi\big({\pmb \vartheta}(t_1+t_3),t_1+2t_2+t_3\big)$.

For (iii) similarly, \eqref{eq:JMfactorable} holds if ${\pmb \vartheta} \le \psi\big({\pmb \vartheta}(t_2+t_3),2t_1+t_2+t_3\big)$.

For (iv), we have that
\begin{align}
t_3 \ \le \ u-v &= \min\Big(\frac{1-(1+{\theta}){\pmb \vartheta}}{3/2-2{\theta}},\, \frac{2 -2 {\pmb \vartheta}-{\alpha\theta}}{3} \Big) - (2{\pmb \vartheta}-1) \nonumber\\
& = \min\Big(\frac{(5/2-2{\theta})-(4-3{\theta}){\pmb \vartheta}}{3/2-2{\theta}},\ 
\frac{5-{\alpha\theta}-8{\pmb \vartheta}}{3}\Big),
\end{align}
and ($t_1 \le \, v = 2{\pmb \vartheta}-1$ and $2t_2 \le \, 1 - {\pmb \vartheta}$) or ($t_2 \le \, v = 2{\pmb \vartheta}-1$ and $2t_1 \le \, 1 - {\pmb \vartheta}$), is equivalent to either
\begin{align*}
\frac{1+t_1}{2}\ \le \ {\pmb \vartheta} \ \le \ \min\Big\{\frac{(5/2-2{\theta})-(3/2-2{\theta})t_3}{4-3{\theta}},\, \frac{5-{\alpha\theta}-3t_3}{8},\, 1-2t_2\Big\},
\end{align*}
or
\begin{align*}
\frac{1+t_2}{2}\ \le \ {\pmb \vartheta} \ \le \ \min\Big\{\frac{(5/2-2{\theta})-(3/2-2{\theta})t_3}{4-3{\theta}},\, \frac{5-{\alpha\theta}-3t_3}{8},\, 1-2t_1\Big\}.
\end{align*}
Note the maximal ${\pmb \vartheta} \in [y,x]$ is given by ${\pmb \vartheta} = x\cdot\1_{x\ge y}=:\psi(x,y)$.
Thus \eqref{eq:JMfactorable} holds if ${\pmb \vartheta}\le \, \max\{w(t_1,t_2,t_3),w(t_1,t_3,t_2)\}$ where
\begin{align}
w(t_1,t_2,t_3) := 
\psi\Big(\min\Big\{\;\frac{(5-4{\theta})-(3-4{\theta})t_3}{8-6{\theta}},\; \frac{5-{\alpha\theta}-3t_3}{8},\;1-2t_2\Big\},\,\frac{1+t_1}{2}\Big).
\end{align}

Taking the maximum ${\pmb \vartheta}$ allowed among cases (i)--(iv), we see \eqref{eq:JMfactorable} holds if
\begin{align}\label{eq:theta123}
{\pmb \vartheta} = \max\Big\{ {\pmb \vartheta}(t_1)&,\,{\pmb \vartheta}(t_2),\,{\pmb \vartheta}(t_1+t_2),\,{\pmb \vartheta}(t_1+t_2+t_3),\,w(t_1,t_2,t_3),\,w(t_1,t_3,t_2) \nonumber\\
& \ \ 
\psi\big({\pmb \vartheta}(t_1+t_3),\,t_1+2t_2+t_3\big), \ \psi\big({\pmb \vartheta}(t_2+t_3),\,2t_1+t_2+t_3\big)\Big\},
\end{align}
crucially provided that $t_1 \le\, u$ holds. 

To this, we have ${\pmb \vartheta} = {\pmb \vartheta}(\tau)$ for some subsum $\tau$ in \eqref{eq:theta123}.
If $\tau=t_1$, then $t_1 \le\, u= \min\big(\frac{1-(1+{\theta}){\pmb \vartheta}(t_1)}{3/2-2{\theta}},\, \frac{2 -2 {\pmb \vartheta}(t_1)-{\alpha\theta}}{3} \big)$ automatically holds, by the definition of ${\pmb \vartheta}(t_1)$.

Else the subsum is $\tau\neq t_1$, i.e. ${\pmb \vartheta}(\tau) > {\pmb \vartheta}(t_1)$. Then $t_1\le \mu$ by assumption in Proposition \ref{prop:piecewiseeta}.\footnote{Recall $\mu= \min(\frac{1-{\theta}}{4-3{\theta}},\, \frac{1-{\alpha\theta}}{4})$ is the {\it balance point} of the (unimodal) function ${\pmb \vartheta}(t)$, from \eqref{eq:thetat}. } Thus $t_1\le \mu < u = \min\big(\frac{1-(1+{\theta}){\pmb \vartheta}(\tau)}{3/2-2{\theta}},\, \frac{2 -2 {\pmb \vartheta}(\tau)-{\alpha\theta}}{3} \big)$.
This completes the proof in all cases.
\end{proof}

\subsection{Level of distribution with linear sieve weights} \label{subsec:Iwaniecweights}

Recall the (upper and lower) linear sieve weights $\lambda^\pm$ of level $D$ are defined as
\[\lambda^\pm(d) = \begin{cases}
    \mu(d) & \text{if }d\in \mathcal D^\pm(D),\\
    0 & \text{else,}
\end{cases}\]
for support sets $\mathcal D^\pm$, given by
\begin{align*}
\mathcal D^+(D) \ &= \ \big\{ \, d=p_1\cdots p_r \ : \  p_1\cdots p_{m-1}p_m^3 < D \quad {\rm for \ odd} \ \ m\le r\big\},\\
\mathcal D^-(D) \ &= \ \big\{ \, d=p_1\cdots p_r \ : \ p_1\cdots p_{m-1}p_m^3 < D \quad {\rm for \ even} \ \ m\le r, \ p_1^2 < D\big\}. \nonumber
\end{align*}
Observe the support sets $\mathcal D^\pm$ are contained in $\mathcal D^{\rm well}$ from \eqref{eq:Dwelldef}. Consider the analogous set of well-factorable vectors $\mathbf D_r^{\textnormal{well}}$,
\begin{align}\label{eq:Drwellvector}
\mathbf D_r^{\textnormal{well}}(D) & =
\{(D_1,\ldots,D_r): D_1\cdots D_{m-1}D_m^2 < D \quad\text{ for all }m\le r\},
\end{align}
where all entries $D_1,\ldots,D_r$ are numbers in the sequence $(D^{\eps^2(1+\eps^9)^j})_{j\ge 1}$.

In reality, we shall work with Iwaniec's modified linear sieve weights $\widetilde{\lambda}^\pm$, which enjoy the same sieve upper and lower bounds as $\lambda^\pm$, but are also well-factorable (the original weights $\lambda^\pm$ are not, by a minor technicality). 
The construction of $\widetilde{\lambda}^\pm$ makes use of the vectors in \eqref{eq:Drwellvector}; see \cite[\S12.4]{Opera} for a precise definition of $\widetilde{\lambda}^\pm$ and further details.

The key point is that Iwaniec's well-factorable weights $\widetilde{\lambda}^\pm(d)$ essentially inherit the same factorization properties enjoyed by the integers $d$ in the support $d\in\mathcal D^{\rm well} (x^{\pmb \vartheta})$. As such, the factorization of $\mathcal D^{\rm well}$ in Proposition \ref{prop:piecewiseeta} implies an improved level of distribution for primes with linear sieve weights.

\begin{proposition}\label{prop:lvlalpha1}
Let $(D_1,\ldots,D_r)\in\mathbf D_r^{\textnormal{well}}(D)$ and write $D=x^{\pmb \vartheta}$, $D_i=x^{t_i}$ for $i\le r$. 

If ${\pmb \vartheta} \le {\pmb \vartheta}(t_1)-\eps$ as in \eqref{eq:thetat}, then
\begin{align}\label{eq:lvlalpha1}
\sum_{\substack{b=p_1\cdots p_r\\ D_i< p_i \le D_i^{1+\eps^9}}}
\sum_{\substack{d=bc\le x^{{\pmb \vartheta}}\\c\mid P(p_r)\\(d,a)=1}} \widetilde{\lambda}^\pm(d) \,\Big(\pi(x;d,a)-\frac{\pi(x)}{\phi(d)}\Big) \ \ll_{a,A,\epsilon} \ 
 \frac{x}{(\log x)^A}.
\end{align}
And if $t_1 \le \,\min(\frac{1-{\theta}}{4-3{\theta}},\, \frac{1-\alpha\theta}{4})$ and $r\ge3$, then \eqref{eq:lvlalpha1} holds if ${\pmb \vartheta} \, \le \, {\pmb \vartheta}(t_1,t_2,t_3)-\eps$ as in \eqref{eq:thetat123}. 
\end{proposition}
\begin{proof}
This follows just as in the proof of \cite[Proposition 5.4]{Ltwin}, just substituting Proposition \ref{prop:piecewiseeta} above in for \cite[Proposition 3.3]{Ltwin}, with an updated level ${\pmb \vartheta}$.
\end{proof}

\section{Bounds for twin primes and Goldbach representations}

In this section, we complete the proofs of Theorems \ref{thm:twinbound} and \ref{thm:Goldbachbound}, by combining the sieve bounds from \cite{Ltwin} with our level of distribution results for the linear sieve weights $\widetilde{\lambda}^\pm$, from the previous section. Indeed, we substitute Proposition \ref{prop:lvlalpha1} in from \cite[Proposition 5.4]{Ltwin}, which in practice amounts to modifying the level to ${\pmb \vartheta}$, as in \eqref{eq:thetat} and \eqref{eq:thetat123}, into the prior bounds from \cite{Ltwin}. Also see \cite{LprimeAP} for related applications of level of distribution results.

As a brief summary of the sieve-theoretic ideas, following the spirit of Fouvry--Grupp \cite{FG} we essentially use a weighted sieve inequality, and iterate the Buchstab identity in a prescribed fashion. We select certain terms to drop (by positivity), and to certain other terms we apply the Chen--Iwaniec switching principal. This approach was recursively optimized in Wu \cite{WuII}. Finally, to each such term, we apply the linear sieve upper and lower bounds. The remainder terms in such sieve bounds are controlled using level of distribution results for $\pi(x;q,a)$, as in Proposition \ref{prop:lvlalpha1}. with residue $a=-2$. We refer the reader to \cite{Ltwin} for further details about this computation.

\subsection{Sieve-theoretic bounds} We now recall the notations and bounds in \cite[\S6]{Ltwin}.
Given $\mu\le \mu_\alpha$ and parameters $0.1\le \rho' \le \tau_1 < \mu \le \tau_2 < \tau_3 \le \rho \le 0.3.$, we define the following integrals $I_{n}=I_{n}(\rho,\rho',\tau_1,\tau_2,\tau_3)$ by
\begin{align}\label{eq:defIn}
I_n &= \int_{\D_{n}}\omega\Big(\frac{1-t-u-v}{u}\Big)\frac{\dd{t}\dd{u}\dd{v}}{tu^2v} \quad\qquad\qquad(9\le n\le 15),\nonumber\\
I_n &= \int_{\D_{n}}\omega\Big(\frac{1-t-u-v-w}{v}\Big)\frac{\dd{t}\dd{u}\dd{v}\dd{w}}{tuv^2w} \quad\qquad(16\le n\le19), \\
I_{20} &= \int_{\D_{20}}\omega\Big(\frac{1-t-u-v-w-x}{w}\Big)\frac{\dd{t}\dd{u}\dd{v}\dd{w}\dd{x}}{tuvw^2x},\nonumber\\
I_{21} &= \int_{\D_{21}}\omega\Big(\frac{1-t-u-v-w-x-y}{x}\Big)\frac{\dd{t}\dd{u}\dd{v}\dd{w}\dd{x}\dd{y}}{tuvwx^2y},\nonumber
\end{align}
where $\omega$ is the Buchstab function, and where the domains $\D_{n}$ are
\begin{align*}
\D_{9} \ &= \{(t,u,v) : \tau_1 < t<u<v < \tau_3\}, \\
\D_{10} &= \{(t,u,v) : \tau_1 < t<u< \tau_2 < v < \rho\}, \\
\D_{11} &= \{(t,u,v) : \tau_1 < t<\tau_2<u < v < \tau_3\}, \\
\D_{12} &= \{(t,u,v) : \rho' < t<u<\tau_1, \ \tau_3 < v < \rho\}, \\
\D_{13} &= \{(t,u,v) : \rho' < t<\tau_1 < u <\tau_2< v < \rho\}, \\
\D_{14} &= \{(t,u,v) : \rho' < t<\tau_1, \ \tau_2 < u < v < \rho\}, \\
\D_{15} &= \{(t,u,v) : \tau_1 < t<\tau_2 < u < \tau_3 <  v < \rho\}, \\
\D_{16} &= \{(t,u,v,w) : \tau_2 < t<u<v<w<\tau_3\}, \\
\D_{17} &= \{(t,u,v,w) : \tau_2 < t<u<v<\tau_3<w<\rho\}, \\
\D_{18} &= \{(t,u,v,w) : \tau_2 < t<u<\tau_3<v<w<\rho\}, \\
\D_{19} &= \{(t,u,v,w) : \tau_1 < t<\tau_2, \ \tau_3 <u<v<w<\rho\}, \\
\D_{20} &= \{(t,u,v,w,x) : \tau_2 < t<\tau_3 < u<v<w<x<\rho\}, \\
\D_{21} &= \{(t,u,v,w,x,y) : \tau_3 < t< u<v<w<x<y<\rho\}.
\end{align*}

As in \eqref{eq:thetat}, denote
\begin{align}\label{eq:thetat4}
{\pmb \vartheta}(t) = {\pmb \vartheta}_\alpha(t) &:= 
\begin{cases}
\min\big(
\frac{1-(3/2-2{\theta})t}{1+{\theta}}, \ 
1 - \frac{\alpha\theta + 3t}{2}\big) & \text{if} \ \ t > \mu_\alpha,\\
\frac{1+t}{2} & \text{if} \ \ t \le \, \mu_\alpha.
\end{cases}
\end{align}
where $\mu_\alpha = \min(\frac{1-{\theta}}{4-3{\theta}},\, \frac{1-\alpha\theta}{4})$, and from \eqref{eq:thetat123},
\begin{align}\label{eq:thetat1234}
{\pmb \vartheta}(t,u,v) := \max\Big\{ {\pmb \vartheta}(t)&,\,{\pmb \vartheta}(u),\,{\pmb \vartheta}(t+u),\,{\pmb \vartheta}(t+u+v),\,w(t,u,v),\,w(t,v,u) \nonumber\\
& \ \ 
\psi\big({\pmb \vartheta}(t+v),\,t+2u+v\big), \ \psi\big({\pmb \vartheta}(u+v),\,2t+u+v\big)\Big\},
\end{align}
where $\psi(x,y) := x\cdot\1_{x\ge y}$ and
\begin{align*}
w(t,u,v) &= \psi\Big(\min\Big\{\;\frac{(5-4{\theta})-(3-4{\theta})v}{8-6{\theta}},\; \frac{5-{\alpha\theta}-3v}{8},\;1-2u\Big\},\,\frac{1+t}{2}\Big).
\end{align*}

We also define
\begin{align}\label{eq:defG14}
G_1 &= 4G(\rho') + G(\tau_1), & G_3 &= G_0 + \overline{G}(\tau_2),\\
G_2 &= G_0 + \overline{G}(\rho), & G_4 &= G_0 + \overline{G}(\tau_3),\nonumber
\end{align}
where for $c\le \mu$,
\begin{align}\label{eq:Gc}
G(c) &= \frac{1}{\epsilon}\,F\big({\pmb \vartheta}_\epsilon/\epsilon\big) - \frac{1}{\epsilon}\int_{\epsilon}^{c}\frac{\dd{t}}{t}f\big(({\pmb \vartheta}(t,\epsilon,\epsilon) - t)/\epsilon\big) + \frac{1}{\epsilon}\int_{\epsilon}^{c}\int_{\epsilon}^{t}\frac{\dd{t}\dd{u}}{t u} F\big(({\pmb \vartheta}(t,u,\epsilon) - t-u)/\epsilon\big) \nonumber\\
&\
 -\int_{\epsilon}^{c}\int_{\epsilon}^{t}\int_{\epsilon}^u \frac{\dd{t}\dd{u}\dd{v}}{t u v^2} f\big(({\pmb \vartheta}(t,u,v) - t-u-v)/v\big),
\end{align}
and for $c>\mu$,
\begin{align}\label{eq:barGc}
\overline{G}(c) &= - \frac{1}{\epsilon}\int_{\mu}^{c}\frac{\dd{t}}{t}f\big(({\pmb \vartheta}(t) - t)/\epsilon\big) + \int_{\mu}^{c}\int_{\epsilon}^{\rho'}\frac{\dd{t}\dd{u}}{tu^2}F\big(({\pmb \vartheta}(t) - t-u)/u\big)
\end{align}
as well as
\begin{align}
G_0 &= - \frac{1}{\epsilon}\int_{\rho'}^{\mu}\frac{\dd{t}}{t}f\big(({\pmb \vartheta}(t,\epsilon,\epsilon) - t)/\epsilon\big) + \frac{1}{\epsilon}\int_{\rho'}^{\mu}\int_{\epsilon}^{\rho'}\frac{\dd{t}\dd{u}}{t u} F\big(({\pmb \vartheta}(t,u,\epsilon) - t-u)/\epsilon\big)\\
&\
-\int_{\rho'}^{\mu}\int_{\epsilon}^{\rho'}\int_{\epsilon}^u \frac{\dd{t}\dd{u}\dd{v}}{t u v^2} f\big(({\pmb \vartheta}(t,u,v) - t-u-v)/v\big). \nonumber
\end{align}

We similarly let
\begin{align}\label{eq:defG68}
G_5 & = \frac{1}{\epsilon}\int_{\rho'}^{\mu}\int_{\rho'}^{t} \frac{\dd{t}\dd{u}}{t u} \,F\big(({\pmb \vartheta}(t,u,\epsilon) - t-u)/\epsilon\big) \ + \ \frac1{\rho'}\int_{\mu}^{\tau_2}\int_{\rho'}^{t} \frac{\dd{t}\dd{u}}{t u} \,F\big(({\pmb \vartheta}(t) - t-u)/\rho'\big)\nonumber\\
& \ - \int_{\rho'}^{\mu}\int_{\rho'}^{t}\int_{\epsilon}^{\rho'} \frac{\dd{t}\dd{u}\dd{v}}{t u v^2} \,f\big(({\pmb \vartheta}(t,u,v) - t-u-v)/v\big),\\
G_6 & = \frac{1}{\rho'}\int_{\tau_2}^{\tau_3}\int_{\rho'}^{\tau_1} \frac{\dd{t}\dd{u}}{t u} \,F\big(({\pmb \vartheta}(t) - t-u)/\rho'\big), \nonumber
\end{align}
\begin{align*}
G_7 &= \frac1{\epsilon}\int_{\rho'}^{\tau_1}\int_{\rho'}^{t} \frac{\dd{t}\dd{u}}{t u} \,F\big(({\pmb \vartheta}_\epsilon - t-u)/\epsilon \big)\\
&\ - \int_{\rho'}^{\tau_1}\int_{\rho'}^{t} \int_{\epsilon}^{u} \frac{\dd{t}\dd{u}\dd{v}}{t u v^2} \,f\big(({\pmb \vartheta}(t,u,v) - t-u-v)/v\big),\\
G_8 &= \frac{1}{\epsilon}\int_{\tau_1}^{\mu}\int_{\rho'}^{\tau_1} \frac{\dd{t}\dd{u}}{t u} \,F\big(({\pmb \vartheta}(t,u,\epsilon) - t-u)/\epsilon\big) \ + \int_{\mu}^{\tau_2}\int_{\rho'}^{\tau_1} \frac{\dd{t}\dd{u}}{t u^2} \,F\big(({\pmb \vartheta}(t) - t-u)/u\big)\\
&\ - \int_{\tau_1}^{\mu}\int_{\rho'}^{\tau_1} \int_{\epsilon}^{u} \frac{\dd{t}\dd{u}\dd{v}}{t u v^2} \,f\big(({\pmb \vartheta}(t,u,v) - t-u-v)/v\big).
\end{align*}

Recall the standard linear sieve functions $F,f$ satisfy $F(s)=2e^\gamma/s$ for $s\in[1,3]$, $f(s) = 2e^\gamma\log(s-1)/s$ for $s\in[2,4]$ and $F(s) = 2e^\gamma/s\cdot[1+\int_2^{s-1} f(t)\dd{t}]$ for all $s\ge1$.

We also use the savings from Wu \cite{WuII} over $G_2$. Namely, 
\begin{align}\label{eq:H2Wu}
H_2^{\rm Wu} &= \int_{\mu}^{\rho}\int_{\epsilon}^{\rho'}\frac{\dd{t}\dd{u}}{tu^2}F^{\rm Wu}\big(({\pmb \vartheta}(t) - t-u)/u\big),
\end{align}
where $F^{\rm Wu}$ is given by $F^{\rm Wu}(s) = F(s)\cdot H_{\pmb\vartheta}^{\rm Wu}(s)$. Here the savings function $H_{\pmb \vartheta}^{\rm Wu}$ is monotonically increasing in the level of distribution ${\pmb \vartheta}$. For ${\pmb \vartheta}=4/7$ and parameters as in Tables 1 and 2 \cite[pp.30--32]{WuII},
\[
H_{4/7}^{\rm Wu}(t) \ \ge \ \begin{cases}
0.0287118 \qquad \text{if} \ 2.0 \, \le \, t \le \, 2.1,\\
0.0280509 \qquad \text{if} \ 2.1 <  t \le \, 2.2, \\
0.0264697 \qquad \text{if} \ 2.2 < t \le \, 2.3,\\
0.0241936 \qquad \text{if} \ 2.3 < t \le \, 2.4,\\
0.0214619 \qquad \text{if} \ 2.4 < t \le \, 2.5, \\
0.0183875 \qquad \text{if} \ 2.5 < t \le \,2.6, \\
0.0149960 \qquad \text{if} \ 2.6 < t \le  \,2.7, \\
0.0117724 \qquad \text{if} \ 2.7 < t \le \, 2.8, \\
0.0094724 \qquad \text{if} \ 2.8 < t \le \, 2.9, \\
0.0090024 \qquad \text{if} \ 2.9 < t \le  \,3.0, \\
0  \qquad\qquad\qquad\textnormal{else.}
\end{cases}
\]

Consider the set of shifted primes $\mathcal A = \{|p-a| : p\le\, x\}$, and recall the main object of interest is the sifted sum,
\begin{align*}
S(\mathcal A, z) :=\#\{n\in A \;:\; p\mid n\implies p\ge\, z \ \ {\rm and } \ \ p\nmid a\}.
\end{align*}

\begin{proposition}\label{prop:wu}
For $a\neq0$ let $\mathcal A = \{|p-a| : p\le\, x\}$. Let $0<\epsilon \le 0.1\le \rho' \le \tau_1 < \mu \le \tau_2 < \tau_3 \le \rho \le 0.3$. Then for $I_n$, $G_n$, and $G(c)$ as in \eqref{eq:defIn}, \eqref{eq:defG14}, \eqref{eq:defG68}, and \eqref{eq:Gc}, we have
\begin{align}
S(\mathcal A, x^\rho) \  \lesssim \ \frac{\Pi_a(x)}{5e^\gamma}
\bigg( \sum_{1\le n\le 8} G_n - H_2^{\textnormal{Wu}} \ + \  G(\mu)\sum_{9\le\, n\le\, 21} I_n \bigg).
\end{align}
\end{proposition}
\begin{proof}
This follows just as in the proof of \cite[Proposition 6.3]{Ltwin} (more precisely, (6.24) in \cite{Ltwin}), using Proposition \ref{prop:lvlalpha1} above as the key level of distribution input in place of \cite[Proposition 5.5]{Ltwin}. Also $G_2^{\textnormal{Wu}} = G_2-H_2^{\textnormal{Wu}}$ and we use $\mu$ in place of $1/5$.
\end{proof}

\subsection{Deduction of Theorem \ref{thm:twinbound}} \label{sec:linearsievelevel}
Let $a=-2$. The primes in our set $\mathcal A = \{p+2 : p\le x\}$ count twin primes, and so $\pi_2(x) \le\, S(\mathcal A, z) + O(\sqrt{x})$ for $z\le \sqrt{x}$.

When $\alpha=0$, we may simplify ${\pmb \vartheta}_\alpha(t)$ in \eqref{eq:thetat4} as
\begin{align}
{\pmb \vartheta_0}(t) =
\begin{cases}
\frac{1-(3/2-2{\theta})t}{1+{\theta}} & \text{if} \ \ t > \mu_0,\\
\frac{1+t}{2} & \text{if} \ \ t \le \, \mu_0.
\end{cases}, \qquad{\rm where}\quad
\mu_0 = \frac{1-{\theta}}{4-3{\theta}}.
\end{align}
Numerically $\mu_0 = \frac{25}{107}\approx 0.233$, since ${\theta} = \frac{7}{32}$ by Kim--Sarnak \cite{KimSarn}. We set parameters
\begin{align}\label{eq:params2}
\rho  &= 0.275, & 
\tau_3 &= 0.24589, &  \mu &= 0.210, \nonumber\\
\rho' &= 0.12313, &
\tau_2 &= 0.211, & \epsilon &= 0.002.\\
& &
\tau_1 &= 0.163 & & \nonumber
\end{align}
For such choices of parameters, we obtain 
\begin{align*}
\sum_{1\le n\le 8} G_n \ \le \ 27.7086,\quad
H_2^{\rm Wu} \ \ge \ 0.019309,\qquad
G(\mu) \ \le \ 5.90044,\quad
\sum_{9\le n\le 21} I_n \ \le \ 0.180677.
\end{align*}
These computations were performed in Mathematica, using the `SieveFunction.m' standard package (Galway). We also record the integrals $I_n$, $G_n$ in the table below.
\vspace{3mm}
\begin{center}
\begin{tabular}{r|r|rlrl}
$n$ & $G_n$  & $n$ & $I_n$ & $n$ & $I_n$\\
\hline
1 & 38.9215  & 9  & 0.0330294 & 17 & 0.000282\\
2&$-$5.80465 & 10 & 0.0247846 & 18 & 0.000287\\
3&$-$4.10858 & 11 & 0.0084670 & 19 & 0.000231\\
4&$-$5.17066 & 12 & 0.0167535 & 20 & $\le \ 3.80\cdot10^{-6}$\\
5 &  1.87682 & 13 & 0.0566827 & 21 & $\le \ 1.02\cdot10^{-8}$\\
6 &  0.636696 & 14 & 0.0264459 &    & \\
7 &  0.428799 & 15 & 0.0136088 &    & \\
8 &  0.928682 & 16 & 0.0000988 &    & 
\end{tabular}
\end{center}
\vspace{3mm}
Thus by Proposition \ref{prop:wu}, we obtain the bound
\begin{align}\label{eq:pi2oneiteration}
\pi_2(x) 
\ &\lesssim \ S(\mathcal A, x^\rho) \  \lesssim \ \frac{\Pi_2(x)}{5e^\gamma} \bigg( \sum_{1\le n\le 8} G_n - H_2^{\textnormal{Wu}} \ + \  G(\mu)\sum_{9\le\, n\le\, 21} I_n \bigg) \nonumber\\
&\lesssim  \ 3.22899\,\Pi_2(x).
\end{align}
This completes the proof of Theorem \ref{thm:twinbound}.

\subsection{Deduction of Theorem \ref{thm:Goldbachbound}}

Let $x=a$. The primes in our set $\mathcal A = \{a-p : p\le a\}$ count Goldbach representations of $a$, and so ${\rm G}(a) \le\, S(\mathcal A, z)+ O(\sqrt{a})$ for $z\le \sqrt{a}$.

When $\alpha=1$, we may simplify ${\pmb \vartheta}_\alpha(t)$ in \eqref{eq:thetat4} as
\begin{align*}
{\pmb \vartheta_1}(t) =
\begin{cases}
1 - \tfrac{{\theta}+3t}{2} & \text{if} \ \ t > \mu_1,\\
\frac{1+t}{2} & \text{if} \ \ t \le \, \mu_1.
\end{cases}, \qquad{\rm where}\quad
\mu_1 = \frac{1-{\theta}}{4}.
\end{align*}
Numerically $\mu_1 = \frac{25}{128}\approx 0.195$, since ${\theta} = \frac{7}{32}$ by Kim--Sarnak \cite{KimSarn}. We set parameters
\begin{align}\label{eq:paramsa}
\rho  &= 0.2445, & 
\tau_3 &= 0.224, &  \mu &= 0.169, \nonumber\\
\rho' &= 0.128, &
\tau_2 &= 0.205, & \epsilon &= 0.002.\\
& &
\tau_1 &= 0.163 & & \nonumber
\end{align}
For such choices of parameters, we obtain
\begin{align*}
\sum_{1\le n\le 8} G_n \ \le \ 29.6847,\quad
H_2^{\rm Wu} \ \ge \ 0.025787,\qquad
G(\mu) \ \le \ 6.34862,\quad
\sum_{9\le n\le 21} I_n \ \le \ 0.084421.
\end{align*}
These computations were performed in Mathematica, using the `SieveFunction.m' standard package (Galway). We also record the integrals $I_n$, $G_n$ in the table below.
\vspace{3mm}
\begin{center}
\begin{tabular}{r|r|rlrl}
$n$ & $G_n$  & $n$ & $I_n$ & $n$ & $I_n$\\
\hline
1 & 37.9006        & 9  & 0.0153459 & 17 & $\le \ 4.63\cdot10^{-5}$\\
2&$-$4.13212       & 10 & 0.0130481 & 18 & $\le \ 6.53\cdot10^{-5}$\\
3&$-$3.29997       & 11 & 0.0023251 & 19 & 0.000109\\
4&$-$3.80586       & 12 & 0.0095937 & 20 & $\le \ 9.20\cdot10^{-7}$\\
5 &  1.53741       & 13 & 0.0296655 & 21 & $\le \ 2.62\cdot10^{-9}$\\
6 &  0.365983      & 14 & 0.0093697 &    & \\
7 &  0.362074      & 15 & 0.0048386 &    & \\
8 &  0.756609      & 16 & $\le \ 1.19\cdot10^{-5}$ &    & 
\end{tabular}
\end{center}
\vspace{3mm}
Thus by Proposition \ref{prop:wu}, we obtain the bound
\begin{align}
{\rm G}(a) \ &\lesssim \ S(\mathcal A, a^\rho)
\ \lesssim \ \frac{\Pi_a(a)}{5e^\gamma} \bigg( \sum_{1\le n\le 8} G_n -H_2^{\textnormal{Wu}}\ + \  G(\mu)\sum_{9\le\, n\le\, 21} I_n \bigg) \nonumber \\
&\lesssim  \ 3.39064\,\Pi_a(a).
\end{align}
This completes the proof of Theorem \ref{thm:Goldbachbound}.

%
%
%
%
%
%
%
%

\section{Spectral large sieve}

%
%
%
%
%
%
%
%

In this section, we briefly recall preliminary results from \cite{DI} on the spectral theory of automorphic forms, starting with the Kuznetsov trace formula. See \cite[\S1]{DI} for definitions and further details.

\begin{theorem}[Kuznetsov trace formula]
Let $m,n$ be positive integers and $\phi$ a $C^3$-class function with compact support in $(0,\infty)$. Let $\a$, $\b$ be cusps of $\Gamma=\Gamma_0(q)$. Then we have
\begin{align}
\sum_{\substack{\gamma\in\R\\\exists\smqty(*&*\\\gamma&*)\in \sigma_{\a}^{-1}\Gamma\sigma_{\b}}} &\frac{1}{\gamma}S_{\a\b}(m,n;\gamma)\,\phi\Big(\frac{4\pi}{\gamma}\sqrt{mn}\Big) \nonumber\\
& =\frac{1}{2\pi}\sum_{{\rm even}\,k}\sum_{1\le j\le \theta_k(q)} \frac{i^k(k-1)!}{(4\pi\sqrt{mn})^{k-1}}\, \overline{\psi_{jk}(\a,m)}\psi_{jk}(\b,n)\,\widetilde{\phi}(k-1) \\
& \ + \sum_{j\ge1}\overline{\rho_{j\a}(m)}\rho_{j\b}(n) \frac{\widehat{\phi}(\kappa_j)}{\cosh\pi\kappa_j}
+\frac{1}{\pi}\sum_{\c}\int_\R \Big(\frac{n}{m}\Big)^{ir}\overline{\phi_{\c\a m}(\tfrac{1}{2}+ir)} \phi_{\c\b n}(\tfrac{1}{2}+ir)\widehat{\phi}(r)\dd{r}\nonumber
\end{align}
and
\begin{align}
\sum_{\substack{\gamma\in\R\\\exists\smqty(*&*\\\gamma&*)\in \sigma_{\a}^{-1}\Gamma\sigma_{\b}}} \frac{1}{\gamma}&S_{\a\b}(m,-n;\gamma)\,\phi\Big(\frac{4\pi}{\gamma}\sqrt{mn}\Big) \nonumber\\
& = \sum_{j\ge1}\rho_{j\a}(m)\rho_{j\b}(n) \frac{\check{\phi}(\kappa_j)}{\cosh\pi\kappa_j}
+\frac{1}{\pi}\sum_{\c}\int_\R (mn)^{ir} \phi_{\c\a m}(\tfrac{1}{2}+ir) \phi_{\c\b n}(\tfrac{1}{2}+ir)\check{\phi}(r)\dd{r}
\end{align}
where $\kappa_j$ is defined by $\lambda_j=\frac{1}{4}+\kappa_j^2$, and the Bessel transforms are defined by
\begin{align}
\widetilde{\phi}(l)&=\int_0^\infty J_l(y)\phi(y)\frac{\dd{y}}{y}\label{def:tildephi}\\
\widehat{\phi}(r)&=\frac{\pi}{\sinh \pi r}\int_0^\infty \frac{J_{2ir}(x)-J_{-2ir}(x)}{2i}\phi(x)\frac{\dd{x}}{x} \label{def:hatphi}\\
\check{\phi}(l)&= \frac{4}{\pi}\cosh\pi r\int_0^\infty K_{2ir}(x)\phi(x)\frac{\dd{x}}{x}\label{def:checkphi}
\end{align}
\end{theorem}
\begin{proof}
This is \cite[Theorem 1]{DI}.
\end{proof}

A key ingredient is the following large sieve inequality for Fourier coefficients of cusp forms (both holomorphic and Maass) and Eisenstein series.

Denote $\mu(\infty)=1/q$, and $\mu(\a)=(w,q/w)\,/q$ for a cusp $\a=u/w$.

\begin{theorem}[Spectral large sieve] \label{thm:largesieve}
Let $K\ge1$, $N\ge\frac{1}{2}$, $\eps>0$ be real numbers, a complex sequence ${\bf a}=(a_n)_n$, and $\a$ a cusp $\Gamma_0(q)$. Then each of the following three expressions
\begin{align}
\sum_{\substack{0<k\le K\\k\,{\rm even}}}\frac{(k-1)!}{(4\pi)^{k-1}}\sum_{1\le j\le \theta_{k}(q)} &\bigg|\sum_{n\sim N}a_n n^{-\frac{k-1}{2}}\psi_{jk}(\a,n)\bigg|^2\\
\sum_{|\kappa_j|\le K}\frac{1}{\cosh \pi\kappa_j} &\bigg|\sum_{n\sim N}a_n\rho_{j\a}(n)\bigg|^2 \\
\sum_{\c}\int_{-K}^K &\bigg|\sum_{n\sim N}a_n n^{ir}\phi_{\c\a n}(\tfrac{1}{2}+ir)\bigg|^2\dd{r}
\end{align}
are each bounded by
\begin{align*}
\ll_\eps (K^2 + \mu(\a)N^{1+\eps})\|{\bf a}_N\|_2^2.
\end{align*}
\end{theorem}
\begin{proof}
This is \cite[Theorem 2]{DI}.
\end{proof}

For the smallest positive eigenvalue $\lambda_1=\lambda_1(q)$ for $\Gamma=\Gamma_0(q)$, recall ${\theta}=\sup_{q}{\theta}_q$ for
\begin{align}
{\theta}_q = \max\big(0,\,\sqrt{1-4\lambda_1}\big).
\end{align}
Note Selberg's lower bound $\lambda_1(q)\ge 3/16$ implies ${\theta}_q \le 1/2$. The current record bound is ${\theta}_q\le 7/32$ by Kim--Sarnak \cite{KimSarn}.

We also use bounds on Bessel-Kuznetsov transforms $\check{f}(r),\, \widehat{f}(r),\, \widetilde{f}(r)$ from \eqref{def:tildephi}--\eqref{def:checkphi}.

\begin{lemma}\label{lem:DI7.1}
Suppose $f\in \mathcal C^2$ is supported in $[X,8X]$ and 
\begin{align}\label{eq:smoothfbounds}
\|f\|_{\infty} \le 1, \qquad
\|f'\|_1 \ll 1, \qquad
\|f''\|_1 \ll \frac{1}{X}.
\end{align}
Then we have
\begin{align} \label{eq:7.1}
\check{f}(ir),\, \widehat{f}(ir) \ &\ll \ \frac{1+X^{-2r}}{1+X} \qquad\qquad r\in (0,\tfrac{1}{2})
\end{align}
and
\begin{align}
\check{f}(r),\, \widehat{f}(r),\, \widetilde{f}(r) \ &\ll \ \frac{1+|\log X|}{1+X} \qquad\qquad r\in\R \label{eq:7.2}\\
&\ll \ \ |r|^{-3/2} + X/|r|^2 \qquad \ |r|\ge1 \label{eq:7.3}\\
&\ll \ |r|^{-5/2} + X/|r|^3  \qquad |r|\ge \max(2X,1). \label{eq:7.4}
\end{align}
\end{lemma}
\begin{proof}
This is \cite[Lemma 2.1]{BHM} with $Z=1$, correcting an error in \eqref{eq:7.4} of \cite[Lemma 7.1]{DI}.
\end{proof}

\section{Exceptional spectrum}\label{sec:largesieve}

In this section, we bound the contribution of exceptional eigenvalues $\lambda_j<1/4$.
For simplicity of exposition, we assume $q_0=1$. Indeed, $q_0=1$ is the only case needed for our applications in this article. Results may obtained with general $q_0$-dependence, as in \cite[\S4.2.3]{Drapp2}.

\subsection{The Case of Fixed Level} We first consider the exceptional spectrum for a fixed congruence subgroup $\Gamma_0(q)$. For a sequence ${\bf a}=(a_n)_n$ recall the norm $\|{\bf a}_N\|_2^2 = \sum_{n\sim N}|a_n|^2$.

Let $\phi(x)\in\mathcal C^\infty$ be a test function, whose derivatives satisfy $\phi^{(l)}(x) \ll x^{-l}$. We have the following result generalizing \cite[Theorem 5]{DI}. In private communication, A. Pascadi has independently obtained corresponding results.

\begin{theorem}\label{thm:DI5}
Let $N,X\ge1$, a sequence ${\bf a}=(a_n)_n\subset \C$, and $\a$ a cusp $\Gamma_0(q)$. Then we have
\begin{align}
\sum_{\lambda_j<1/4}^{(q)} X^{2i\kappa_j} \bigg|\sum_{n\sim N}a_n \rho_{j\a}(n)\bigg|^2 \ \ll_\eps \ 
\big(1+ (\mu(\a)NX)^{{\theta}_q}\big)(1+ (\mu(\a)N^{1+\eps})^{1-{\theta}_q})\|{\bf a}_N\|_2^2.
\end{align}
\end{theorem}
\begin{proof}
Recall $\lambda_j=\tfrac{1}{4}+\kappa_j^2\ge \frac{3}{16}$ by the Selberg bound, so $\kappa_j^2>-\tfrac{1}{16}$. For exceptional $\lambda_j<\tfrac{1}{4}$ this means $i\kappa_j\in[0,\tfrac{1}{4}]$. Thus if $\mu(\a)N >\eps$, by the spectral large sieve in Theorem \ref{thm:largesieve} with $K=\tfrac{1}{4}$, (note $\cosh\pi\kappa \gg 1$ for $|\kappa|\le 1/4$)
\begin{align*}
\sum_{\lambda_j<1/4}^{(q)} X^{2i\kappa_j} \bigg|\sum_{n\sim N}a_n\rho_{j\a}(n)\bigg|^2  &\le \sum_{|\kappa_j|\le \tfrac{1}{4}}\frac{X^{{\theta}_q}}{\cosh \pi\kappa_j} \bigg|\sum_{n\sim N}a_n\rho_{j\a}(n)\bigg|^2\\
&\ll_\eps X^{{\theta}_q}(1 + \mu(\a)N^{1+\eps})\|{\bf a}_N\|_2^2 \\
& \ll \big(1+ (\mu(\a)NX)^{{\theta}_q}\big)(1+ (\mu(\a)N^{1+\eps})^{1-{\theta}_q})\|{\bf a}_N\|_2^2.
\end{align*}

Hence it suffices to assume $Y:=4\pi \mu(\a)N <\eps$ is sufficiently small. Let $\phi(x)=w(x/Y)$ be a test function supported on $[Y,2Y]$. We apply the Kuznetsov formula with $\a=\b$, multiply both sides by $\overline{a_m} a_n$ and sum over $m,n\sim N$,
\begin{align}\label{eq:thm5Kuznet}
0 &= \sum_{m,n\sim N}\overline{a_m} a_n \sum_{\substack{\gamma\in\R\\\exists\smqty(*&*\\\gamma&*)\in \sigma_{\a}^{-1}\Gamma\sigma_{\a}}}  \frac{1}{\gamma}S_{\a\a}(m,n;\gamma)\,\phi\Big(\frac{4\pi}{\gamma}\sqrt{mn}\Big) \nonumber\\
&=\frac{1}{2\pi}\sum_{{\rm even}\,k}\widetilde{\phi}(k-1)\sum_{1\le j\le \theta_k(q)}\frac{i^k(k-1)!}{(4\pi)^{k-1}}\,\bigg|\sum_{n\sim N}a_n n^{\frac{1-k}{2}}\psi_{jk}(\a,n)\bigg|^2
\nonumber\\
& + \sum_{j\ge1} \frac{\widehat{\phi}(\kappa_j)}{\cosh\pi\kappa_j}\bigg|\sum_{n\sim N} a_n\rho_{j\a}(n)\bigg|^2
+\frac{1}{\pi}\sum_{\c}\int_\R \bigg|\sum_{n\sim N} a_n\, n^{ir}\phi_{\c\a n}(\tfrac{1}{2}+ir)\bigg|^2\widehat{\phi}(r)\dd{r}
\end{align}

Importantly, we used that the LHS of \eqref{eq:thm5Kuznet} is empty, i.e. the sum of Kloosterman sums $S_{\a\a}(m,n;\gamma)$ with respect to $\gamma\in \sigma_\a^{-1}\Gamma\sigma_\a$. Indeed by Lemma 2.4 $\mu(\a)^{-1}\mid \gamma$, in particular $1/\gamma \le \mu(\a)$, so the argument of $\phi(4\pi\sqrt{mn}\,/\gamma)$ lies below the support $Y=4\pi\mu(\a)N$ of $\phi$. 

On the RHS of \eqref{eq:thm5Kuznet}, we shall apply the spectral large sieve in Theorem \ref{thm:largesieve}. Indeed, we split regular $\kappa\in\R$ into dyadic ranges $|\kappa|\sim K$, in which range $\widehat{\phi}(\kappa) \ll (K^{-5/2}+Y/K^3 )$ by \eqref{eq:7.4} 
in Lemma \ref{lem:DI7.1}. Thus by Theorem \ref{thm:largesieve},
\begin{align*}
\sum_{\substack{|\kappa_j|\sim K\\\kappa_j\in\R}} \frac{\widehat{\phi}(\kappa_j)}{\cosh\pi\kappa_j}\bigg|\sum_{n\sim N} a_n\rho_{j\a}(n)\bigg|^2 
& \ll (K^{-5/2}+Y/K^3 )\, (K^2+ \mu(\a)N^{1+\eps})\|{\bf a}_N\|_2^2\\
& \ll \, (\frac{1}{\sqrt{K}} + \frac{\mu(\a)N^{1+\eps}}{K^{5/2}})\|{\bf a}_N\|_2^2
\end{align*}
recalling $Y<\eps$. Thus summing over all dyadic intervals, the Maass regular spectra is
\begin{align*}
\sum_{\lambda_j \ge 1/4} \frac{\widehat{\phi}(\kappa_j)}{\cosh\pi\kappa_j}\bigg|\sum_{n\sim N} a_n\rho_{j\a}(n)\bigg|^2  \ll \sum_{K=2^l\ge1/4}(\frac{1}{\sqrt{K}} + \frac{\mu(\a)N^{1+\eps}}{K^{5/2}}) \|{\bf a}_N\|_2^2 \ll  (1+\mu(\a)N^{1+\eps})\|{\bf a}_N\|_2^2. 
\end{align*}
Thus for the RHS of \eqref{eq:thm5Kuznet}, we may similarly bound the regular spectra (with holomorphic and Eisenstein contributions) by $O(1+\mu(\a)N^{1+\eps})$, so that
\begin{align}  \label{eq:thm5Kuznet2}
 &\sum_{\lambda_j\ge 1/4} \frac{\widehat{\phi}(\kappa_j)}{\cosh\pi\kappa_j}\bigg|\sum_{n\sim N} a_n\rho_{j\a}(n)\bigg|^2
+\frac{1}{\pi}\sum_{\c}\int_\R \bigg|\sum_{n\sim N} a_n\, n^{ir} \phi_{\c\a n}(\tfrac{1}{2}+ir)\bigg|^2\widehat{\phi}(r)\dd{r}
\nonumber\\
& \quad+ 
\frac{1}{2\pi}\sum_{{\rm even}\,k}\widetilde{\phi}(k-1)\sum_{1\le j\le \theta_k(q)}\frac{i^k(k-1)!}{(4\pi)^{k-1}}\,\bigg|\sum_{n\sim N}a_n n^{\frac{1-k}{2}}\psi_{jk}(\a,n)\bigg|^2 \ll \  (1+\mu(\a)N^{1+\eps})\|{\bf a}_N\|_2^2.
\end{align}
Hence combining \eqref{eq:thm5Kuznet2}, \eqref{eq:thm5Kuznet} we deduce the exceptional spectra is also bounded by $O(1+\mu(\a)N^{1+\eps})$. That is,
\begin{align}
(1+\mu(\a)N^{1+\eps})\|{\bf a}_N\|_2^2 & \gg \sum_{\lambda_j<1/4} \frac{\widehat{\phi}(\kappa_j)}{\cosh\pi\kappa_j}\bigg|\sum_{n\sim N} a_n\,\rho_{j\a}(n)\bigg|^2 \gg \sum_{\lambda_j<1/4}Y^{-{\theta}_q} \bigg|\sum_{n\sim N} a_n\,\rho_{j\a}(n)\bigg|^2 \nonumber\\
& \gg \frac{1}{1+(XY)^{{\theta}_q}}\sum_{\lambda_j<1/4} X^{2i\kappa_j} \bigg|\sum_{n\sim N}a_n\rho_{j\a}(n)\bigg|^2.
\end{align}
noting $\cosh\pi\kappa=\cos(i\pi\kappa)\le 1$ for $0<i\kappa<\tfrac{1}{4}$ and, c.f. \cite[eq. (8.3)]{DI},
\begin{align}\label{eq:8.3DI}
\widehat{\phi}(\kappa) \gg Y^{-2i\kappa} \gg Y^{-{\theta}_q}.
\end{align}

Recalling $Y\ll \mu(\a)N < \eps$, we conclude
\begin{align*}
\sum_{\lambda_j<1/4} X^{2i\kappa_j} \bigg|\sum_{n\sim N}a_n \rho_{j\a}(n)\bigg|^2 \ &\ll_\eps \ \big(1+ (\mu(\a)NX)^{{\theta}_q}\big)(1+ \mu(\a)N^{1+\eps})\|{\bf a}_N\|_2^2\\
& \ll \big(1+ (\mu(\a)NX)^{{\theta}_q}\big)(1+ (\mu(\a)N^{1+\eps})^{1-{\theta}_q})\|{\bf a}_N\|_2^2.
 \qedhere
\end{align*}
\end{proof}

\subsection{Results on Average}

In this section we prove Theorem \ref{thm:DI6}. Crucially, the test function $\phi(x)$ has smaller support, so the sums over $c$ of Kloosterman sums $S(m,n;qc)$ will no longer be empty. Let
\begin{align}
S(Q,Y,N;s) = \sum_{Q<q\le 16Q} \; \sum_{\lambda_j<1/4}^{(q)} Y^{2i\kappa_j}
\bigg|\sum_{n\sim N} a_n\,n^s \rho_{j\infty}(n)\bigg|^2.
\end{align}

\begin{lemma}(Recurrence for $S$) \label{lem:DI8.1}
Let $Q,N,Y\ge1$ and ${\bf a}=(a_n)_n\subset \C$. Then we have
\begin{align}
S(Q,Y,N,0) \ \ll_\eps \ &\int_{\R}S(\pi NY/Q, Y, N,it)\frac{\dd{t}}{t^4+1}
\ + \ (YN)^\eps (Q+N+NY/Q)\|{\bf a}_N\|_2^2.
\end{align}
\end{lemma}
\begin{proof}
See \cite[Lemma 8.1]{DI}
\end{proof}

\begin{theorem}\label{thm:DI6}
Let $Q,N,X>0$ and ${\bf a}=(a_n)_n\subset \C$. Then we have
\begin{align*}
S(Q,X^2,N,0) &=\sum_{q\asymp Q}\sum_{\lambda_j<1/4}^{(q)} X^{4i\kappa_j} \bigg|\sum_{n\le N} a_n \rho_{j\infty}(n)\bigg|^2 \\
&\ll_\eps \ (QN)^{\eps}\,(Q+N+ NX^{2{\theta}}+Q(\sqrt{N}X/Q)^{2{\theta}}) \|{\bf a}_N\|_2^2.
\end{align*}
\end{theorem}
\begin{proof}
If ${\theta}=0$ there is no exceptional spectrum so $S(Q,X^2,N,0)=0$. Else assume ${\theta}>0$. 

First, by the spectral large sieve in Theorem \ref{thm:largesieve} with $K=\tfrac{1}{4}$, ($\mu(\infty)=1/q$)
\begin{align*}
\sum_{\lambda_j<1/4}^{(q)} \frac{1}{\cosh\pi\kappa_j} \bigg|\sum_{n\le N} \rho_{j\infty}(n)\bigg|^2 \ll (1+N^{1+\eps}/q)\|{\bf a}_N\|_2^2,
\end{align*}
and so summing over $q\asymp Q$,
\begin{align}\label{eq:SQ1N0sieve}
S(Q,1,N,0)=\sum_{q\asymp Q}\sum_{\lambda_j<1/4}^{(q)} \frac{1}{\cosh\pi\kappa_j} \bigg|\sum_{n\le N} \rho_{j\infty}(n)\bigg|^2 \ll (Q+N^{1+\eps})\|{\bf a}_N\|_2^2.
\end{align}
In particular, when $0<X\le 1$ we see $S(Q,X^2,N,0)\le S(Q,1,N,0)$ gives a bound (much stronger than) Theorem \ref{thm:DI6}. Thus it suffices to show that, for all $Y\ge1$,
\begin{align}\label{eq:thm6Y}
S(Q,Y,N,0) \ll (QN)^{2\eps}\,(Q+N Y^{{\theta}}+Q(N Y/Q^2)^{{\theta}}) \|{\bf a}_N\|_2^2.
\end{align}
We have $Y^{2i\kappa_j}<Y^{\theta}$ for exceptional $\lambda_j<1/4$, so \eqref{eq:SQ1N0sieve} implies
\begin{align}\label{eq:thm6thm2}
S(Q,Y,N,0) \le Y^{\theta} S(Q,1,N,0) \ll Y^{{\theta}}(Q+N^{1+\eps})\|{\bf a}_N\|_2^2.
\end{align}

If $N> Q^{1-\eps}$, then \eqref{eq:thm6thm2} implies $S(Q,Y,N,0)  \ll Y^{{\theta}}N^{1+2\eps} \|{\bf a}_N\|_2^2$, which gives \eqref{eq:thm6Y}.

Else we assume $N\le  Q^{1-\eps}$. In this case, we claim that
\begin{align}\label{eq:thm6induct}
\sup_{(a_n)_{n}\subset \C}S(Q,Y_0,N,0) \ll Q^{1+\eps} \|{\bf a}_N\|_2^2
\end{align}
for the specific choice
\[Y_0 := \min\big((Q^{1-\eps}/N)^{1/{\theta}},\,Q^{2-\eps}/N\big). \]
Assuming \eqref{eq:thm6induct}, for arbitrary $Y\ge1$ we conclude
\begin{align*}
S(Q,Y,N,0) &\le \big(1+(Y/Y_0)^{{\theta}}\big)\,S(Q,Y_0,N,0)\\
& \ll \Big(1+ Y^{\theta}/\min\big((Q^{1-\eps}/N),\,(Q^{2-\eps}/N)^{\theta}\big)\Big)Q^{1+\eps} \|{\bf a}_N\|_2^2\\
& \ll \Big(1+ Y^{\theta}\max\big((N/Q^{1-\eps}),\,(N/Q^{2-\eps})^{\theta}\big)\Big)Q^{1+\eps} \|{\bf a}_N\|_2^2\\
& \ \le \ (QN)^{2\eps}\,\big(Q+NY^{{\theta}}+Q(NY/Q^2)^{\theta} \big)\|{\bf a}_N\|_2^2.
\end{align*}
Hence it suffices to show \eqref{eq:thm6induct}. We shall prove \eqref{eq:thm6induct} for all $N\le Q^{1-\eps}$, by induction on $Q\ge1$. 

To this, if $Q \le Q_0(\eps)$ then the bound \eqref{eq:thm6thm2} implies \eqref{eq:thm6induct}, since
\begin{align*}
S(Q,Y_0,N,0) &\ll Y_0^{{\theta}}(Q+N^{1+\eps})\|{\bf a}_N\|_2^2\\
&\ll  (Q^{1-\eps}/N)(Q+N^{1+\eps})\|{\bf a}_N\|_2^2 \ll_\eps Q^{1+\eps} \|{\bf a}_N\|_2^2.
\end{align*}

Now if $Q> Q_0(\eps)$ then we apply Lemma \ref{lem:DI8.1}, giving
\begin{align}\label{eq:Lem8.1induct}
S(Q,Y_0,N,0) &\ll \int_{\R}S(Q_1, Y_0, N,it)\frac{\dd{t}}{t^4+1} \ + \ (NY_0)^{\eps/3} (Q+Q_1+N)\|{\bf a}_N\|_2^2 \nonumber\\
& \ll \sup_{t\in\R} S(Q_1,Y_0,N,it) + Q^{1+\eps} \|{\bf a}_N\|_2^2,
\end{align}
where $Q_1 := \pi NY_0/Q$. Note $Q_1 \ll N(Q^{2-\eps}/N)/Q \ll Q^{1-\eps}$. In particular $Q_1<Q-1$ provided the constant $Q_0(\eps)$ is sufficiently large. Next, if $N> Q_1^{1-\eps}$ we obtain \eqref{eq:thm6induct}, since by \eqref{eq:thm6thm2} and $Y_0^{\theta} \le Q^{1-\eps}/N$,
\begin{align*}
S(Q_1,Y_0,N,0) & \ll Y_0^{{\theta}}(Q_1+N^{1+\eps})\|{\bf a}_N\|_2^2 \ll (Q^{1-\eps}/N)N^{1+2\eps} \|{\bf a}_N\|_2^2 \ll Q^{1+\eps}\|{\bf a}_N\|_2^2.
\end{align*}
Plugging back into \eqref{eq:Lem8.1induct} gives $S(Q,Y,N,0) \ll Q^{1+\eps} \|{\bf a}_N\|_2^2$.
Thus it remains to assume $N\le Q_1^{1-\eps}$. Let $Y_1 := \min\big((Q_1^{1-\eps}/N)^{1/{\theta}},\,Q_1^{2-\eps}/N\big)$. 

If $Y_1 = (Q_1^{1-\eps}/N)^{1/{\theta}}$ then $(Y_0/Y_1)^{\theta} \le (Q/Q_1)^{1-\eps}$. Also if $Y_0=Q_1^{2-\eps}/N$, then
\begin{align*}
(Y_0/Y_1)^{\theta} \le (Q/Q_1)^{(2-\eps){\theta}} < (Q/Q_1)^{1-\eps/2},
\end{align*}
since ${\theta}\le 1/2$. Thus since $Q_1<Q-1$ and $N\le Q_1^{1-\eps}$, by the induction hypothesis (now with coefficients $a_n\, n^{it}$),
\begin{align}\label{eq:thm6induct2}
S(Q_1, Y_0,N,it) &\le (Y_0/Y_1)^{\theta}\,S(Q_1, Y_1,N,it) \nonumber\\
& \ll_\eps (Q/Q_1)^{1-\eps/2}\, Q_1^{1+\eps}\|{\bf a}_N\|_2^2
= Q^{1-\eps/2}Q_1^{3\eps/2}\|{\bf a}_N\|_2^2 \nonumber\\
& \ \ll \ Q^{1+\eps}\,\|{\bf a}_N\|_2^2.
\end{align}
Hence plugging \eqref{eq:thm6induct2} back into \eqref{eq:Lem8.1induct} gives $S(Q,Y,N,0) \ll Q^{1+\eps} \|{\bf a}_N\|_2^2$. This completes the proof of \eqref{eq:thm6induct}, and hence Theorem \ref{thm:DI6}.
\end{proof}

We use the following auxiliary estimate for Kloosterman sums.

\begin{theorem}\label{thm:DI14}
For $C,M,N\ge1$, 
we have
\begin{align}
\sum_{c\le C}\bigg|\sum_{m\le M}\sum_{n\le N}S(m,n;c) 
\bigg| \ \ll \ (CMN)^\eps (C^2+CMN).
\end{align}
\end{theorem}
\begin{proof}
This is \cite[Theorem 14]{DI}. 
\end{proof}

Now, using Theorem \ref{thm:DI14}, we improve upon Theorem \ref{thm:DI6} for the special sequence $a_n=1$.

\begin{theorem}\label{thm:DI7}
Let $Q,X\ge1$ and $1\le N<N_1\le 2N$. Then we have
\begin{align}
\sum_{q\asymp Q}\sum_{\lambda_j<1/4}^{(q)} X^{4i\kappa_j} \bigg|\sum_{N<n\le N_1} \rho_{j\infty}(n)\bigg|^2 \ \ll_\eps \ (QN)^\eps\, \Big(1+\Big(\frac{NX^2}{(Q+N)^2}\Big)^{\theta}\Big)(Q+N)N.
\end{align}
\end{theorem}
\begin{proof}
Define the sequence $a_n=\mathbf{1}_{[N,N_1]}(n)$. It suffices to show that for $Y\ge1$,
\begin{align}\label{eq:thm7goal}
S(Q,Y,N,0) := \sum_{q\asymp Q}\sum_{\lambda_j<1/4}^{(q)} Y^{2i\kappa_j} \bigg|\sum_{N<n\le N_1} \rho_{j\infty}(n)\bigg|^2 \ll (QN)^{5\eps}\Big(1+\Big(\frac{NY}{(Q+N)^2}\Big)^{\theta}\Big)(Q+N)N.
\end{align}

Indeed, by the Kuznetsov formula for the cusps $\a=\b=\infty$ of the group $\Gamma_0(q)$, summing over $m,n\sim N$ gives
\begin{align}\label{eq:Kuzinfty}
\sum_{\lambda_j< \tfrac{1}{4}}^{(q)} \frac{\widehat{\phi}(\kappa_j)}{\cosh\pi\kappa_j} \bigg|\sum_{n\sim N} \rho_{j\infty}(n)\bigg|^2& = 
\sum_{m,n\sim N}\sum_{c\in \Z^+} \frac{1}{qc}\,\phi\Big(\frac{4\pi}{qc}\sqrt{mn}\Big)\,S(m,n;qc)\\
& \quad + O\big((1+N^{1+\eps}/q)N\big), \nonumber
\end{align}
bounding the regular spectra as in \eqref{eq:thm5Kuznet2}, by the spectral large sieve in Theorem \ref{thm:largesieve}.

Since $\cosh\pi\kappa \le 1$ and $\widehat{\phi}(\kappa) \gg Y^{2i\kappa}$ by \eqref{eq:8.3DI}, we see \eqref{eq:Kuzinfty} gives
\begin{align}\label{eq:StoTcal}
S(Q,Y,N,0) &\ll  \sum_{q\asymp Q}\sum_{\lambda_j<1/4}^{(q)} \frac{\widehat{\phi}(\kappa_j)}{\cosh\pi\kappa_j} \bigg|\sum_{n\sim N} \rho_{j\infty}(n)\bigg|^2 \ \ll \ \mathcal T + (Q+N^{1+\eps})N,
\end{align}
where
\begin{align*}
\mathcal T := \sum_{Q<q\le 16Q} \sum_{c\in\Z^+} \frac{1}{qc}\bigg|\sum_{m,n\sim N} \phi\Big(\frac{4\pi}{qc}\sqrt{mn}\Big)\,S(m,n;qc)\bigg|.
\end{align*}
Since supp$(\phi) \subset[1/2Y,5/2Y]$, we note $c$ runs over the interval $[C/40, 16C]$ with $C=\pi NY/Q$. 

Next, by Mellin inversion $\phi(x) = \frac{1}{2\pi}\int_{\R} \breve{\phi}(it) x^{-it}\dd{t}$ for the Mellin transform $\breve{\phi}(it) = \int_{\R^+} \phi(y)y^{it-1} \dd{y}$, which is bounded by $\breve{\phi}(it)\ll (1+t^4)^{-1}$.
\begin{align*}
\mathcal T &\ll \int_{\R}
\sum_{q\asymp Q} \sum_{c\asymp C} \frac{1}{qc}\bigg|\sum_{N\le m,n\le N_1}(4\pi\sqrt{mn}/qc)^{-it}S(m,n;qc)\bigg|\frac{\dd{t}}{1+t^4}\nonumber\\
&\ll (QC)^{\eps-1}\int_{\R}
\sum_{k\asymp QC} \bigg|\sum_{N\le m,n\le N_1}(mn)^{-it}S(m,n;k)\bigg|\frac{\dd{t}}{1+t^4}.
\end{align*}
merging $k=qc$ as a single variable, and using the divisor bound $\tau(k)\ll k^\eps$.

Using $m^{-it}=N_1 + it\int_m^{N_1} u^{-it-1}\dd{u}$, we obtain by Theorem \ref{thm:DI14},
\begin{align*}
\mathcal T &\ll (QC)^{\eps-1} \sup_{N\le M',N'\le N_1} 
\sum_{k\asymp QC} \bigg|\sum_{\substack{m\le M' \\ n\le N'}} S(m,n;k)\bigg| \nonumber\\
& \ll (QC)^{\eps-1} \sup_{N\le M',N'\le N_1} QC(QC+M'N') \ \ll \ (QCN)^{\eps}\,(QC+N^2).
\end{align*}
Plugging back into \eqref{eq:StoTcal} gives
\begin{align}\label{eq:thm7S}
S(Q,Y,N,0) &\ll \mathcal T + (Q+N^{1+\eps})N \nonumber\\
& \ll (QCN)^{\eps}\,(QC+N^2) + (Q+N^{1+\eps})N \ll (NY)^\eps (Q+N+Y)N,
\end{align}
recalling $QC=\pi NY$. When $Y\ge Q+N$, this bound is self-improving in the $Y$-aspect by the following trick: indeed, by \eqref{eq:thm7S} with $Y_1=Q+N$,
\begin{align}\label{eq:thm7S1}
S(Q,Y,N,0) & \ll \big(1+(Y/Y_1)^{{\theta}}\big)S(Q,Y_1,N,0) \nonumber\\
& \ll \big(1+(Y/Y_1)^{{\theta}}\big)(NY_1)^\eps (Q+N+Y_1)N \nonumber\\
& \ll (QN)^\eps (Q+N+Y^{\theta}(Q+N)^{1-{\theta}})N.
\end{align}
In particular, if $N>Q^{1-2\eps}$ then 
\begin{align*}
S(Q,Y,N,0) &\ll N^\eps (N+Y^{\theta}\,N^{1-{\theta}})N
 \ll N^{2+\eps} \Big(1+\Big(\frac{NY}{(Q+N)^2}\Big)^{\theta}\Big)
\end{align*}
as desired for \eqref{eq:thm7goal}.

It remains to consider $N\le Q^{1-2\eps}$. We shall eliminate the $Y^{\theta} Q^{1-{\theta}}$ term in \eqref{eq:thm7S1} to complete the proof for all $N$. To this, we prove that for all $1\le N\le Q$ with $Y=Q^{2-2\eps}/N$,
\begin{align}\label{eq:thm7induct}
\sup_{\substack{N_1\le 2N\\a_n=\mathbf{1}_{[N,N_1](n)}}}S(Q,Y,N,0) \ll Q^{1+4\eps}N.
\end{align}
Assuming this, for arbitrary $Y\ge1$ we conclude
\begin{align*}
S(Q,Y,N,0) &\le \big(1+ (Y/Y_0)^{\theta}\big)S(Q,Y_0,N,0)\\
&\le \big(1+ (Y/Y_0)^{\theta}\big)Q^{1+4\eps}N\\
& \ll Q^{1+5\eps}N \big(1+(NY/Q^2)^{\theta}\big)
\ \ll \ Q^{1+5\eps}N\Big(1+\Big(\frac{NY}{(Q+N)^2}\Big)^{\theta}\Big)
\end{align*}
as desired for \eqref{eq:thm7goal}, where $Y_0=Q^{2-2\eps}/N$. Hence it suffices to show \eqref{eq:thm7induct}.

We shall prove \eqref{eq:thm7induct} by induction on $Q\ge1$. If $Q\le Q_0(\eps)$ the above follows by the spectral large sieve. Namely, by \eqref{eq:thm6thm2} with $a_n=1$ (so $\|a_N\|_2^2=\sum_{n\sim N}1=N$),
\begin{align*}
S(Q,Y_0,N,0) &\le \sqrt{Y}S(Q,1,N,0) \ll \sqrt{Y_0}(Q+N^{1+\eps})\|a_N\|_2^2\\
& \ll \frac{Q^{1-\eps}}{\sqrt{N}}(Q+N^{1+\eps})N \ll Q_0\,Q^{1+4\eps}N.
\end{align*}
Also if $Q^{1-2\eps}<N\le Q$, then $N<Q^{1+4\eps}$, so \eqref{eq:thm7S1} with $Y_0=Q^{2-2\eps}/N<Q$ gives
\begin{align*}
S(Q,Y_0,N,0) \ll (QN)^\eps (Q+N+\sqrt{NY_0}+\sqrt{QY})N \ll Q^{1+4\eps}N.
\end{align*}

Now consider $Q> Q_0(\eps)$ and $1\le N\le Q^{1-2\eps}$. By Lemma \ref{lem:DI8.1} with $a_n=1$ (so $\|a_N\|_2^2=N$),
\begin{align}\label{eq:thm7lem8.1}
S(Q,Y_0,N,0) \ll \int_{\R} S(Q_1,Y_0,N,it)\, \frac{\dd{t}}{t^4+1} + Q^{1+3\eps}N
\end{align}
where $Q_1=\pi NY/Q=\pi Q^{1-2\eps}< Q-1$, provided the constant $Q_0(\eps)$ is sufficiently large. Moreover, writing $Y_1= Q_1^{2-2\eps}/N$ we have $\sqrt{Y_0/Y_1}=(Q/Q_1)^{1-\eps} \ll Q^{2\eps(1-\eps)}$, and so
\begin{align}\label{eq:thm7inductit}
S(Q_1,Y_0,N,it) &\le \sqrt{Y_0/Y_1}\,S(Q_1,Y_1,N,it) \nonumber\\
 &\le Q^{2\eps(1-\eps)}\,S(Q_1,Y_1,N,it).
\end{align}
We now wish to apply the induction hypothesis (with $a_n=1$), but this is not possible immediately, due to the coefficients $a_n=n^{it}$ in $S(Q_1,Y_1,N,it)$. 

To remedy this, by partial summation
\begin{align*}
\sum_{N<n\le N_1}n^{it}\,\rho_{j\infty}(n) = N_1^{it}\sum_{N<n\le N_1}\rho_{j\infty}(n) - it\int_N^{N_1}\sum_{N<n\le u}\rho_{j\infty}(n)\,u^{it-1}\dd{u}
\end{align*}
and so Cauchy-Schwarz gives
\begin{align*}
\bigg|\sum_{N< n\le N_1} n^{it}\rho_{j\infty}(n)\bigg|^2 &\ll \bigg|\sum_{N< n\le N_1}\rho_{j\infty}(n)\bigg|^2 + t^2\bigg|\int_N^{N_1}\sum_{N<n\le u}\rho_{j\infty}(n)\,u^{it-1}\dd{u}\bigg|^2\\
&\ll \bigg|\sum_{N< n\le N_1}\rho_{j\infty}(n)\bigg|^2 + t^2\int_N^{N_1}\bigg|\sum_{N<n\le u}\rho_{j\infty}(n)\bigg|^2\dd{u}\cdot \int_N^{N_1}\frac{\dd{u}}{u^2}\\
&\ll (1+t^2)\sup_{N_1\le 2N}\bigg|\sum_{N< n\le N_1}\rho_{j\infty}(n)\bigg|^2.
\end{align*}
Thus we have
\begin{align*}
S(Q_1,Y_1,N,it) & := \sum_{q\sim Q_1}\sum_{\lambda_j<1/4}^{(q)} Y_1^{2i\kappa_j} \bigg|\sum_{N< n\le N_1} n^{it}\rho_{j\infty}(n)\bigg|^2\\
& \ll (1+t^2)\sup_{N_1\le 2N}\sum_{q\sim Q_1}\sum_{\lambda_j<1/4}^{(q)} Y_1^{2i\kappa_j}\bigg|\sum_{N< n\le N_1}\rho_{j\infty}(n)\bigg|^2\\
& = (1+t^2)\sup_{\substack{N_1\le 2N\\a_n=\mathbf{1}_{[N,N_1]}(n)}}S(Q_1,Y_1,N,0)\\
& \ll (1+t^2)Q_1^{1+4\eps}N.
\end{align*}
by the induction hypothesis. Hence \eqref{eq:thm7inductit} becomes
\begin{align*}
S(Q_1,Y_0,N,it) \ll Q^{2\eps(1-\eps)}\,S(Q_1,Y_1,N,it) &\ll Q^{2\eps(1-\eps)}(1+t^2)Q_1^{1+4\eps}N
\end{align*}
Plugging back into \eqref{eq:thm7lem8.1} we obtain
\begin{align*}
S(Q,Y_0,N,0) & \ll_\eps Q^{2\eps(1-\eps)}Q_1^{1+4\eps}N\int_{\R}\frac{t^2+1}{t^4+1}\dd{t}+ Q^{1+3\eps}N\\
& \ll_\eps Q^{2\eps(1-\eps)}Q_1^{1+4\eps}N + Q^{1+3\eps}N\\
& \ll_\eps Q^{1+4\eps- 10\eps^2}N \ \ll \ Q^{1+4\eps}N.
\end{align*}
This completes the proof of \eqref{eq:thm7induct}, and hence Theorem \ref{thm:DI7}.
\end{proof}

\subsection{Factorization of Fourier coefficients}
In the proof we assume $q_0=1$ and $t=0$ for simplicity, though general results may obtained similarly with $q_0$-dependence, as in \cite[\S4.2.3]{Drapp2}. Indeed, $q_0=1$ is the only case needed for our applications in this article.

We use an approximate factorization of Fourier coefficients by Assing--Blomer--Li \cite{ABL}. 
\begin{lemma}\label{lem:ABLfactor}
For any $r,s\in\N$ and any sequence $b_{n}=b_{n,r,s}\in\C$, we have
\begin{align*}
&\bigg|\sum_{n\sim N} b_{n} \,\rho_{j,\,1/s}(an)\bigg| 
\ \ll \ 
\sum_{\substack{n''\mid a^\infty\\ n''\ll N}}
(an'')^{{\theta}/2}
\bigg|\sum_{\substack{n\sim N/n''\\(n,a)=1}} b_{nn''} \,\rho_{j,\,1/s}(n)\bigg|.
\end{align*}
Corresponding statements hold for holomorphic and Eisenstein contributions.    
\end{lemma}
\begin{proof}
This is a slight rephrasing of \cite[Lemma 3.2]{ABL} with $q=Q''=K=1$, and follows by the same proof. The basic idea is to expand with respect to an orthonormal basis of newforms, and then apply the approximate factorization of Fourier coefficients, in \cite[Lemma 3.1]{ABL}. 

This is the only place where the Ramanujan--Petersson exponent $\theta$ is needed, generalizing that of Selberg.
As noted, \cite{ABL} use different normalization, with $2\theta$ instead of $\theta$.
\end{proof}

\section{Sums of Kloosterman sums}

In this section, we estimate quintilinear sums of Kloosterman sums using the spectral large sieve estimates, and prove Theorem \ref{thm:DI12}.

\begin{theorem}\label{thm:DI10}
Let $C,M,N,R,S>0$ and let $g\in \mathcal C^\infty\big([C,2C]\times (\R^+)^4\big)$ satisfy
\begin{align}\label{eq:1.53}
\bigg|\frac{\partial^{v_1+v_2+v_3+v_4+v_5}}{\partial c^{v_1}\,\partial m^{v_2}\,\partial n^{v_3}\,\partial r^{v_4}\,\partial s^{v_5}}\, g(c,m,n,r,s)\bigg| \ \ll_{v_i} \ c^{-v_1}m^{-v_2}n^{-v_3}r^{-v_4}s^{-v_5}
\end{align}
for any $v_i\ge0$, $i\le 5$. For a Dirichlet character $\chi$ (mod $q_0$), $t\in\R$, sequences $a_m, b_{n,r,s}\subset \C$, define the quintilinear sum $\mathcal L^\pm_{(a_m)}=\mathcal L^\pm_{(a_m)}(C,M,N,R,S)$ by
\begin{align}\label{eq:quintL}
\mathcal L^\pm_{(a_m)} = \sum_{\substack{r\sim R\\s\sim S\\(q_0r,s)=1}}\sum_{\substack{m\sim M\\n\sim N}} a_m\, b_{n,r,s} \,\overline{\chi}(c) \sum_{(c,q_0s)=1} g(c,m,n,r,s)\, e(mt)\, S(m\overline{r}, \pm an; sc).
\end{align}
Then we have
\begin{align}
\mathcal L_{(a_m)} \ \ll_\eps \ (q_0 CMNRS)^\eps \, L \,\|{\bf a}\|_2
\end{align}
where $L = L(C,M,N,R,S)$ is given by
\begin{align*}
L^2 = \frac{C^4S^3q_0 R}{aMN+C^2S^2q_0R}& \Bigg(\sum_{\substack{n''\mid a^\infty\\ n''\ll N}} (an'')^{{\theta}} \Big(\frac{q_0 N}{n''}+RS+\frac{aMN}{q_0 C^2S}\Big)\Big(M+RS+\frac{aMN}{q_0 C^2S}\Big)\|\widetilde{\bf b}(n'')\|_2^2 \\
& \qquad \ + \ (CS\sqrt{q_0 R})^{2{\theta}} (q_0 N+RS)^{1-{\theta}} \big(M^{1-{\theta}}+(RS)^{1-2{\theta}}\big)  
\|{\bf b}\|_2^2\Bigg).
\end{align*}
\end{theorem}

The bound for $\mathcal L_{(a_m)}$ is refined when $a_m$ is the characteristic sequence of an interval.

\begin{theorem}\label{thm:DI11}
Let $C,M,N,R,S>1$ and let $g\in \mathcal C^\infty\big([C,2C]\times (\R^+)^4\big)$ satisfy \eqref{eq:1.53}. For a Dirichlet character $\chi$ (mod $q_0$), $t\in\R$,  any sequence $b_{n,r,s}\subset \C$, we have
\begin{align}
\mathcal L_{(\mathbf{1}_{m\sim M})} \ \ll_\eps \ (q_0 CMNRS)^\eps \, L_1 \,\sqrt{M}
\end{align}
for $\mathcal L_{(a_m)}$ in \eqref{eq:quintL} with the sequence $a_m=\mathbf{1}_{m\sim M}$, and where $L_1 = L_1(C,M,N,R,S)$ is
\begin{align*}
L_1^2 = \frac{C^4S^3q_0 R}{aMN+C^2S^2q_0 R} & \bigg(\sum_{\substack{n''\mid a^\infty\\ n''\ll N}} (an'')^{{\theta}} \Big(\frac{q_0 N}{n''} + RS+\frac{aMN}{q_0 C^2S}\Big) \Big(M+RS+\frac{aMN}{q_0 C^2S}\Big)\|\widetilde{\bf b}(n'')\|_2^2 \nonumber\\
&\qquad +\, (CS\sqrt{q_0 R})^{2{\theta}} (q_0 N + RS)^{1-{\theta}} (M+RS)^{1-2{\theta}}\|{\bf b}\|_2^2\Bigg).
\end{align*}
\end{theorem}
Note $L$ differs from $L_1$ only in the term $M^{1-{\theta}}$, compared to $M^{1-2{\theta}}$, in the final factor for the exceptional spectrum.

For simplicity of exposition, we assume $q_0=1$ and $t=0$. Indeed, $q_0=1$, $t=0$ is the only case needed for our applications in this article. Results may obtained with general $q_0$-dependence, as in \cite[\S4.2.3]{Drapp2}.

As usual we restrict ourselves to weights $g$ of the form
\begin{align*}
g(c,m,n,r,s) = \frac{CS\sqrt{R}}{cs\sqrt{r}}\,f\Big(\frac{4\pi\sqrt{amn}}{cs\sqrt{r}}\Big)
\end{align*}
where $f\in \mathcal C^\infty([1/X,2/X])$ for $X = CS\sqrt{R}/4\pi\sqrt{aMN}$, and whose derivatives satisfy $|f^{(l)}(x)|\ll X^{l}$ for $l\ge0$.
Proofs of the general forms of these results may be deduced by adapting the techniques in section 7, using the relevant transforms of smooth weight functions (separation of variables)

\begin{proof}[Proofs of Theorem \ref{thm:DI10} and \ref{thm:DI11}]
Consider coprime $r,s\in\Z$. Recall that
\begin{align*}
S_{\infty, 1/s}(m,an;\gamma) = e\Big(\frac{an\overline{s}}{r}\Big)S(m\overline{r},an;sc)
\end{align*}
where $\gamma=cs\sqrt{r}$ for $(c,r)=1$. Thus letting $b'_{n,r,s}=e(-an\overline{s}/r)b_{n,r,s}$, we have
\begin{align*}
\mathcal L = CS\sqrt{R}\sum_{\substack{r\sim R\\s\sim S\\(r,s)=1}}\sum_{\substack{m\sim M\\n\sim N}} a_m\, b_{n,r,s}' \sum_{\substack{\gamma\in\R\\\exists\smqty(*&*\\\gamma&*)\in \Gamma_0(rs)\sigma_{1/s}}} \frac{1}{\gamma}\,f\Big(\frac{4\pi\sqrt{amn}}{\gamma}\Big)\, S_{\infty,\,1/s}(m, \pm an; \gamma).
\end{align*}
where the innermost summation is taken over numbers $\gamma=cs\sqrt{r}$, which are the left corner entries of matrices from $\Gamma_0(rs)\sigma_{1/s}$. For simplicity we suppose $n>0$ ($n<0$ is similar). Next for the inner sum we apply the Kuznetsov formula, giving
\begin{align}\label{eq:LKuznet}
\mathcal L &= \mathcal H + \mathcal E + \mathcal M
\end{align}
where
\begin{align*}
\mathcal H &= CS\sqrt{R}\sum_{\substack{r\sim R\\s\sim S\\(r,s)=1}}\sum_{\substack{m\sim M\\n\sim N}} a_m\, b_{n,r,s}' \sum_{{\rm even}\,k}\frac{\widetilde{f}(k-1)}{2\pi}\sum_{1\le j\le \theta_k(q)} \frac{i^k(k-1)!}{(4\pi\sqrt{mn})^{k-1}}\, \overline{\psi_{jk}(\infty,m)}\psi_{jk}(1/s,an)\nonumber\\
\mathcal E  & = CS\sqrt{R}\sum_{\substack{r\sim R\\s\sim S\\(r,s)=1}}\sum_{\substack{m\sim M\\n\sim N}} a_m\, b_{n,r,s}' \sum_{\c}\frac{1}{\pi}\int_\R \Big(\frac{n}{m}\Big)^{ir}\overline{\phi_{\c\infty m}(\tfrac{1}{2}+ir)} \phi_{\c\,1/s\, an}(\tfrac{1}{2}+ir)\widehat{f}(r)\dd{r}\\
\mathcal M &= CS\sqrt{R}\sum_{\substack{r\sim R\\s\sim S\\(r,s)=1}}\sum_{\substack{m\sim M\\n\sim N}} a_m\, b'_{n,r,s} \sum_{j\ge1}\overline{\rho_{j\infty}(m)}\rho_{j\,1/s}(an) \frac{\widehat{f}(\kappa_j)}{\cosh\pi\kappa_j}.
\end{align*}
In the following argument, we shall focus on the Maass contribution $\mathcal M$, and split $\mathcal M=\mathcal M_{\rm reg}+\mathcal M_{\rm exc}$ into regular and exceptional spectra.  (the holomorphic $\mathcal H$ and Eisenstein $\mathcal E$ contributions may be handled similarly to $\mathcal M_{\rm reg}$).

For the Maass contribution $\mathcal M$, we first apply the factorization as in Lemma \ref{lem:ABLfactor}, giving
\begin{align}\label{eq:MaassfactorABL}
\mathcal M & \ll CS\sqrt{R}\sum_{\substack{n''\mid a^\infty\\ n''\ll N}} (an'')^{{\theta}} \sum_{\substack{r\sim R\\s\sim S\\(r,s)=1}} \sum_{\lambda_j<\tfrac{1}{4}}^{(rs)} \frac{|\widehat{f}(\kappa_j)|}{\cosh\pi\kappa_j}\bigg|\sum_{\substack{m\sim M}} a_m \overline{\rho_{j\infty}(m)}\bigg| \bigg|\sum_{\substack{n\sim N/n''\\(n,a)=1}} b_{nn'',r,s} \, \rho_{j\,1/s}(n)\bigg|.
\end{align}
We split $\mathcal M=\mathcal M_{\rm reg}+\mathcal M_{\rm exc}$ according to whether 
$\lambda_j\ge\frac{1}{4}$ or $\lambda_j<\frac{1}{4}$, and consider each in turn. 
For the regular spectrum $\mathcal M_{\rm reg}$, we use the bound $\widehat{f}(\kappa_j)\ll \frac{1+|\log X|}{1+1/X}$ by Lemma \ref{lem:DI7.1}, for $X := |{\rm supp} f| \asymp CS\sqrt{R/aMN}$. Thus by Cauchy-Schwarz,
\begin{align*}
\mathcal M_{\rm reg}& \ll CS\sqrt{R}\, \frac{1+|\log X|}{1+1/X} \bigg(\sum_{\substack{r\sim R\\s\sim S\\(r,s)=1}} \sum_{\lambda_j\ge\tfrac{1}{4}}^{(rs)} \frac{1}{\cosh\pi\kappa_j} \bigg| \sum_{\substack{m\sim M}} a_m \overline{\rho_{j\infty}(m)} \bigg|^2\bigg)^{1/2}\\
&\qquad\qquad \cdot\sum_{\substack{n''\mid a^\infty\\ n''\ll N}} (an'')^{{\theta}/2} \bigg(\sum_{\substack{r\sim R\\s\sim S\\(r,s)=1}}\sum_{\lambda_j\ge\tfrac{1}{4}}^{(rs)} \frac{1}{\cosh\pi\kappa_j} \bigg| \sum_{\substack{n\sim N/n''\\(n,a)=1}} b_{nn'',r,s} \rho_{j\,1/s}(n) \bigg|^2\bigg)^{1/2}
\end{align*}

For each $r,s$, we bound the regular spectra by Theorem \ref{thm:largesieve} with $\mu(\infty)=\mu(1/s)=1/rs$, giving
\begin{align*}
\mathcal M_{\rm reg} & \ll (CMNRS)^\eps CS\sqrt{R}\, \frac{1+|\log X|}{1+1/X}\bigg(\sum_{\substack{r\sim R\\s\sim S\\(r,s)=1}} \Big(1+\frac{1}{X^2 } + \frac{M^{1+\eps}}{rs}\big) \|{\bf a}_M\|_2^2\bigg)^{1/2}\\
&\qquad \cdot \sum_{\substack{n''\mid a^\infty\\ n''\ll N}} (an'')^{{\theta}/2} \bigg(\sum_{\substack{r\sim R\\s\sim S\\(r,s)=1}} \Big(1+\frac{1}{X^2} + \frac{(N/n'')^{1+\eps}}{rs} \Big)\|\widetilde{\bf b}_{N,r,s}(n'')\|_2^2\bigg)^{1/2}\\
& \ll (CMNRS)^\eps CS\sqrt{R}\, \frac{1+|\log X|}{1+1/X} \Big(1+\frac{1}{X^2}+ \frac{M}{RS}\Big)^{1/2} \sqrt{RS}\|{\bf a}_M\|_2\\
&\qquad \cdot \sum_{\substack{n''\mid a^\infty\\ n''\ll N}}(an'')^{{\theta}/2} \bigg(\Big(1+\frac{1}{X^2} + \frac{N}{n''RS}\Big)^{1/2}\|\widetilde{\bf b}_{N,R,S}(n'')\|_2
\end{align*}
by the divisor bound, and noting $\sum_{r,s}\|{\bf b}_{N,r,s}\|_2^2 =\sum_{n,r,s} |{\bf b}_{n,r,s}|^2 = \|{\bf b}_{N,R,S}\|_2^2$.
Recalling $X^2 \asymp C^2S^2 R/aMN$, we have
\begin{align}\label{eq:thm10R2}
\mathcal M_{\rm reg} 
& \ll (CMNRS)^\eps \frac{CS\sqrt{R}}{1+1/CS\sqrt{R/aMN}} \Big(1+ \frac{aMN}{C^2S^2 R} + \frac{M}{RS}\Big)^{1/2} \sqrt{RS}\|{\bf a}_M\|_2 \nonumber\\
&\qquad \cdot \sum_{\substack{n''\mid a^\infty\\ n''\ll N}} (an'')^{{\theta}/2} \Big(1 + \frac{aMN}{C^2S^2 R} + \frac{Q''N}{n''RS}\Big)^{1/2} 
\|\widetilde{\bf b}_{N,R,S}(n'')\|_2
= : J_{\rm reg}. \nonumber
\end{align}

The same argument holds for holomorphic $\mathcal H$ and Eisenstein $\mathcal E$ contributions, so that
\begin{align}
\mathcal H + \mathcal E + \mathcal M_{\rm reg} \ \ll \ J_{\rm reg}.
\end{align}

It is remains to handle the exceptional spectra $\mathcal M_{\rm exc}$. Recall $|\widehat{f}(\kappa_j)| \ll \frac{1+X^{2i\kappa_j}}{1+1/X}$ by \eqref{eq:7.1} in Lemma \ref{lem:DI7.1} with $r=i\kappa_j\in (0,1/2)$. This gives
\begin{align*}
\mathcal M_{\rm exc} & \ll\sum_{\substack{n''\mid a^\infty\\ n''\ll N}} (an'')^{{\theta}/2} \frac{CS\sqrt{R}}{1+1/X}\sum_{\substack{r\sim R\\s\sim S\\(r,s)=1}} \sum_{\lambda_j<\tfrac{1}{4}}^{(rs)} \frac{1+X^{2i\kappa_j}}{\cosh\pi\kappa_j}\bigg|\sum_{m\sim M} a_m \overline{\rho_{j\infty}(m)}\bigg|\bigg|\sum_{\substack{n\sim N/n''\\(n,a)=1}} b_{nn'',r,s}\rho_{j\,1/s}(n)\bigg| 
\end{align*}
For $X_0$ to be determined, we split $1+X^{2i\kappa_j} \le (1+X/\sqrt{X_0})^{2i\kappa_j}\, X_0^{i\kappa_j}$ and apply Cauchy-Schwarz,
\begin{align}
\mathcal M_{\rm exc} & \ll \frac{CS\sqrt{R}}{1+1/X}\sum_{\substack{n''\mid a^\infty\\ n''\ll N}} (an'')^{{\theta}/2} \bigg(\sum_{\substack{r\sim R\\s\sim S\\(r,s)=1}} \sum_{\lambda_j<\tfrac{1}{4}}^{(rs)}\big(1+X/\sqrt{X_0}\big)^{4i\kappa_j}\bigg|\sum_{m\sim M}a_m\,\overline{\rho_{j\infty}(m)}\bigg|^2\bigg)^{1/2} \label{eq:thm10M}\\
&\qquad\qquad\qquad\qquad\quad\quad \ \cdot \bigg(\sum_{\substack{r\sim R\\s\sim S\\(r,s)=1}} \sum_{\lambda_j<\tfrac{1}{4}}^{(rs)} (1+X_0)^{2i\kappa_j} \bigg|\sum_{\substack{n\sim N/n''\\(n,a)=1}} \,b_{nn'',r,s}\,\rho_{j,1/s}(n)\bigg|^2\bigg)^{1/2}.\label{eq:thm10N}
\end{align}
We shall choose $X_0=1+RSn''/N$.

Note the sum in \eqref{eq:thm10M} over $a_{m}\,\rho_{j\,\infty}(m)$ equals $S(RS,(1+X/\sqrt{X_0})^2,M,0)$.

For the sum in \eqref{eq:thm10N}, we apply Theorem \ref{thm:DI5} (with `$\a$'$=1/s$, `$X$'$=1+X_0$, `$a_n$'$=b_{n,r,s}^*$) to get
\begin{align*}
\sum_{\lambda_j<\tfrac{1}{4}}^{(rs)} X_0^{2i\kappa_j} \bigg|\sum_{\substack{n\sim N/n''\\(n,a)=1}} \,b_{nn'',r,s}\,\overline{\rho_{j,1/s}(n}&) \bigg|^2
 \ll \big(1+ X_0^{{\theta}}\big)\Big(1+ \Big(\frac{(N/n'')^{1+\eps}}{rs}\Big)^{1-{\theta}}\Big) 
\|\widetilde{\bf b}_{N,r,s}(n'')\|_2^2\\
& \ll N^{\eps}\Big(1+ \Big(\frac{N}{n''RS}\Big)^{{\theta}}\Big)\Big(1+ \Big(\frac{N}{n''RS}\Big)^{1-{\theta}}\Big) \|{\bf b}^*_{N,r,s}(n'')\|_2^2\\
& \ll \ N^{\eps}\Big(1+ \frac{N}{n''RS}\Big) \|\widetilde{\bf b}_{N,r,s}(n'')\|_2^2.
\end{align*}
Plugging this into \eqref{eq:thm10N}, and noting $\sum_{r,s}\|{\bf b}_{N,r,s}\|_2^2 = \sum_{n,r,s}|b_{n,r,s}|^2= \|{\bf b}_{N,R,S}\|_2^2$, we obtain
\begin{align}\label{eq:thm10E2}
\mathcal M_{\rm exc} & \ll \sum_{\substack{n''\mid a^\infty\\ n''\ll N}} (an'')^{{\theta}/2} \,S\big(RS,(1+X/\sqrt{X_0})^2,M,0\big)^{1/2}\Big(1+ \frac{N}{n''RS}\Big)^{1/2} \|\widetilde{\bf b}_{N,R,S}(n'')\|_2 \,\times \nonumber\\
& \qquad \times (CMNRS)^\eps\frac{CS\sqrt{R}}{1+1/X}. 
\end{align}

To prove Theorem \ref{thm:DI10}, we apply Theorem \ref{thm:DI6} with $Q=RS$, `$X$'$=1+X/\sqrt{X_0}$,
\begin{align}
S_{(a_m)} \big(&RS, (1+X/\sqrt{X_0})^2,M,0\big) \nonumber\\
&\ \ll \
(MRS)^{\eps}\,\big(RS+M+ M(1+X/\sqrt{X_0})^{2{\theta}}+RS\big((1+X/\sqrt{X_0})\frac{\sqrt{M}}{RS}\big)^{2{\theta}}\big) \|{\bf a}_M\|_2^2 \nonumber\\
&\ \ll \
(MRS)^{\eps}\,\Big(RS+M+ (1+X^2/X_0)^{{\theta}}\Big(M+M^{\theta}\, (RS)^{1-2{\theta}}\Big)\Big) \|{\bf a}_M\|_2^2 \nonumber\\
&\ \ll \
(MRS)^{\eps}\,\Big(RS+M+ \Big(1+\frac{C^2S^2 R}{aM(N+RSn'')}\Big)^{{\theta}}\Big(M+M^{\theta}\, (RS)^{1-2{\theta}}\Big)\Big) \|{\bf a}_M\|_2^2 \nonumber\\
& \ \ll \ (MRS)^{\eps}\,\Big(RS+M+ \Big(\frac{C^2S^2 R}{a(N+RSn'')}\Big)^{{\theta}} \big(M^{1-{\theta}}+(RS)^{1-2{\theta}}\big)\Big) \|{\bf a}_M\|_2^2
\end{align}
recalling $X^2 \asymp C^2S^2 R/aMN$, $X_0=1+RSn''/N$ (so $X^2/X_0\asymp C^2S^2 R/aM(N+RSn'')$ along with $1/(1+1/X)\asymp CS\sqrt{R}/(CS\sqrt{R} + \sqrt{aMN})$).
Thus \eqref{eq:thm10E2} becomes
\begin{align*}
\mathcal M_{\rm exc} & \ll  \frac{C^2S^2R (CMNRS)^\eps}{\sqrt{aMN}+CS\sqrt{R}}\,\|{\bf a}_M\|_2\sum_{\substack{n''\mid a^\infty\\ n''\ll N}} (an'')^{{\theta}/2} \|\widetilde{\bf b}_{N,R,S}(n'')\|_2 \ \times\\
& \qquad \times\bigg(RS+M+ (an'')^{-{\theta}}\Big(\frac{C^2S^2 R}{N/n''+RS}\Big)^{{\theta}} \big(M^{1-{\theta}} + (RS)^{1-2{\theta}} \big) \bigg)^{\frac{1}{2}}  \Big(1+ \frac{N}{n''RS}\Big)^{\frac{1}{2}}
\end{align*}
For the third term in the sum over $n''\mid a^\infty$, we apply $\|\widetilde{\bf b}_{N,R,S}(n'')\|_2\le \|{\bf b}_{N,R,S}\|_2$, $\theta\le \theta$, and the divisor bound, giving
\begin{align}\label{eq:thm10E3}
\mathcal M_{\rm exc} & \ll \frac{C^2S\sqrt{RS}(CMNRS)^\eps}{\sqrt{aMN}+CS\sqrt{R}}\,\|{\bf a}_M\|_2
\Big( \sum_{\substack{n''\mid a^\infty\\ n''\ll N}} (an'')^{{\theta}} (M+RS) \Big(\frac{N}{n''}+RS\Big)\|\widetilde{\bf b}_{N,R,S}(n'')\|_2^2 \nonumber\\
&\qquad\qquad + \ (CS\sqrt{R})^{2{\theta}} (M^{1-{\theta}} + (RS)^{1-2{\theta}})(N+RS)^{1-{\theta}}\|{\bf b}_{N,R,S}\|_2^2\Big)^{\frac{1}{2}} \ =: J_{\rm exc},
\end{align}
Combining \eqref{eq:thm10E3} and \eqref{eq:thm10R2} (we factor out~$C\sqrt{S}$ from each parenthesis) gives
\begin{align*}
\mathcal L & = \mathcal H + \mathcal E + \mathcal M_{\rm reg} + \mathcal M_{\rm exc} \ll J_{\rm reg} + J_{\rm exc}\\
&\ll \frac{ C^2S \sqrt{RS}(CMNRS)^\eps}{\sqrt{aMN}+CS\sqrt{R}}\, \|{\bf a}_M\|_2\Bigg(\sum_{\substack{n''\mid a^\infty\\ n''\ll N}} (an'')^{{\theta}}\|\widetilde{\bf b}_{N,R,S}(n'')\|_2^2 \; \times \\
&\qquad\qquad\times\bigg\{\Big(\frac{N}{n''}+RS + \frac{aMN}{C^2S}\Big) \Big(M+RS + \frac{aMN}{C^2S}\Big)+ \Big(\frac{N}{n''}+RS\Big) (M+RS)\bigg\} \\
&\qquad\qquad + (CS\sqrt{R})^{2{\theta}} (N+RS)^{1-{\theta}} \big(M^{1-{\theta}} + (RS)^{1-2{\theta}}\big)\|{\bf b}_{N,R,S}\|_2^2\Bigg)^{\frac{1}{2}} \nonumber
\end{align*}
The term `$+\,(N/n''+RS)(M+RS)\big\}$' in the sum above may be absorbed into the contribution of $J_{\rm reg}$, giving
\begin{align}
\mathcal L &\ll (CMNRS)^\eps\frac{C^2S\sqrt{RS}\|{\bf a}_M\|_2}{\sqrt{aMN}+CS\sqrt{R}}  \nonumber\\
& \cdot  \Bigg( \sum_{\substack{n''\mid a^\infty\\ n''\ll N}} (an'')^{{\theta}} \Big(\frac{N}{n''}+RS+\frac{aMN}{C^2S}\Big)\Big(M+RS+\frac{aMN}{C^2S}\Big)\|\widetilde{\bf b}_{N,R,S}(n'')\|_2^2 \\
& \qquad\qquad\qquad + \ (CS\sqrt{R})^{2{\theta}} (N+RS)^{1-{\theta}} \big(M^{1-{\theta}}+(RS)^{1-2{\theta}}\big)\|{\bf b}_{N,R,S}\|_2^2  \Bigg)^{\frac{1}{2}}. \nonumber
\end{align}
This completes the proof of Theorem \ref{thm:DI10}.

To prove Theorem \ref{thm:DI11}, we apply Theorem \ref{thm:DI7} with $Q=RS$, `$X$'$=1+X/\sqrt{X_0}$,
\begin{align*}
S_{(\mathbf{1}_{m\sim M})}\big(RS, &(1+X/\sqrt{X_0})^2,M,0\big)\\
&\ \ll \
(MRS)^\eps\, \Big(1+\Big(\frac{M(1+X/\sqrt{X_0})^2}{(RS+M)^2}\Big)^{\theta}\Big)(RS+M)M\\
&\ \ll \
(MRS)^\eps\, \Big(1+(an'')^{-{\theta}}\Big(\frac{C^2S^2 R}{(RS+M)^2(N/n''+RS)}\Big)^{\theta}\Big)(RS+M)M
\end{align*}

So plugging back into \eqref{eq:thm10E2} gives
\begin{align}
\mathcal M_{\rm exc}  &\ll \frac{CS\sqrt{R}}{1+1/X}\sum_{\substack{n''\mid a^\infty\\ n''\ll N}} (an'')^{{\theta}/2} \,S(RS,(1+X/\sqrt{X_0})^2,M,0)^{\frac{1}{2}} \Big(1+ \Big(\frac{N}{n''RS}\Big)\Big)^{\frac{1}{2}}\|\widetilde{\bf b}_{N,R,S}(n'')\|_2 \nonumber\\
& \qquad \quad \ll \sum_{\substack{n''\mid a^\infty\\ n''\ll N}}\,\Big((an'')^{{\theta}/2} + \Big(\frac{C^2S^2 R}{(RS+M)^2(N/n''+RS)}\Big)^{{\theta}/2}\Big)\sqrt{N/n''+RS}\\
& \qquad\qquad\qquad\cdot\frac{C^2S\sqrt{RS}}{\sqrt{aMN}+CS\sqrt{R}}  \sqrt{M+RS}\sqrt{M} \|\widetilde{\bf b}_{N,R,S}(n'')\|_2. \nonumber
\end{align}
Thus combining with \eqref{eq:thm10R2}, we obtain
\begin{align*}
\mathcal L & = \mathcal H + \mathcal E + \mathcal M_{\rm reg} + \mathcal M_{\rm exc} \ll J_{\rm reg} + J_{\rm exc}\\
&\ll (CMNRS)^\eps\frac{C^2 S \sqrt{RS}}{\sqrt{aMN}+CS\sqrt{R}}\,\sqrt{M}\times \nonumber\\
& \times\sum_{\substack{n''\mid a^\infty\\ n''\ll N}}\Bigg((an'')^{{\theta}} \Big(\frac{N}{n''} + RS+\frac{aMN}{C^2S}\Big) \Big(M+RS+\frac{aMN}{C^2S}\Big) \|\widetilde{\bf b}_{N,R,S}(n'')\|_2^2 \nonumber\\
&\qquad +\,\Big((an'')^{{\theta}}+\Big(\frac{C^2S^2 R}{(RS+M)^2(N/n''+RS)}\Big)^{{\theta}}\Big) (N/n'' + RS) (M+RS) \|\widetilde{\bf b}_{N,R,S}(n'')\|_2^2\Bigg)^{\tfrac{1}{2}}
\end{align*}
The term `$(an'')^{{\theta}}\,+$' in the factor $\big((an'')^{{\theta}}+(C^2S^2 R/\cdots)^{{\theta}} \big)$ may be absorbed into the contribution of $J_{\rm reg}$. Hence applying $\|\widetilde{\bf b}_{N,R,S}(n'')\|_2\le \|{\bf b}_{N,R,S}\|_2$ and the divisor bound, we obtain
\begin{align*}
\mathcal L
&\ll (CMNRS)^\eps\frac{C^2 S \sqrt{RS}}{\sqrt{aMN}+CS\sqrt{R}}\,\sqrt{M}\times \nonumber\\
& \times\Bigg(\sum_{\substack{n''\mid a^\infty\\ n''\ll N}} (an'')^{{\theta}} \Big(\frac{N}{n''} + RS+\frac{aMN}{C^2S}\Big) \Big(M+RS+\frac{aMN}{C^2S}\Big) \|\widetilde{\bf b}_{N,R,S}(n'')\|_2^2 \nonumber\\
&\qquad\qquad +\,(C^2S^2 R)^{{\theta}} (N + RS)^{1-{\theta}} (M+RS)^{1-2{\theta}}\|{\bf b}_{N,R,S}\|_2^2\Bigg)^{\tfrac{1}{2}}
\end{align*}
This completes the proof of Theorem \ref{thm:DI11}.
\end{proof}

\begin{proof}[Proof of Theorem \ref{thm:DI12} from Theorem \ref{thm:DI11}]
Standard completion of sums argument, using Poisson summation.
See \cite[Theorem 2.1]{Drapp2} or \cite[Theorem 12]{DI}.
\end{proof}

\appendix

\section{Optimality of Theorems \ref{thm:DI6} and \ref{thm:DI7}} \label{appndx:optimal}
\begin{center}
\sc by Sary Drappeau and Jared Duker Lichtman
\end{center}
\vspace{1em}

In this appendix, we give a heuristic argument to illustrate the bounds given in Theorems \ref{thm:DI6} and \ref{thm:DI7}.
We hope these heuristics shed some light on the proofs of Theorems \ref{thm:DI6} and \ref{thm:DI7}. The actual argument we describe is slightly different, but we feel it better motivates the steps in the original arguments of~\cite{DI}, and explains the shape of the final bound.

Let $Q,N,Y>0$, and ${\theta} = \sup_{q\sim Q}{\theta}_q$ with ${\theta}_q = 2i\kappa_1$. Define
\begin{equation}
  S(Q, N, Y) := \sup_{{\bf a}=(a_n)_n} \frac{1}{\| {\bf a}_N\|} \sum_{q\sim Q} \sum_{\lambda_j<1/4}^{(q)} Y^{2i\kappa_j} \Big| \sum_{n\sim N} a_n \rho_{j\infty}(n)\Big|,\label{eq:def-SQNY}
\end{equation}
where the supremum is over all non-zero sequences~$a_n$.
The heuristics we explain below will give a genuine proof of Theorems~\ref{thm:DI6} and~\ref{thm:DI7}, but only in the regime where~$Q=N^q$, $Y=N^y$ with~$q, y>0$ fixed, or varying in a bounded set; the implied constant could depend on the size of~$q$ and~$y$, whereas Theorems \ref{thm:DI6} and \ref{thm:DI7} are uniform.

The upper-bound~$2i\kappa_j \leq {\theta}$ implies, for all $1\leq Z\leq Y$,
\begin{equation}
  S(Q, N, Y) \leq (Y/Z)^{\theta} S(Q, N, Z).\label{eq:prop-a}
\end{equation}
By the recursion in Lemma \ref{lem:DI8.1}, we have
\begin{equation}
  S(Q, N, Y) \ll_\eps S(\pi N Y / Q, N, Y) + (YN)^\eps(Q+N+NY/Q).\label{eq:prop-b}
\end{equation}
Finally, by the spectral large sieve in Theorem \ref{thm:largesieve},
\begin{equation}
  S(Q, N, 1) + S(Q, N, Q/N) \ll Q + N^{1+\eps}.\label{eq:prop-c}
\end{equation}
Our aim is to explain the shape of the best bound for~$S(Q, N, Y)$ one can extract from these three hypotheses only.

For~$q, y\geq 0$, we define
$$ E(q, y) := \inf\{\sigma\geq 0 \,:\, S(N^q, N, N^y) \ll N^{\sigma} \}. $$
The bounds in \eqref{eq:prop-a}--\eqref{eq:prop-c} translate into the following for $E(q, z)$:
\begin{align}
  E(q, z) \leq {}& E(q, y) \leq E(q, z) + {\theta}(y-z) & & \text{for } \ 0\leq z\leq y, \label{eq:model-a}\\
  {}& E(q, y) \leq \max(E(1+y-q, y), q, 1, 1+y-q) & & \text{for } \ 0\leq q\leq y+1, \label{eq:model-b} \\
  {}& E(q, y) \leq \max(q, 1) & &\text{for } \  y\leq \max(0, q-1). \label{eq:model-c}
\end{align}

We show that maximal map satisfying \eqref{eq:model-a}--\eqref{eq:model-c} is given by the following map $M(q, y)$,
\[ M(q, y) := \max\Big(q,\, 1+{\theta} y,\, q + {\theta}(1+y-2q)\Big). \]

\begin{proposition}\label{prop:DI-th6-theta}
  The map~$M$ satisfies the bounds~\eqref{eq:model-a}--\eqref{eq:model-c}
  Moreoover, any map~$E:\R_+^2 \to \R_+$ that satisfies the bounds~\eqref{eq:model-a}--\eqref{eq:model-c} also satisfies~$E(q, y)\leq M(q, y)$.
\end{proposition}

Proposition \ref{prop:DI-th6-theta} indicates that the bound
$$ S(Q, N, Y) \ll (QNY)^\eps \big(Q + NY^{\theta} + Q(NY/Q^2)^{\theta}\big) $$
should hold true, which witnesses Theorem \ref{thm:DI6}, as in \eqref{eq:thm6Y}.
Moreover, this is the optimal result possible from the bounds \eqref{eq:prop-a}--\eqref{eq:prop-c}.

When the supremum in the definition~\eqref{eq:def-SQNY} is restricted to the characteristic sequence of an interval $a_n=\1_{[N,N_1]}(n)$, then from \cite[p.277]{DI} with $Y=Q+N$, we have
$$ S(Q, N, Q+N) \ll (QN)^\eps(Q + N). $$
Translating to the model, this gives
\begin{align}\label{eq:superbound-th7}
  {}& E(q, y) \leq y \qquad\qquad \text{for } \ y=\max(q,1).
\end{align}
We show that maximal map satisfying \eqref{eq:model-a}--\eqref{eq:superbound-th7} is given by $M^*$, as follows
\begin{align*}
  M^*(q, y) := \max\Big(q,\, 1,\, 1+{\theta}(y-1),\, q+{\theta}(1+y-2q)\Big).
\end{align*}

\begin{proposition}\label{prop:DI-th7-model}
  The map~$M^*:\R_+^2\to\R_+$ satisfies the conditions~\eqref{eq:model-a}--\eqref{eq:superbound-th7}.
  Moreover, any map~$E:\R_+^2\to\R_+$ that satisfies~\eqref{eq:model-a}--\eqref{eq:superbound-th7} also satisfies $E(q, y) \leq  M^*(q, y)$.
\end{proposition}

Proposition \ref{prop:DI-th7-model} may be interpreted as showing
\begin{align*}
  S(Q, N, Y) &\ll (QNY)^\eps \big(Q + N + N(Y/N)^\theta + Q(NY/Q^2)^\theta\big) \\
  & \asymp (QNY)^\eps (Q+N)\big(1 + (NY/[Q+N]^2)^{\theta}\big),
\end{align*}
when restricting the supremum $S$ in~\eqref{eq:def-SQNY} to characteristic sequences of intervals. 
This witnesses Theorem \ref{thm:DI7}. Moreover, this is the optimal result attained from the bounds \eqref{eq:prop-a}--\eqref{eq:prop-c} and $S(Q, N, Q+N) \ll (QN)^\eps(Q + N)$.

\subsection{Verifying the maps $M$ and $M^*$}

Here we quickly prove the first part of Propositions \ref{prop:DI-th6-theta} and \ref{prop:DI-th7-model}, by verifying that $M$ and $M^*$ satisfy the stated conditions.

To show \eqref{eq:model-a} for $M$: for $z\le y$ by definition we have $M(q,z)\le M(q,y)$ and
\begin{align*}
  M(q,y) \le \max\big(q+{\theta}(y-z),1+{\theta} y, q+{\theta}(1+y-2q)\big) = M(q,z) + {\theta}(y-z).
\end{align*}

To show \eqref{eq:model-b} for $M$: for $q\le 1+y$ we have
\begin{align*}
  &\max\big(M(1+y-q,y),q,1,1+y-q\big)\\ 
  &= \max\big(1+{\theta} y, 1+y-q+{\theta}(1+y-2(1+y-q)),q,1,1+y-q\big)\\
  &= \max\big(1+{\theta} y, q+(1-{\theta})(1+y-2q),q,1+y-q\big)\\
  &\ge \max\big(1+{\theta} y,q+{\theta}(1+y-2q),q,0\big) = M(q,y).
\end{align*}

To show \eqref{eq:model-c} for $M$: If $q<1$ then $M(q,0) = \max\big(q,1, q+{\theta}(1-2q)\big) = 1$, since ${\theta}<1/2$. And if $q\ge 1$ then for $y\le q-1$ we have
\begin{align*}
  M(q,y) 
  \le \max\big(q,1+{\theta}(q-1), q+{\theta}(q-2q)\big)
  = q.
\end{align*}
Thus combining, we conclude $M(q,y)\le \max(q,1)$ for $y\le \max(0,q-1)$.

Hence ~$M(q, y)$ satisfies~\eqref{eq:model-a}--\eqref{eq:model-c}.
We similarly check that~$M^*(q, y)$ satisfies~\eqref{eq:model-a}--\eqref{eq:model-c}.

Finally, to show \eqref{eq:superbound-th7} for $M^*$:  For $0\le q\le 1$,  since $\theta<1/2$ we have 
\begin{align*}
  M^*(q,1) = \max(q,1,1,q+2\theta(1-q)) = \max(1, 2\theta +q(1-2\theta))=1,
\end{align*}
and for $q\ge1$, we have
\begin{align*}
  M^*(q,q) 
  = \max(q,1,1+\theta(q-1),q+\theta(1-q)) 
  = \max(q,1+\theta(q-1)) = q.
\end{align*}
Combining gives $M(q,y)\le y$ for $y=\max(q,1)$, showing \eqref{eq:superbound-th7}.

\subsection{Region where the optimal bound holds}

Define the region
$$ K = \{(q, y) \in \R_+^2 \,:\,  E(q, y) \le \max(q, 1)\}. $$
We shall identify certain geometric structures inside $K$, which will guide our proofs.

First note if~$(q,y)\in K$, then~$(q, y')\in K$ for all~$0\leq y' \leq y$. By~\eqref{eq:model-c} we know that~$K$ contains the line~$\{(1+\lambda, \lambda) \,:\, \lambda\geq 0\}$ and the segment~$\{(\lambda, 0) \,:\, 0\leq \lambda \leq 1\}$. Also \eqref{eq:superbound-th7} implies $K$ contains the line $\{(q, q)\,:\, q\geq 1\}$. 

In the following lemma, we show each point $(q,y)\in K$ induces another point in $K$.

\begin{lemma}\label{lem:K-stab-1}
  Let~$y, q$ be such that~$q\geq 1$ and $y\leq 2q-1$, and assume~$(q, y)\in K$. Then~$(q', y') \in K$, where
  \begin{align*}
    q' ={}& \frac{(1+{\theta}) q - {\theta}(1+y)}{1-{\theta}} \\
    y' ={}& q-1+q'.
  \end{align*}
\end{lemma}
\begin{proof}
  Note $0\le {\theta}<1/2$ and the hypothesis~$y+1\leq 2q$ imply
  \begin{align*}
    q' \ge \frac{(1+{\theta})q-2{\theta} q}{1-{\theta}} = q
  \end{align*}
  which in turn gives $y' \ge 2q-1 \ge y$. So combining with \eqref{eq:model-b} and \eqref{eq:model-a}, we have
  \begin{align*}
    E(q', y') \leq{}& \max(E(1+y'-q', y'), q', 1+y'-q') \\
    ={}& \max(E(q, y'), q',q) = \max(E(q, y'), q') \\
    \leq{}& \max(E(q, y) + {\theta}(y' - y), q') \\
    \leq{}& \max(q + {\theta}(y'-y), q').
  \end{align*}
  Here $E(q, y)\le \max(q,1)=q$ since~$(q, y)\in K$. Finally, by definition of $q'$ we have 
  \begin{align*}
    q'={\theta} q'+(1+{\theta})q-{\theta}(1+y)=q+{\theta}(q-1+q'-y) = q+{\theta}(y'-y).
  \end{align*}
  Thus $E(q', y') \le \max(q',q') = \max(q',1)$ and hence~$(q', y') \in K$.
\end{proof}

\begin{lemma}\label{lem:K-stab-2}
  We have~$\{(1+\lambda, \lambda/{\theta}) \,:\, 0\leq \lambda\leq \frac{{\theta}}{1-{\theta}}\} \subset K$.
\end{lemma}
\begin{proof}
  Let~$\lambda\in [0, \frac{{\theta}}{1-{\theta}}]$. By \eqref{eq:model-b} we have
  \begin{align*}
    E(1+\lambda, \lambda/{\theta}) \leq \max(E(\lambda/{\theta}-\lambda, \lambda/{\theta}), 1+\lambda, \lambda/{\theta} - \lambda).
  \end{align*}
  By assumption on~$\lambda$, we have~$\lambda/{\theta} - \lambda \leq 1$, and so \eqref{eq:model-a} and \eqref{eq:model-c} gives
  \begin{align*}
    E(1+\lambda, \lambda/{\theta}) \leq{}& \max(E(\lambda/{\theta} - \lambda, \lambda/{\theta}), 1+\lambda) \\
    \leq{}& \max(E(\lambda/{\theta} - \lambda, 0)+\lambda, 1+\lambda) \\
    \leq{}& \max(\max(\lambda/{\theta} - \lambda, 1)+\lambda, 1+\lambda) \\
    ={}& 1+\lambda.
  \end{align*}
  This shows that~$(1 + \lambda, \lambda/{\theta}) \in K$.
\end{proof}

\begin{proof}[Proof of Proposition~\ref{prop:DI-th7-model}]
  Define a sequence~$(\alpha_n)_{n\geq 0}$ by $\alpha_0 := 1$ and
  $$\alpha_{n+1} = \frac{2-{\theta} \alpha_n}{1-{\theta}(\alpha_n-1)}. $$
  We claim 
  \begin{align}\label{eq:alphanlim2}
    \lim_{n\to\infty}\alpha_n = 2.
  \end{align}
  To prove this, first observe the sequence $(\alpha_n)$ is the orbit a M\"obius transformation,
  \begin{align*}
    \alpha_n = g^n\cdot 1, \qquad\qquad\text{for}\quad g:=\mqty(-{\theta} & 2 \\ -{\theta} & 1+{\theta}).
  \end{align*}
  Since~${\theta}<1/2$, the matrix~$g=\smqty(-{\theta} & 2 \\ -{\theta} & 1+{\theta})$ has distinct eigenvalues~${\theta}$ and~$1-{\theta}$, with corresponding eigenvectors~$v_1=(2; 1)$ and~$v_2=(1/{\theta}; 1)$: Indeed, we have
  \begin{align*}
    g v_1 &= \mqty(-{\theta} & 2 \\ -{\theta} & 1+{\theta})\mqty(2\\1) = \mqty(2-2{\theta}\\1-{\theta}) = (1-{\theta})v_1,\\
    g v_2 &= \mqty(-{\theta} & 2 \\ -{\theta} & 1+{\theta})\mqty(1/{\theta}\\1) = \mqty(1\\{\theta}) = {\theta}v_2.
  \end{align*}
  Then decomposing $(1;1) = cv_1+(1-c)v_2$ for $c=\frac{1-{\theta}}{1-2{\theta}} > 0$, we see
  \begin{align*}
    g^n \mqty(1\\1) &= cg^n v_1 + (1-c)g^n v_2\\
    &= c(1-{\theta})^n \mqty(2\\1) + (1-c){\theta}^n \mqty(1/{\theta}\\1) = \mqty(2c(1-{\theta})^n+O({\theta}^n)\\c(1-{\theta})^n+O({\theta}^n)).
  \end{align*}
  Hence we deduce $\alpha_n = g^n\cdot 1 = 2 + O(\tfrac{{\theta}}{1-{\theta}})^n$. In particular \eqref{eq:alphanlim2} follows, since ${\theta}<1/2$.

  We claim that~$K$ contains the line~$\{(1+\lambda, 1+\alpha_n\lambda)\,:\, \lambda\geq 0\}$ for each~$n\geq 0$. This is proved by induction. For~$n=0$, this follows from~\eqref{eq:superbound-th7}. Suppose it is proven for some~$n\geq 0$, and let~$\lambda\geq 0$. By assumption the point~$(q, y) = (1+\lambda, 1+\alpha_n \lambda)$ belongs to~$K$. By Lemma~\ref{lem:K-stab-1}, we deduce that~$K$ contains the point~$(q', y')$,
  where
  $$ q' = \frac{(1+{\theta})(1+\lambda) - {\theta}(2+\alpha_n \lambda)}{1-{\theta}}. $$
  We compute
  $$ q' = 1 + \frac{1-{\theta}(\alpha_n-1)}{1-{\theta}}\lambda, \qquad y'=q+q'-1 = 1 + \lambda \frac{2-{\theta} \alpha_n}{1-{\theta}}. $$
  Letting~$\lambda\geq 0$ vary, we see $K$ contains a line passing through~$(1, 1)$ of slope~$\frac{y'-1}{q'-1}=\frac{2-{\theta} \alpha_n}{1-{\theta}(\alpha_n-1)}=\alpha_{n+1}$. This completes the induction.  

  Given~$q\geq 1$, by definition of $K$ we have~$E(q, 1+\alpha_n(q-1)) \leq q$ for each~$n$. Note $E(q,y)$ is continuous in $y$ by \eqref{eq:model-a}. Thus $\alpha_n\to2$ as~$n\to\infty$, by \eqref{eq:alphanlim2}, gives
  \begin{align}\label{eq:limitE}
    {}& E(q, 2q-1) \leq q, & \qquad \text{for }\ q \geq 1.
  \end{align}

  Let now~$q, y\geq 0$ be arbitrary. 
  
  Assume first~$q\leq 1$. If~$y\leq 1$, then by~\eqref{eq:superbound-th7} we have~$E(q, y) \leq E(q, 1) = 1 = M^*(q, y)$. If~$y\geq 1$, then by~\eqref{eq:model-a} we have~$E(q, y) \leq  E(q, 1)+{\theta}(y-1) = 1 + {\theta}(y-1) = M^*(q, y)$. 
  
  Assume next that~$q\geq 1$. If~$y\leq 2q-1$, then by~\eqref{eq:limitE} we have~$E(q, y) \leq E(q, 2q-1) = q = M^*(q, y)$. If~$y\geq 2q-1$, then by~\eqref{eq:model-a} we have~$E(q, y) \leq E(q, 2q-1)+{\theta}(y-2q+1) = q + {\theta}(y-2q+1) = M^*(q, y)$. Thus in all cases we have~$E(q, y) \leq M^*(q, y)$.
\end{proof}

Moving on to the proof of Proposition \ref{prop:DI-th6-theta}, we first develop intuition in the case~${\theta} = 1/2$.

\begin{lemma}\label{lem:DI-th6-selberg}
  For $\theta=1/2$, we have~$(q, 2q-2) \in K$ for all~$q\geq 1$.
\end{lemma}

\begin{proof}
  Denote~$D_n = \{(1+\lambda, 2\lambda)\,:\, n\leq \lambda\leq n+1\}$. We will show by induction that~$D_n\subset K$ for all~$n\geq 0$. By Lemma~\ref{lem:K-stab-2}, we have~$D_0\subset K$. Suppose~$D_n\in K$ for some~$n$, and let~$\lambda\in[n, n+1]$. By Lemma~\ref{lem:K-stab-1}, since~$(1+\lambda, 2\lambda) \in K$, we deduce that~$(q', \lambda+q')\in K$, where for ${\theta}=1/2$,
  $$ q' = \frac{(1+{\theta})(1+\lambda) - {\theta}(1+2\lambda)}{1-{\theta}} = \lambda + 2. $$
  Therefore~$\{(2+\lambda, 2+2\lambda)\,:\, n\leq \lambda \leq n+1\} = D_{n+1} \subset K$, which completes the induction.
\end{proof}

This implies $E(q, 2q-2) \leq q$ for $q\ge1$ and better motivates the choice $y=2q-2$ in \cite{DI}, when $\theta=1/2$. With Lemma~\ref{lem:DI-th6-selberg} in mind, we similarly consider $\theta<1/2$.

\begin{lemma}\label{lem:optim-1}
  For $\theta<1/2$, we have~$(q, (q-1)/{\theta}) \in K$ for~$1\leq q< \frac{1-{\theta}}{1-2{\theta}}$.
\end{lemma}
\begin{proof}
  Consider~$\lambda\in[0, \frac{{\theta}}{1-{\theta}}]$. We have~$(1+\lambda, \lambda/{\theta})\in K$ by Lemma~\ref{lem:K-stab-2}. Thus by Lemma~\ref{lem:K-stab-1}, we deduce~$(q_1, \lambda+q_1)\in K$ as well, where
  $$ q_1 = \frac{(1+{\theta})(1+\lambda) - {\theta}(1+\lambda/{\theta})}{1-{\theta}}.$$
  Letting $\lambda_1 := q_1-1= \frac{{\theta}}{1-{\theta}}(1+\lambda)$, we have $(1+\lambda_1, \lambda_1/{\theta})=(q_1, \lambda+q_1)\in K$, since
  \begin{align*}
    \lambda + q_1 = \lambda+1+\lambda_1 = (\lambda+1)(1+\tfrac{{\theta}}{1-{\theta}}) = \frac{\lambda+1}{1-{\theta}} = \frac{\lambda_1}{{\theta}}.
  \end{align*}
  Let~$\rho := \frac{{\theta}}{1-{\theta}}$ so $\lambda_1=\rho(1+\lambda)$. As~$\lambda$ varies in~$[0, \rho]$, we see~$\lambda_1$ varies in~$[\rho, \rho+\rho^2]$, and hence~$(1+\lambda, \lambda/{\theta}) \in K$ for all~$0\leq \lambda \leq \rho+\rho^2$. An immediate induction shows that in fact~$(1+\lambda, \lambda/{\theta}) \in K$ whenever~$0\leq \lambda \leq \rho + \cdots + \rho^n$ for any~$n\geq 1$. Since~$\sum_{n\geq 1} \rho^n = \frac{{\theta}}{1-2{\theta}}$, we deduce that~$(1+\lambda, \lambda/{\theta})\in K$ for all~$0\leq \lambda < \frac{{\theta}}{1-2{\theta}}$, that is,~$1\le q = 1+\lambda< \frac{1-{\theta}}{1-2{\theta}}$.
\end{proof}

\begin{lemma}\label{lem:optim-2}
  We have~$(q, 2q-1)\in K$ for all~$q\geq \frac{1-{\theta}}{1-2{\theta}}$.
\end{lemma}
\begin{proof}
  Define the sequences
  \begin{align*}
  \lambda_0 &:=0, \qquad \lambda_{n+1} = \frac{{\theta}}{1-{\theta}}(1+\lambda_n), \\
  \alpha_0 &:= 1, \qquad \alpha_{n+1} = \frac{2-{\theta} \alpha_n}{1-{\theta}(\alpha_n-1)}.
  \end{align*}
  Recall from \eqref{eq:alphanlim2} that~$\alpha_n\to 2$ as~$n\to\infty$. Also the sequence~$(\lambda_n)$ is increasing and converges to~$\frac{{\theta}}{1-2{\theta}}$.
  We prove by induction that for all~$n\geq 0$, $K$ contains the line
  \begin{align}\label{eq:Klinemun}
    \{(1 + \lambda_n + \mu, \, \lambda_n/{\theta} + \mu\alpha_n)\,:\, \mu\geq 0\} \ \subset \ K.
  \end{align}
  Indeed, for the base case $n=0$, we have $(1+\mu,\mu)\in K$ for all $\mu\ge0$. Then assuming $(1 + \lambda_n + \mu, \tfrac1{\theta} \lambda_n + \alpha_n \mu)\in K$, by Lemma~\ref{lem:K-stab-1} we have $(q',q'+\lambda_n + \mu)\in K$, where $q'$ is
  \begin{align*}
    q' &= \tfrac{1+{\theta}}{1-{\theta}}(1 + \lambda_n + \mu) - \tfrac{{\theta}}{1-{\theta}}(1+\tfrac1{\theta} \lambda_n + \alpha_n \mu) = \frac{1+{\theta} \lambda_n}{1-{\theta}} + \frac{1+{\theta}-{\theta}\alpha_n}{1-{\theta}}\mu.
  \end{align*}
  Noting $\frac{1+{\theta}\lambda_n}{1-{\theta}} = 1+\tfrac{{\theta}}{1-{\theta}}(1+\lambda_n) = 1+\lambda_{n+1}$, we let $\mu' =\frac{1+{\theta}-{\theta}\alpha_n}{1-{\theta}} \mu$ so that
  \begin{align*}
    K\ni (q',q'+\lambda_n + \mu) &= \bigg(1+\lambda_{n+1} + \frac{1+{\theta}-{\theta}\alpha_n}{1-{\theta}}\mu,\, \frac{1+\lambda_n}{1-{\theta}} + \frac{2-{\theta}\alpha_n}{1-{\theta}}\mu\bigg)\\
    & = \bigg(1+\lambda_{n+1} + \mu',\, \frac{\lambda_{n+1}}{{\theta}} + \alpha_{n+1}\mu'\bigg).
  \end{align*}
   This completes the induction, and hence \eqref{eq:Klinemun} follows.
  
  Given~$q>\frac{1-{\theta}}{1-2{\theta}}$ and~$n\geq 0$, we have~$q-1>\frac{{\theta}}{1-2{\theta}}$, and we let~$\mu := q-1-\lambda_n>0$. Since~$(1+\lambda_n+\mu, \lambda_n/{\theta}+\alpha_n\mu)\in K$, we deduce that
  $$ E(q, y_n) \leq q $$
  for~$y_n = \lambda_n / {\theta} + \alpha_n \mu$. Since $\alpha_n\to 2$ and $\lambda_n\to\frac{{\theta}}{1-2{\theta}}$ we obtain
  \begin{align*}
    y_n = \frac{\lambda_n}{{\theta}} + \alpha_n(q-1-\lambda_n) \ \to \ \tfrac{1}{1-2{\theta}}+2(q-1-\tfrac{{\theta}}{1-2{\theta}}) = \tfrac{1-2{\theta}}{1-2{\theta}} + 2(q-1) = 2q-1.
  \end{align*}
  By~\eqref{eq:model-a}, the map~$y\mapsto E(q, y)$ is continuous. So taking the limit as~$n\to\infty$ we see $E(q, y_n) \leq q$ implies~$E(q, 2q-1) \leq q$. Hence~$(q, 2q-1)\in K$ as claimed.
\end{proof}

\begin{proof}[Proof of Proposition~\ref{prop:DI-th6-theta}]
  Assume first that~$q\leq 1$. Then we have~$E(q, y) \leq {\theta} y + E(q, 0) = 1 + {\theta} y = M(q, y)$. Assume next that~$1\leq q \leq \frac{1-{\theta}}{1-2{\theta}}$. Then we have~$E(q, (q-1)/{\theta}) \leq q$ by Lemma~\ref{lem:optim-1}. If~$y\leq (q-1)/{\theta}$, this implies~$E(q, y) \leq E(q, (q-1)/{\theta}) \leq q$, and if~$y\geq (q-1)/{\theta}$, this implies~$E(q, y) \leq {\theta}(y-(q-1)/{\theta}) + E(q, (q-1)/{\theta}) \leq 1 + {\theta} y$. In both cases we find~$E(q, y) \leq M(q, y)$. Finally, if~$q\geq \frac{1-{\theta}}{1-2{\theta}}$, we have by Lemma~\ref{lem:optim-2} that~$E(q, 2q-1) \leq q$, and by an argument similar to the proof of Proposition~\ref{prop:DI-th7-model}, we obtain~$E(q, y) \leq q + \max(0, {\theta}(y-2q+1)) = M(q, y)$. In all cases, we get~$E(q, y) \leq M(q, y)$ as claimed.
\end{proof}

\bibliographystyle{amsplain}

\end{document}